\documentclass[11pt,a4paper,twoside,bibtotoc,final]{scrartcl}
\usepackage{a4wide}
\usepackage{amsfonts}
\usepackage{amsmath}
\usepackage{amssymb}
\usepackage{amsthm}
\usepackage[numbers]{natbib}
\usepackage{booktabs}
\usepackage{algorithm,algpseudocode}
\usepackage{pgfplots}
\pgfplotsset{compat=newest}
\usepackage{graphicx}
\usepackage{color}
\usepackage{colortbl}
\usepackage{scrpage2}
\usepackage{multirow}
\usepackage{boxedminipage}
\usepackage{enumerate}
\usepackage[caption=false,font=footnotesize]{subfig}
\usepackage{rotating}

\usepackage{hyperref}

\usepackage{todonotes}

\newcommand{\N}{\ensuremath{\mathbb{N}}}

\newcommand{\T}{\ensuremath{\mathbb{T}}}

\newcommand{\Z}{\ensuremath{\mathbb{Z}}}

\newcommand{\R}{\ensuremath{\mathbb{R}}}
\newcommand{\C}{\ensuremath{\mathbb{C}}}

\newcommand{\ii}{\textnormal{i}}
\newcommand{\e}{\textnormal{e}}

\newcommand{\ceil}[1]{\left\lceil#1\right\rceil}
\newcommand{\floor}[1]{\left\lfloor#1\right\rfloor}
\newcommand{\zb}[1]{\ensuremath{\boldsymbol{#1}}}

\renewcommand{\ln}{\mathrm{ln\,}}

\DeclareMathOperator*{\argmin}{arg\,min}

\DeclareMathOperator*{\diag}{diag}
\newcommand{\cond}{\mathrm{cond}}

\newtheorem{theorem}{Theorem}[section]
\newtheorem{lemma}[theorem]{Lemma}
\newtheorem{remark}[theorem]{Remark}
\newtheorem{generalisation}[theorem]{Generalisation}
\newtheorem{definition}[theorem]{Definition}
\newtheorem{example}[theorem]{Example}
\newtheorem{corollary}[theorem]{Corollary}
\newtheorem{proposition}[theorem]{Proposition}

\newenvironment{Theorem}{\goodbreak \begin{theorem}\normalfont \slshape}{\end{theorem}}

\newenvironment{Example}{\goodbreak \begin{example}\normalfont \rmfamily}{\bend\end{example}}

\newenvironment{Corollary}{\goodbreak \begin{corollary}\normalfont \slshape}{\end{corollary}}

\makeatletter
\def\imod#1{\allowbreak\mkern10mu({\operator@font mod}\,\,#1)}
\makeatother

\numberwithin{equation}{section}
\numberwithin{table}{section}
\numberwithin{figure}{section}

\newcommand{\bend}{\hspace*{0ex} \hfill \hbox{\vrule height
    1.5ex\vbox{\hrule width 1.4ex \vskip 1.4ex\hrule  width 1.4ex}\vrule
    height 1.5ex}}

\long\def\symbolfootnote[#1]#2{\begingroup
\def\thefootnote{\fnsymbol{footnote}}\footnote[#1]{#2}\endgroup}
\newcommand{\sspan}{\textnormal{span}}

\newcommand{\OO}[1]{\mathcal{O}\left(#1\right)}

\renewcommand{\mathbf}[1]{\ensuremath{\boldsymbol{#1}}}
\renewcommand{\textbf}[1]{{\ensuremath{\boldsymbol{#1}}}}
\renewcommand{\thefootnote}{\fnsymbol{footnote}}
\pgfplotscreateplotcyclelist{MR1LOF}{
dashed, mark=diamond, mark options={solid, scale=1.5}\\%
dashed, mark=o, mark options={solid, scale=1.5}\\%
densely dotted, mark=triangle, mark options={solid, scale=1.5}\\%
densely dotted, mark=square, mark options={solid, scale=1.25}\\%
dashdotted,mark=x, mark options={solid, scale=1.5}\\%
dashdotted, mark=+, mark options={solid, scale=1.5}\\%
loosely dashed, mark=square, mark options={solid,scale=0.75}\\%
loosely dashdotted, mark=asterisk, mark options={solid,scale=0.5}, opacity=0.5\\%
dasdotdotted, every mark/.append style={solid},mark=star\\%
densely dashdotted,every mark/.append style={solid, fill=gray},mark=diamond*\\%
}

\title{Constructing spatial discretizations for sparse multivariate trigonometric polynomials
that allow for a fast discrete Fourier transform}

\date{\today}
\author{
Lutz K\"ammerer\footnotemark[1]}

\hypersetup{pdfauthor=Lutz K\"ammerer,
          pdftitle={Discretizing sparse multivariate trigonometric polynomials for FFT},
          pdfsubject={},
          pdfcreator = {pdflatex},
          plainpages=false,
          pdfstartview=FitH,       
          pdfview=FitH,            
          pdfpagemode=UseOutlines, 
          bookmarksnumbered=true, 
          bookmarksopen=false,     
          bookmarksopenlevel=0,   
          colorlinks=true,       
          linkcolor=black,
          citecolor=black,
          urlcolor=black}

\newif\ifshowextendedpaperversion
\showextendedpaperversionfalse

\begin{document}

\maketitle

\begin{abstract}
\small
The paper discusses the construction of high dimensional spatial discretizations
for arbitrary multivariate trigonometric polynomials, where the frequency support of the trigonometric polynomial is known.
We suggest a construction based on
the union of several rank\mbox{-}1 lattices as sampling scheme.
We call such schemes multiple rank\mbox{-}1 lattices.
This approach automatically makes available a fast discrete Fourier transform (FFT)
on the data.

The key objective of the construction of spatial discretizations is the
unique reconstruction of the trigonometric polynomial using the sampling values
at the sampling nodes. We develop different construction methods
for multiple rank\mbox{-}1 lattices that allow for this unique reconstruction.
The symbol $M$  denotes the total number of sampling nodes within the multiple rank\mbox{-}1 lattice.
In addition, we assume that the multivariate trigonometric polynomial
is a linear combination of $T$ trigonometric monomials.
The ratio of the number $M$ of sampling points that are sufficient for the unique reconstruction
to the number $T$ of distinct monomials is called oversampling factor in this context.
The presented construction methods for multiple rank\mbox{-}1 lattices allow for estimates of
this number $M$.
Roughly speaking, the oversampling factor $M/T$ is independent of the spatial dimension and, with high probability,
only logarithmic in $T$, which is much better than the oversampling
factor that is expected for a sampling method that uses one single rank\mbox{-}1 lattice. 

The newly developed approaches for the construction of spatial discretizations are probabilistic methods.
The arithmetic complexity of these algorithms depend only linearly on the spatial dimension and,
with high probability, only linearly on $T$ up to some logarithmic factors.

Furthermore, we analyze the computational complexities of the resulting FFT algorithms,
that exploits the structure of the suggested multiple rank-1 lattice spatial discretizations, in detail and obtain
upper bounds in $\OO{M\log M}$, where the constants depend only linearly on the spatial dimension.
With high probability, we construct spatial discretizations where $M/T\le C\log T$ holds,
which implies that the complexity of the corresponding FFT converts to $\OO{T\log^2 T}$.
\medskip

\noindent {\textit{Keywords and phrases}} : 
sparse multivariate trigonometric polynomials, lattice rule,  multiple rank\mbox{-}1 lattice, fast Fourier transform

\medskip

{\small
\noindent {\textit{2010 AMS Mathematics Subject Classification}} : 
60B15, 
65T50, 
68Q25, 
68W40, 
94A20  
}
\end{abstract}
\footnotetext[1]{
  Chemnitz University of Technology, Faculty of Mathematics, 09107 Chemnitz, Germany\\
  kaemmerer@mathematik.tu-chemnitz.de
}

\medskip

\ifshowextendedpaperversion
\tableofcontents
\newpage
\fi

\section{Introduction}

Since almost 60 years rank\mbox{-}1 lattice rules as Quasi-Monte Carlo type cubature rules were investigated 
in  the field of numerical integration. An overview on the early work on lattice rules can be found in \cite{Nie78}. A lot of
meaningful theoretical facts in the field of numerical integration
could be proved using rank\mbox{-}1 lattices as sampling schemes.
The crucial breakthrough for practical applications
was the development of component--by--component constructions for rank\mbox{-}1 lattices.
Under specific assumptions, this construction method is very practicable and the corresponding cubature rules guarantee 
optimal worst case error rates for the numerical integration of specific multivariate functions, cf., e.g., \cite{CoNu07}.
Already in the late 1950s, N.\,M.~Korobov developed the ideas and a substantial theory
on that topic in, e.g., \cite{Ko59,Ko63}. Unfortunately, N.\,M.~Korobov published in Russian, which
is at least one reason for the lack of awareness of his results.
In 2002, 
the component--by--component idea was re-invented in
\cite{SlRe02}. This paper can be regarded as the initiation of the subsequent
activities of a lot of researchers on component--by--component constructions.

The book of N.\,M.~Korobov, cf.~\cite[Chapter IV]{Ko63}, already contains estimates for
the error of approximation methods based on sampling along single rank-1 lattices.
The main focus was on functions of dominating mixed smoothness, which is still
a highly topical research field. In the 1980's, V.\,N.~Temlyakov \cite{Tem86} improves the 
results of N.\,M.~Korobov using number theoretic argumentations, which
lead to existence results but does not allow for the construction of
suitable rank-1 lattices.
Later, in \cite{KuSlWo06,KaPoVo13,KaPoVo14} similar considerations led to
upper bounds on the worst case error in terms of the number of used sampling values, that were similar to those in \cite{Tem86} and still unsatisfactory in relation to, e.g., sparse grid approximation errors~\cite{DuTeUl16}.
However, the results in \cite{KuSlWo06,KaPoVo13} provides practicable methods for the construction
of suitable rank-1 lattices.
A more recent paper, cf.~\cite{ByKaUlVo16}, discusses the aforementioned non-optimal approximation errors in detail and presents a lower bound on the worst case approximation error for rank\mbox{-}1 lattice sampling that is essentially the same as the already known upper bound.
Accordingly, the optimal worst case errors for approximation cannot even nearly be reached using rank\mbox{-}1 lattice sampling for specific approximation problems, e.g., hyperbolic cross approximations.

Nevertheless rank\mbox{-}1 lattices as sampling schemes provide
stability, available efficient algorithms for computing the discrete Fourier transform \cite{kaemmererdiss}, and
good approximation properties, cf., e.g., \cite{ByKaUlVo16}, in particular for high dimensional
approximation problems. For specific frequency sets $I\subset\Z^d$, the crucial disadvantage of rank\mbox{-}1 lattices
is the large number $M$ of sampling values 
one necessarily needs in order to uniquely reconstruct a trigonometric
polynomial with frequencies supported on the frequency set $I$.
Depending on the structure of the frequency set $I$, the number $M$ is bounded from below by a term $M\ge C|I|^2$, cf. \cite{kaemmererdiss}.

In order to overcome the limitations of the single rank\mbox{-}1 lattice approach, the author presented in \cite{Kae16} the idea to use multiple rank\mbox{-}1 lattices as spatial discretizations for multivariate trigonometric polynomials.
A first simple construction method for multiple rank\mbox{-}1 lattices
is published therein, where the construction is mainly based on
the random determination of generating vectors of single rank\mbox{-}1 lattices.
Various numerical experiments illustrate the advantages of these
spatial discretizations, such as
\begin{itemize}
\item the existence of a fast discrete Fourier transform (FFT),
\item low oversampling factors of sampling sets that allow for a unique reconstruction of multivariate trigonometric polynomials from their sampling values,
\item small condition numbers of the corresponding Fourier matrices,
\item excellent approximation properties in specific problems.
\end{itemize}

Supposing that the above advantages actually hold in general --- or at least with high probability ---
we will overcome limitations of sparse grid sampling as well as limitations of single rank\mbox{-}1 lattice sampling,
which are unbounded condition numbers, cf. \cite{KaKu10}, and necessarily huge oversampling factors, cf. \cite{kaemmererdiss}, respectively.

While the fast discrete Fourier transform algorithms are already presented in \cite{Kae16},
all other above listed advantages are not proven yet.

In this paper, we incorporate ideas from \cite{ArGiRo16}
into the algorithm that is developed in \cite[Alg. 5]{Kae16} such that
the construction of multiple rank\mbox{-}1 lattices will allow for estimates of the sizes and the number of single rank\mbox{-}1 lattices 
that are joined to a spatial discretization.

We reflect some basics on multivariate trigonometric polynomials and (multiple) rank\mbox{-}1 lattices from \cite{kaemmererdiss,Kae16}.
At first, we define the torus $\T\simeq[0,1)$ and the multivariate trigonometric polynomial
\begin{align*}
p\colon\T^d\rightarrow \C,\qquad
p(\zb x):=\sum_{\zb k\in I}\hat{p}_{\zb k}\e^{2\pi\ii\zb k\cdot\zb x},
\end{align*}
where $I\subset\Z^d$ is called frequency set, its cardinality $|I|<\infty$ is finite,
and $\hat{p}_{\zb k}$ is named Fourier coefficient to the frequency $\zb k$ of the multivariate trigonometric polynomial $p$.
For an arbitrary set of sampling nodes $\mathcal{X}\subset\T^d$, $|\mathcal{X}|<\infty$, the Fourier matrix
\begin{align}
\zb A(\mathcal{X},I):=\left( \e^{2\pi\ii\zb k\cdot\zb x}\right)_{\zb x\in\mathcal{X},\zb k\in I}\label{eq:Fourier_matrix_X}
\end{align}
allows for the computation of the evaluation of $p$ at all nodes of $\mathcal{X}$ using the matrix vector product
\begin{align*}
\zb p=\zb A(\mathcal{X},I)\zb{\hat{p}},
\end{align*}
where $\zb p=(p(\zb x))_{\zb x\in\mathcal{X}}$ and $\zb{\hat{p}}=(\hat{p}_{\zb k})_{\zb k\in I}$
are vectors that contain the sampling values of $p$ at all sampling nodes from $\mathcal{X}$ and the Fourier coefficients of $p$, respectively.
At this point, we need to assume that the elements of $I$ and $\mathcal{X}$ are in a fixed order to interrelate
the matrix with the vectors. Knowing the Fourier coefficients $\hat{p}_{\zb k}$, the matrix vector product will compute the sampling values
of $p$ at all nodes of $\mathcal{X}$. On the other hand, we can also uniquely reconstruct the Fourier coefficients from known
sampling values, iff the matrix $\zb A(\mathcal{X},I)$ is of full column rank.
In that case, the matrix $\zb A(\mathcal{X},I)$ and, synonymously, the sampling scheme ¸$\mathcal{X}$
allow for the unique reconstruction of all trigonometric polynomials $p\in\Pi_I:=\sspan\{\e^{2\pi\ii\zb k\cdot\circ}\colon\zb k\in I\}$.
Usually, one uses
a normal equation $\zb{\hat{p}}=\left(\zb A(\mathcal{X},I)^*\zb A(\mathcal{X},I)\right)^{-1}\zb A(\mathcal{X},I)^*\zb p$ in order to compute all Fourier coefficients $\hat{p}_{\zb k}$, $\zb k\in I$, of the trigonometric polynomial $p$, where
$\zb A^*$  is the adjoint matrix of $\zb A$.

The central building block of our construction of spatial discretizations are so-called rank\mbox{-}1 lattices
\begin{align*}
\Lambda(\zb z,M)=\left\{\frac{j}{M}\zb z\bmod \zb 1\colon j=0,\ldots,M-1\right\}\subset\T^d,
\end{align*}
where $\zb z\in\Z^d$ and $M\in\N$ are called generating vector and lattice size of $\Lambda(\zb z,M)$, respectively.
The modulo $\zb 1$ operation denotes a component-wise operation $x\bmod 1=x-\floor{x}\in\T$
and implies that $\Lambda(\zb z,M)\subset[0,1)^d$ holds.
A multiple rank\mbox{-}1 lattice
\begin{align*}
\Lambda(\zb z_1,M_1,\ldots,\zb z_s,M_s):=\bigcup_{\ell=1}^s\Lambda(\zb z_{\ell},M_{\ell})
\end{align*}
is a union of a set of rank\mbox{-}1 lattices.
Obviously, we obtain 
\begin{align}
|\Lambda(\zb z_1,M_1,\ldots,\zb z_s,M_s)|\le1-s+\sum_{\ell=1}^{s}M_{\ell},\label{eq:Fourier_matrix_mr1l}
\end{align}
where at least for the case that all $M_{\ell}$ are pairwise coprime and $\zb z_\ell\not\equiv\zb 0\imod{M_{\ell}}$, $\ell=1,\ldots,s$, the equality holds, cf. \cite[Cor. 2.3]{Kae16}.
Since we use multiple rank\mbox{-}1 lattices as spatial discretization of multivariate trigonometric polynomials, we consider
the corresponding Fourier matrix with  $\mathcal{X}=\Lambda(\zb z_1,M_1,\ldots,\zb z_s,M_s)$
\begin{align}
\zb A&:=\zb A(\Lambda(\zb z_1,M_1,\ldots,\zb z_s,M_s),I):=
\left(\begin{array}{c}
\left(\e^{2\pi\ii \frac{j}{M_1}\zb k\cdot\zb z_1}\right)_{j=0,\ldots,M_1-1,\,\zb k\in I}\\
\left(\e^{2\pi\ii \frac{j}{M_2}\zb k\cdot\zb z_2}\right)_{j=1,\ldots,M_2-1,\,\zb k\in I}\\
\vdots\\
\left(\e^{2\pi\ii \frac{j}{M_s}\zb k\cdot\zb z_s}\right)_{j=1,\ldots,M_s-1,\,\zb k\in I}
\end{array}\right). \label{eqn:Fourier_matrix}
\end{align}
We stress on the fact, that this construction does not comply with the general construction in
\eqref{eq:Fourier_matrix_X} since we allow for identical rows within
$\zb A(\Lambda(\zb z_1,M_1,\ldots,\zb z_s,M_s),I)$ in the case that the equality in
\eqref{eq:Fourier_matrix_mr1l} does not hold.
The crucial advantage of this construction of the Fourier matrix and the corresponding vector $\zb p$ as well is the simple
application of $s$ rank\mbox{-}1 lattice FFTs,
where each of them is the fast computation of a matrix vector product $\zb A(\Lambda(\zb z_\ell,M_\ell),I)\zb{\hat{p}}$, $\ell\in\{1,,\ldots,s\}$, instead of using the matrix $\zb A(\Lambda(\zb z_1,M_1,\ldots,\zb z_s,M_s),I)$
for computing the sampling values of $p$ at all nodes within $\Lambda(\zb z_1,M_1,\ldots,\zb z_s,M_s)$.
The corresponding complexity is in $\OO{M\log M+sd|I|}$, where $M=\sum_{\ell=1}^{s}M_{\ell}$, cf. \cite[Alg. 3]{Kae16}.

Even the solution of the normal equation $\zb{\hat{p}}=\zb A^{\dagger}\zb p$, where $\zb A^{\dagger}:=\left(\zb A^*\zb A\right)^{-1}\zb A^*$ is the pseudo-inverse of $\zb A$,
for given sampling values $\zb p$ can be efficiently computed by means of a conjugate gradient method that uses rank\mbox{-}1 lattice FFTs and
its adjoint FFTs, cf. \cite[Sec. 4]{Kae16}.
Additional requirements on the multiple rank\mbox{-}1 lattice allow for a direct reconstruction method, cf. \cite[Alg. 6]{Kae16},
that has a total complexity in $\OO{M\log M+s(d+\log{|I|})|I|}$.
At this point, we would like to emphasize that all constructions in the present paper, if successful,
allow for the application of this direct fast reconstruction algorithm.
We call a multiple rank\mbox{-}1 lattice that allows for a unique reconstruction of all trigonometric polynomials $p\in\Pi_I$, i.e., the matrix $\zb A(\Lambda(\zb z_1,M_1,\ldots,\zb z_s,M_s),I)$ has full column rank, 
\emph{reconstructing multiple rank\mbox{-}1 lattice} for the frequency set $I$ or, synonymously, for all polynomials $p\in\Pi_I$.

Obviously, the aforementioned complexities are all bounded by terms that are at most linear in $M$ up to a logarithmic factor
plus linear in $|I|$ up to a logarithmic factor, which is clearly much less than the complexity $\OO{M|I|}$ of a matrix vector product using
the Fourier matrix~$\zb A(\Lambda(\zb z_1,M_1,\ldots,\zb z_s,M_s),I)$. Nevertheless, the reconstruction problem will only provide unique solutions if the Fourier matrix $\zb A$ is of full column rank, i.e., 
a unique reconstruction may require a huge number of sampling values, cf. e.g. \cite{kaemmererdiss}, where we proved that $M\sim |I|^2$ is necessary for a single rank\mbox{-}1 lattice as spatial discretization $\mathcal{X}$, which can be interpreted as the currently considered approach with $s=1$, in the worst case.

In this  paper, we present algorithms that  determine reconstructing multiple rank\mbox{-}1 lattices
for given frequency sets $I$. The crucial difference to the already published algorithm, cf. \cite[Alg. 5]{Kae16},
is that the new algorithms allow for estimates on the total number of sampling nodes of the multiple rank\mbox{-}1 lattices that arise.
Under the weak assumption that the frequency set $I\subset\Z^d$, $|I|>3$, is contained in a box of edge length $2\,(|I|-1)$, we prove that there exists at least one reconstructing multiple rank\mbox{-}1 lattice for $I$ with a number of sampling nodes that is bounded
by $C|I|\ln|I|$, where $C<10$ is a fixed constant.
Moreover, the presented algorithms are based on the theoretical considerations and their constructive proofs.
The determined reconstructing multiple rank\mbox{-}1 lattices for frequency sets $I$
contain at most $C_{\delta}|I|\log |I|$ sampling nodes with probability at least $1-\delta$,
where $C_{\delta}\lesssim \log{\delta}$.

The last-mentioned constants $C$ and $C_\delta$  do not depend on the dimension $d$.
If we assume the cardinality $M$ of the reconstructing multiple rank\mbox{-}1 lattice for $I$ fulfills
$M\le C|I|\ln|I|$, we achieve computational complexities in $\OO{|I|(d+\log|I|)\log|I|}$ of the fast algorithms computing
the matrix vector products $\zb  A\zb{\hat{p}}$ and $(\zb A^*\zb A)^{-1}\zb A^*\zb p$, cf. \cite{Kae16}.
At this point, we stress on the dimension $d$, that is only a linear factor, and the 
cardinality of the frequency set $I$, that arises as a linear factor times a polynomial of degree two in $\log |I|$.

\begin{table}
\centering
\scalebox{0.875}{
\begin{tabular}{ccccc}
\toprule
&&\multicolumn{2}{c}{computational complexities}\\
$I$ &
$\mathcal{X}$ &
$\zb A(\mathcal{X},I)\zb p$ &
$\zb A^{\dagger}(\mathcal{X},I)\zb{\hat{p}}$ &
references\\
\midrule
\midrule
full grid & full grid & 
\multicolumn{2}{c}{\multirow{2}{*}{$\OO{N^d\log N}$}}&
\multirow{2}{*}{\cite{loan}}
\\
$|I|=N^d$ & $|\mathcal{X}|=N^d$&\\
\midrule
\midrule
hyperbolic cross & sparse grid & 
\multicolumn{2}{c}{\multirow{2}{*}{$\OO{N\log^d N}$}}&
\multirow{2}{*}{\cite{Halla92}}
\\
$|I|\simeq C_dN\log^{d-1}N$ & $|\mathcal{X}|\simeq C_dN\log^{d-1}N$&\\
\midrule
hyperbolic cross & rank\mbox{-}1 lattice & 
\multicolumn{2}{c}{\multirow{2}{*}{$\OO{N^2\log^{d-1} N}$}}&
\multirow{2}{*}{\cite{Kae2012}}
\\
$|I|\simeq C_dN\log^{d-1}N$ & $|\mathcal{X}|\lesssim c_dN^2\log^{d-2}N$&\\
\midrule
hyperbolic cross & multiple rank\mbox{-}1 lattice & 
\multicolumn{2}{c}{\multirow{2}{*}{$\OO{N\log^{d+1} N}$}}&
\multirow{2}{*}{\cite{Kae16} \& Thm. \ref{thm:prob_bound_T}}
\\
$|I|\simeq C_dN\log^{d-1}N$ & $|\mathcal{X}|\lesssim c_dN\log^{d}N$&\\
\midrule
\midrule
arbitrary& rank\mbox{-}1 lattice & 
\multicolumn{2}{c}{\multirow{2}{*}{$\OO{T^2\log T}$}}&
\multirow{2}{*}{\cite{kaemmererdiss}}
\\
$|I|=T$, $N_I\lesssim T$  & $|\mathcal{X}|< T^2$&\\
\midrule
arbitrary & random & 
\multicolumn{2}{c}{\multirow{2}{*}{$\OO{T^2\log T}$}}&
\multirow{2}{*}{\cite{GrPoRa07}}
\\
$|I|=T$ & $|\mathcal{X}| \lesssim T\log{T}$&\\
\midrule
arbitrary & multiple rank\mbox{-}1 lattice & 
\multicolumn{2}{c}{\multirow{2}{*}{$\OO{T\log^2 T}$}}&
\multirow{2}{*}{\cite{Kae16} \& Thm. \ref{thm:prob_bound_T}}
\\
$|I|=T$, $N_I\lesssim T$  & $|\mathcal{X}| \lesssim T\log{T}$&\\
\bottomrule
\end{tabular}}
\caption{Different types of frequency sets and corresponding spatial discretization schemes in comparison. In addition, the computational complexities of known evaluation and reconstruction algorithms and references.}\label{tab:intro}
\end{table}

In Table \ref{tab:intro} we present different types of frequency sets $I$ and corresponding spatial discretization schemes $\mathcal{X}$ and 
also the computational complexity of known fast algorithms for the computation of the evaluation $\zb p=\zb  A\zb{\hat{p}}$ and the reconstruction $\zb{\hat{p}}=\zb  A^\dagger\zb p$ of the corresponding multivariate trigonometric polynomials in order to classify the result of this paper.
We stress the fact that we have left out terms that depend only on $d$ in the computational complexities since
they are not analyzed in specific cases.

The full grid case is well known as FFT of multidimensional arrays
and only provided as reference in the table.
Considering hyperbolic crosses as frequency sets, there already exist
two approaches for spatial discretizations that allow
for a fast Fourier transform. First, sparse grids as spatial discretizations
and the corresponding hyperbolic cross FFT provide a complexity that is almost
linear in the number of frequencies up to a logarithmic factor.
Unfortunately, the associated discrete Fourier transform is not stable, cf. \cite{KaKu10}.
Second, a single rank\mbox{-}1 lattice discretization for hyperbolic cross trigonometric polynomials
requires at least $N^2$ sampling nodes and we are able to construct a single rank\mbox{-}1 lattice of size $c_dN^2\log^{d-2}N$
that allows for a unique reconstruction of hyperbolic cross trigonometric
polynomials, cf. \cite{kaemmererdiss}. The complexity of the corresponding fast Fourier transform is
linear in the number of sampling nodes within the spatial discretization up to a logarithmic factor.
The bottleneck of this approach is the necessarily huge number of sampling nodes of the spatial discretization.

In this paper, we show that the concept of multiple rank\mbox{-}1 lattices allows for spatial discretizations of
hyperbolic cross trigonometric polynomials of much lower cardinality than single rank\mbox{-}1 lattices,
i.e., we prove oversampling factors that are bounded by $C\log N$. For single rank\mbox{-}1 lattices we observed
oversampling factors of approximately $CN/\log{N}$. 
Regarding only the oversampling of the last mentioned spatial discretizations for hyperbolic cross trigonometric polynomials,
we should prefer sparse grids.
However, a critical look at the numerical tests and, in particular, on the condition numbers of the
Fourier matrices may lead to another assessment, cf. Section~\ref{sec:num_test:hc}.
Apart from this, the general concept of sparse grids requires a specific kind of structure of the frequency sets $I$ called ``downward closed''.
If this structure is seriously violated the sparse grid approach is not successfully applicable.

Last, Table \ref{tab:intro} shows results for trigonometric polynomials with arbitrary frequency sets $I$ of cardinality $T$, where the listed sampling
methods based on structured sampling sets, i.e., the single rank-1 lattice approach as well as the multiple rank-1 lattice approach, require some restriction on the expansion $N_I$,
cf. \eqref{eq:def_expansion}, of the frequency set $I$ under consideration.
A spatial discretization using a single rank\mbox{-}1 lattice structure needs less than $T^2$ sampling nodes and more than
$c\,T^2$, $c>0$, sampling nodes in the worst case.
The corresponding fast Fourier transform is linear in the number of used sampling nodes up to a logarithmic factor.
Here again, the number of necessary sampling nodes is the bottleneck of the single rank\mbox{-}1 lattice approach.
On the contrary, random sampling requires only a few oversampling for stable  Fourier matrices.
Since the corresponding
Fourier matrices suffer from a lack of structure, the discrete Fourier transform
is realized using the matrix vector product and, if necessary, a conjugate gradient method.
Hence, the computation is not efficient.
The present paper shows that a multiple rank\mbox{-}1 lattice discretization
involves oversampling that is similar in terms of $T$ to the oversampling that arises at random sampling.
The fast Fourier transform algorithms in \cite{Kae16} exploit the structure
of multiple rank\mbox{-}1 lattices and yield lower computational complexities compared to a random
sampling approach.

The remainder of this paper is structured as follows.
In Section~\ref{sec:basics}, we collect some general basic facts
about spatial discretizations of multivariate trigonometric polynomials.
Subsequently, we present a simple probabilistic strategy that allows for the construction
of spatial discretizations of multivariate trigonometric polynomials in Section~\ref{sec:main_results}.
In particular, carefully chosen lattice sizes $M_1,\ldots, M_s$ and randomly chosen generating vectors $\zb z_\ell\in[0,M_\ell-1]^d\cap\Z^d$, $\ell=1,\ldots,s$,
allow for an estimate of the success probability of this approach.
In addition, we give upper bounds on the number $s$ of such rank-1 lattices
that need to be joined 
in order to obtain -- with a certain probability -- specific properties of the resulting multiple rank-1 lattice that ensure the required full column rank of the Fourier matrix.
We use these bounds on $s$ in order to estimate the number of sampling nodes within the constructed spatial discretizations.
In Algorithms~\ref{alg:construct_mr1l_uniform} and~\ref{alg:construct_mr1l_uniform_different_primes}, we present
an almost non-intrusive approach, i.e., only three basic facts, the dimension $d$, an upper bound $N$ on the expansion $N_I$, and an
upper bound $T$ on the cardinality $|I|$ of the frequency set $I$, need to be known in order to construct a spatial discretization with high probability.
The knowledge about the concrete locations of the elements in the frequency set $I$ allow for some slight modifications on the already developed algorithms
and will decrease the number of used lattices and, with this, the number of sampling nodes within the spatial discretizations in practice, cf.
Algorithms~\ref{alg:construct_mr1l_I} and~\ref{alg:construct_mr1l_I_distinct_primes}. 
In Section~\ref{sec:heu_upgr}, we discuss some improvement ideas on the construction methods. We develop
Algorithms~\ref{alg:construct_mr1l_3} and~\ref{alg:construct_mr1l_4} that are based on an iterative application of the
ideas from Section~\ref{sec:main_results} which allows for the reduction of the lattice sizes $M_\ell$ and leads to slightly larger theoretical bounds on the number $s$ of joined rank-1 lattices.
Moreover, we present an approach, cf. Algorithm~\ref{alg:construct_mr1l_5},
that successively constructs spatial discretizations for parts of the frequency set $I$ under consideration
which has the advantage that the used lattice sizes $M_\ell$ may be further reduced.
In Section~\ref{sec:num_test}, we present extensive numerical tests that confirm the theoretical results.
Moreover, we observe stability of the spatial discretizations, i.e., the numerically determined condition numbers of the Fourier matrices $\zb A$, cf. \eqref{eqn:Fourier_matrix}, are outstandingly small.

\section{Basics}\label{sec:basics}

We collect some basic observations about sampling schemes,
first for multivariate trigonometric monomials and second for multivariate
trigonometric polynomials with frequencies supported on shifted frequency sets.

\begin{lemma}\label{lem:one_fi}
We consider a trigonometric polynomial consisting of at most one scaled monomial, i.e.,
$p(\zb x)=\hat{p}_{\zb k}\e^{2\pi\ii\zb k\cdot\zb x}$.
The corresponding frequency set $I\subset\Z^d$ contains one element $\zb k\in I$.
Then, the single Fourier coefficient $\hat{p}_{\zb k}$, $\zb k\in I$, can be reconstructed using an
arbitrary sampling set that consists of at least  one sampling node $\zb x_0\in \T^d$.
\end{lemma}
\begin{proof}
Due to the fact that $\e^{2\pi\ii\zb k\cdot\zb x_0}\neq0$, we achieve
$\hat{p}_{\zb k}=\frac{p(\zb x_0)}{\e^{2\pi\ii\zb k\cdot\zb x_0}}$.
\end{proof}
\begin{lemma}\label{lem:shifting_I}
Let the numbers $a_j,b_j\in\Z$ with $a_j\le b_j$, $j=1,\ldots,d$, be given.
We consider the frequency set $I\subset\{a_1,\ldots, b_1\}\times\cdots\times \{a_d,\ldots,b_d\}\subset\Z^d$
and a sampling set $\mathcal{X}$ that allow for the unique reconstruction of all trigonometric polynomials
supported on the frequency set $I_0:=\{\zb k\in \Z^d\colon\zb k=\zb h-\zb a,\,\zb h\in I\}$, $\zb a=(a_1,\ldots,a_d)^\top$.
Then, sampling along the sampling set $\mathcal{X}$ also
allows for the unique reconstruction of all trigonometric polynomials within $\Pi_{I}$.
\end{lemma}
\begin{proof}
We consider the matrix 
\begin{align*}
\zb A(\mathcal{X},I_0)&=\left(\e^{2\pi\ii\zb k\cdot\zb x_j}\right)_{\zb k\in I_0,\,\zb x_j\in\mathcal{X}}=\left(\e^{2\pi\ii(\zb h-\zb a)\cdot\zb x_j}\right)_{\zb h-\zb a\in I_0,\,\zb x_j\in\mathcal{X}}\\
&=\left(\e^{-2\pi\ii\zb a\cdot\zb x_j}\e^{2\pi\ii\zb h\cdot\zb x_j}\right)_{\zb h\in I,\,\zb x_j\in\mathcal{X}}
=\zb D\left(\e^{2\pi\ii\zb h\cdot\zb x_j}\right)_{\zb h\in I,\,\zb x_j\in\mathcal{X}}=\zb D\zb A(\mathcal{X},I),
\end{align*}
where $\zb D:=\diag\left(\e^{-2\pi\ii\zb a\cdot\zb x_j}\right)_{\zb x_j\in\mathcal{X}}$.
The matrix $\zb  D$ is of full rank. Consequently, the full column rank of $\zb A(\mathcal{X},I_0)$
implies a full column rank of $\zb A(\mathcal{X},I)$.
\end{proof}
Similar argumentations yield the general result, that shifting the frequency set does
not affect the reconstruction property of a sampling set.
Moreover, for the specific structure of rank\mbox{-}1 lattices $\Lambda(\zb z,M)$ we can show that even
the modulo operation applied to each component of the vectors within $I$ does not affect
the reconstruction property of the sampling set $\Lambda(\zb z,M)$ provided that
the modulo operation on the frequencies in $I$ does not lead to collisions, i.e.,
the inequalities 
$$
\zb h\bmod M\not\equiv\zb k\bmod M
$$
hold for all  $\zb  h,\zb k\in I$, $\zb h\neq \zb k$.
In this context the modulo operation on vectors $\zb k$ is
given by
$$
\zb k\bmod M:=\left(
\begin{array}{c}
k_1\bmod M\\
\vdots\\
k_d\bmod M
\end{array}
\right),
$$
where
$k\bmod M=\min\{l\in\N_0\colon l\equiv k\imod M\}$
is the unique smallest non-negative integer representative of the
residue class of $k$ modulo $M$.

\begin{lemma}\label{lem:ImodM}
We consider the frequency set $I\subset\Z^d$, $|I|<\infty$,
and we fix a prime number $M$.
In addition, we define the frequency set
\begin{align}
I_{\bmod M}:=\{\zb h_{\zb k}:=\zb k\bmod M\colon \zb k\in I\}\label{eq:def_ImodM}
\end{align}
and choose a generating vector  $\zb z\in[0,M-1]^d\cap\Z^d$ at random.
If the equality $|I_{\bmod M}|=|I|$ holds, then the rank\mbox{-}1 lattice $\Lambda(\zb z,M)$ allows for the 
unique reconstruction of $\hat{p}_{\zb k}$ for $p\in \Pi_I$ iff the rank\mbox{-}1 lattice $\Lambda(\zb z,M)$
allows for the unique reconstruction of $\hat{\tilde{p}}_{\zb h_{\zb k}}$ for $\tilde{p}\in \Pi_{I_{\bmod M}}$.
\end{lemma}
\begin{proof}
Since we assume $|I|=|I_{\bmod M}|$, we find for each $\zb k\in I$ unique vectors $\zb a_{\zb k}\in\Z^d$ and $\zb h_{\zb k}\in [0,M-1]^d\cap\Z^d$ such that
$$
\zb k:=M\zb a_{\zb k}+\zb h_{\zb k},
$$
where $\zb h_{\zb k_1}\neq \zb h_{\zb k_2}$ for all $\zb k_1,\zb k_2\in I$, $\zb k_1\neq \zb k_2$.
The set  $I_{\bmod M}\subset[0,M-1]^d$ is given by $I_{\bmod M}=\{\zb h_{\zb k}\colon \zb  k\in I\}$.
The corresponding Fourier matrices read as
\begin{align*}
\zb A(\Lambda(\zb z,M),I_{\bmod M})&=
\left(
\e^{2\pi\ii\zb h_{\zb k}\cdot\zb z\frac{j}{M}}
\right)_{\zb h_{\zb k}\in I_{\bmod M},\,j=0,\ldots,M-1}\\
&=\left(
\e^{2\pi\ii(\zb h_{\zb k}\cdot\zb z\frac{j}{M}+\frac{Mj}{M}\zb a_{\zb k}\cdot\zb z)}
\right)_{\zb h_{\zb k}\in I_{\bmod M},\,j=0,\ldots,M-1}\\
&=\left(
\e^{2\pi\ii\zb k\cdot\zb z\frac{j}{M}}
\right)_{\zb k\in I,\,j=0,\ldots,M-1}=\zb A(\Lambda(\zb z,M),I),
\end{align*}
which yields the assertion.
\end{proof}
Due to the last result, we collect all prime numbers $M$ such that
$|I_{\bmod M}|=|I|$ holds. The set of these prime numbers is denoted by
\begin{align}
P^I:=\{M\in\N\colon M \text{ prime with }|I_{\bmod M}|=|I|\}\label{eq:collision_free_primes}.
\end{align}
In addition, we define the \emph{expansion of a frequency set}  $I$ by
\begin{align}
N_I:=\max_{j=1,\ldots, d}\{\max_{\zb k\in I}k_j-\min_{\zb l\in I}l_j\}.\label{eq:def_expansion}
\end{align}
In fact, the number $N_I$ is the smallest number $N$ such that we can shift the set $I$
in $[0,N]^d$. Moreover, we observe the next Lemma.

\begin{lemma}\label{lem:M_g_NI}
The equality $|I_{\bmod M}|=|I|$ is satisfied for all $M\in\N$, $M>N_I$.
\end{lemma}
\begin{proof}
We assume $M>N_I$ with $|I_{\bmod M}|<|I|$, which implies
that there is at least a pair $\zb k,\zb l\in I$, $\zb k\neq\zb l$, with 
$\zb h_{\zb k}=\zb k\bmod M=\zb l\bmod M=\zb h_{\zb l}$.
According to
$0\neq \zb k-\zb l$ and $\zb k-\zb l\equiv\zb 0\mod M$,
there exists at least one $j_0\in\{1,\ldots,d\}$ with
$|k_{j_0}-l_{j_0}|=\ell M\ge M>N_I$, $\ell\in\N$,
which contradicts the definition of $N_I$.
\end{proof}
\begin{lemma}\label{lem:PI_NI}
The set of prime numbers $P^I$ contains all prime numbers larger than $N_I$.
\end{lemma}
\begin{proof}
See Lemma \ref{lem:M_g_NI}.
\end{proof}

\section{Generate reconstructing multiple rank\mbox{-}1 lattices}\label{sec:main_results}

We consider a frequency set $I\subset\Z^d$ of finite cardinality $T=|I|$
and fix one element $\zb k\in I$.
The probability that another element $\zb h\in I\setminus\{\zb k\}$ aliases to this specific $\zb k$, 
while sampling along a random rank\mbox{-}1 lattice is  estimated in the next lemma.

\begin{lemma}\label{lem:count_aliasing_probability1}
We fix a frequency $\zb k\in I\subset\Z^d$, $|I|=T<\infty$, and a prime number $M$ such that $|I_{\bmod M}|=|I|$, cf. \eqref{eq:def_ImodM}. In addition, we choose a generating vector $\zb z\in[0,M-1]^d$ at random. Then, with probability of at most
$\frac{T-1}{M}$ the frequency $\zb k$ aliases to at least one other frequency $\zb h\in I\setminus\{\zb k\}$.
\end{lemma}
\begin{proof}
First we take advantage of $|I_{\bmod M}|=|I|$ and determine the probability that one specific $\zb h_{\zb k'}\in I_{\bmod M}\setminus\{\zb h_{\zb k}\}$
aliases to the fixed $\zb h_{\zb k}$, i.e.,
\begin{align}
\left(\zb h_{\zb k}-\zb h_{\zb k'}\right)\cdot \zb z\equiv 0\bmod M\label{eq:equiv_one_fi1}
\end{align}
holds.
Since $\zb k\neq\zb k'$ and $\zb h_{\zb k}\neq\zb h_{\zb k'}$ there exists a dimension index $r$ such that
$h_{\zb k,r}\neq h_{\zb k',r}$.
For any choice of
the $d-1$ components $z_1,\ldots,z_{r-1},z_{r+1},\ldots, z_{d}$ of the generating vector $\zb z$, we determine exactly one $z_{r}$ with
$$
[0,M-1]\ni z_{r}=b\sum_{\substack{u=1\\u\neq r}}^d(h_{\zb k,u}-h_{\zb k',u})z_u\bmod M,
$$
where $b:=\{t\in[1,M-1]\cap\N\colon(h_{\zb k', r}-h_{\zb k, r})t\equiv 1\imod M\}$ is the uniquely defined multiplicative inverse of $h_{\zb k', r}-h_{\zb k, r}$ modulo $M$. Accordingly, we observe \eqref{eq:equiv_one_fi1}.
Consequently, a proportion of $1/M$ choices of vectors $\zb z$ will cause an aliasing of the frequencies
$\zb h_{\zb k}$ and $\zb h_{\zb k'}$ or $\zb k$ and $\zb k'$, cf. Lemma \ref{lem:ImodM}. By the union bound, the probability that a vector $\zb z$, that is uniformly chosen at random from $[0,M-1]^d\cap\Z^d$, yields an aliasing of the frequency $\zb k$ to at least one other frequency within $I$ is bounded from above by $\frac{T-1}{M}$.
\end{proof}
Using a specific number of rank\mbox{-}1 lattices as sampling scheme provides for the unique reconstruction
of the trigonometric polynomial $p\in\Pi_I$ with a certain probability,
which can be beneficially estimated by the aid of Hoeffding's inequality \cite {Hoeff63}.
\begin{theorem}\label{thm:gen_mr1l}
We consider the frequency set $I\subset\Z^d$ of cardinality $T$ and fix an element $\zb k\in I$.
In addition,  we fix $\delta\in(0,1)$ and we determine two numbers
\begin{align}
\lambda&\ge c(T-1),\qquad c>1,\label{eq:def_lambda1}\\
s&=\ceil{\left(\frac{c}{c-1}\right)^2\frac{\ln T-\ln \delta}{2}}\label{def:s}
\intertext{and the set of the $n\in\N$ smallest prime numbers in $P^I$, cf. \eqref{eq:collision_free_primes}, larger than $\lambda$}
P^{I}_{\lambda,n}&:=\left\{p_j\in P^I\colon p_j=\begin{cases}
\min\{p \in P^I\colon p>\lambda\}&\colon j=1\\
\min\{p \in P^I\colon p>p_{j-1}\}&\colon j=2,\ldots,n.
\end{cases}
 \right\}\label{eq:def_PIlambdan}
\end{align}
We choose $s$ numbers $M_\ell$, $\ell=1,\ldots,s$, randomly from $P^{I}_{\lambda, n}$.
For each of the numbers $M_\ell$, which are not necessarily distinct, we choose
a generating vector $\zb z_\ell\in [0,M_\ell-1]^d$ at random.
Then, the probability that the frequency $\zb k\in I$ aliases to any other frequency within $I$ for each rank\mbox{-}1 lattice $\Lambda(\zb z_\ell,M_\ell)$ is bounded from above by $\frac{\delta}{T}$.
\end{theorem}

\begin{proof}
For the fixed frequency $\zb k\in I$, we define the random variables
$$
Y_\ell^{\zb k}:=\begin{cases}
0&\colon \zb k \textnormal{ does not alias to another frequency within $I$ using $\Lambda(\zb z_\ell,M_\ell)$},\\
1&\colon \zb k \textnormal{ aliases to at least one other frequency within $I$ using $\Lambda(\zb z_\ell,M_\ell)$}.\\
\end{cases}
$$
The random variables $Y_\ell^{\zb k}$, $\ell=1,\ldots, s$, are independent and identically distributed with a specific mean $\mu\le\frac{T-1}{\min_{p\in P^I_{\lambda,n}}p}<\frac{T-1}{\lambda}\le\frac{1}{c}$, due to Lemma \ref{lem:count_aliasing_probability1}.
Hoeffding's inequality allows for the estimate
\begin{align}
\mathcal{P}\left\{\sum_{\ell=1}^sY_\ell^{\zb k}=s\right\} &=\mathcal{P}\left\{s^{-1}\sum_{\ell=1}^sY_\ell^{\zb k}-\mu=1-\mu\right\}\le\mathcal{P}  \left\{s^{-1}\sum_{\ell=1}^sY_\ell^{\zb k}-\mu\ge 1-\varepsilon-\mu\right\}\\
&\le \e^{-2s\left(1-\varepsilon-\mu\right)^2}= \e^{-2s\left(\frac{c-1}{c}\right)^2}
\le\e^{\ln \delta-\ln T}=\frac{\delta}{T}\label{eq:proof_of_thm:gen_mr1l}
\end{align}
for the specific choice $\varepsilon=\frac{1}{c}-\mu>0$.
\end{proof}

If we choose the rank\mbox{-}1 lattice sizes $M_1,\ldots,M_s$ relatively prime, the resulting multiple rank\mbox{-}1 lattice
$\Lambda(\zb z_1,M_1,\ldots,\zb z_s,M_s)$  is actually a structured subsampling of the rank\mbox{-}1 lattice 
$\Lambda(\zb z',M')$, where  the generating vector and the lattice size are given by 
$\zb z'=\sum_{r=1}^s(\prod_{\substack{l=1\\l\neq r}}^{s}M_l)\zb z_r$ and $M'=\prod_{r=1}^sM_r$, respectively, cf. \cite[Cor. 2.2]{Kae16}. Since specific properties of rank\mbox{-}1 lattices are extensively investigated, our studies on multiple rank\mbox{-}1 lattices as subsampling schemes of single rank\mbox{-}1 lattices may profit from this.
On this account, we modify the results from Theorem \ref{thm:gen_mr1l} to these requirements.

\begin{corollary}\label{cor:distinct_mr1l}
Doing the same as in Theorem \ref{thm:gen_mr1l} but without replacement of the $M_\ell$, i.e., the used $M_\ell$ are pairwise distinct which requires $n\ge s$, we get the analogous result.
\end{corollary}

\begin{proof}
Due to \cite[Theorem 4]{Hoeff63} the bounds on probabilities for random samples without replacement coincide with those of similar random variables with replacement.
\end{proof}

Another simple conclusion allows for the estimate of the probability
that we can uniquely determine all Fourier coefficients $\hat{p}_{\zb k}$, $\zb k\in I$, from the sampling values at the nodes of the multiple rank\mbox{-}1 lattice constructed by the approach of Theorem \ref{thm:gen_mr1l}.

\begin{Theorem}\label{thm:prob_bound_T}
Choosing $s$, $\lambda$ and $\Lambda(\zb z_\ell,M_\ell)$, $\ell=1,\ldots,s$, as stated in Theorem \ref{thm:gen_mr1l} or Corollary \ref{cor:distinct_mr1l}, the probability that  the Fourier matrix $\zb A\left(\bigcup_{\ell=1}^s\Lambda(\zb z_\ell,M_\ell),I\right)$ has full column rank is bounded from below by $1-\delta$.
\end{Theorem}
\begin{proof}
We refer to the proof of Theorem \ref{thm:gen_mr1l} and estimate
$$
1-\mathcal{P}\left(\bigcap_{\zb k\in I}\left\{\sum_{\ell=1}^sY_\ell^{\zb k}<s\right\}\right)=\mathcal{P}\left(\bigcup_{\zb k\in I}\left\{\sum_{\ell=1}^sY_\ell^{\zb k}=s\right\}\right)\le
\sum_{\zb k\in I}\mathcal{P}\left(\sum_{\ell=1}^sY_\ell^{\zb k}=s\right)\le  T\frac{\delta}{T}.
$$
using the union bound. Let us assume that
$\bigcap_{\zb k\in I}\left\{\sum_{\ell=1}^sY_\ell^{\zb k}<s\right\}$ occurs.
Then, the proof of \cite[Thm 4.4]{Kae16}
yields the full column rank of the Fourier matrix $\zb A\left(\bigcup_{\ell=1}^s\Lambda(\zb z_\ell,M_\ell),I\right)$.
\end{proof}

\begin{algorithm}[tb]
\caption{Determining reconstructing multiple rank\mbox{-}1 lattices (Theorems \ref{thm:gen_mr1l} and \ref{thm:prob_bound_T})}\label{alg:construct_mr1l_uniform}
  \begin{tabular}{p{1.4cm}p{4.05cm}p{8.8cm}}
    Input: 	& $T\in \N$ 		& upper bound on the cardinality of a frequency set $I$\\
    		& $d\in \N$			& dimension of the frequency set $I$\\
    		& $N\in \N$			& upper bound on the expansion of the frequency set $I$\\
    		& $\delta\in (0,1)$	& upper bound on failure probability\\
    		& $c\in\R$, $c>1$	& minimal oversampling factor
  \end{tabular}
		
  \begin{algorithmic}[1]
	\State $c=\max\left\{c,\frac{N}{T-1}\right\}$\label{alg:construct_mr1l_uniform:line_c_readjust}
  	\State $\lambda=c(T-1)$
  	\State $s=\ceil{\left(\frac{c}{c-1}\right)^2\frac{\ln T-\ln \delta}{2}}$
  	\State $M=\argmin_{p \textnormal{ prime}}\{p>\lambda\}$
  	\For{$\ell=1 \textnormal{ to }s$}
		\State choose $\zb z_\ell$ from $[0,M-1]^d\cap\Z^d$ uniformly at random
  	\EndFor
  \end{algorithmic}
  \begin{tabular}{p{1.4cm}p{4.05cm}p{8.84cm}}
    Output: & $M$ 									& lattice size of rank\mbox{-}1 lattices and\\
		    & $\zb z_1,\ldots,\zb z_s$ 				& generating vectors of rank\mbox{-}1 lattices such that\\
		    & $\Lambda(\zb z_1,M,\ldots,\zb z_s,M)$ & is a reconstructing multiple rank\mbox{-}1 lattice\\
	    	&										& for $I$ with probability at least $1-\delta$
\\    \cmidrule{1-3}
    \multicolumn{2}{l}{Complexity: $\OO{\lambda\log\log\lambda+ds}$}&{\small for fixed $c>1$: $\lambda\sim\max\{T,N\}$ and $s\sim\log T-\log\delta$}  
    \end{tabular}
\end{algorithm}

Due to Theorem \ref{thm:prob_bound_T},
Theorem \ref{thm:gen_mr1l} describes an approach that determines reconstructing sampling sets for
trigonometric polynomials supported on given frequency sets $I$ with probability at least $1-\delta_s$,
where 
\begin{align}
\delta_s=T\e^{-2\left(\frac{c-1}{c}\right)^2s}\label{eq:def:delta_s}
\end{align}
is an upper bound on the probability that the approach fails.
A slight simplification of the strategy is given as Algorithm \ref{alg:construct_mr1l_uniform}, where
we avoid to construct the set $P^{I}_{\lambda,1}$. To this end, we readjust the oversampling factor $c$
in line \ref{alg:construct_mr1l_uniform:line_c_readjust} such that the set $P^{I}_{\lambda,1}$ is forced to consist of the smallest prime larger than $\lambda$, which is contained in the interval $(\lambda,2\lambda]$ due to Bertrand's postulate and can be found using the sieve of Eratosthenes in $\OO{\lambda\log\log\lambda}$. We solely use this prime number as lattice size for each of the single rank\mbox{-}1 lattice parts of the output of Algorithm \ref{alg:construct_mr1l_uniform}, i.e., we fix $n=1$.
Furthermore, Algorithm \ref{alg:construct_mr1l_uniform_different_primes}
uses the same strategy as Algorithm \ref{alg:construct_mr1l_uniform} with
the difference that the output of Algorithm \ref{alg:construct_mr1l_uniform_different_primes}
is a multiple rank\mbox{-}1 lattice that consists of single rank\mbox{-}1 lattices with pairwise distinct
lattice sizes. At this point, we would like to discuss the choice of the rank\mbox{-}1 lattice
sizes $M_j$ in Algorithm \ref{alg:construct_mr1l_uniform_different_primes} line \ref{alg:construct_mr1l_uniform_different_primes:line_Mj}. 
The theoretical considerations in Corollary \ref{cor:distinct_mr1l} suggests
to choose $s$ distinct prime numbers $M_j$ from $P^I_{\lambda,n}$ at random without
replacement.
Due to the readjustment of $c$, cf. line \ref{alg:construct_mr1l_uniform_different_primes:line_readjust_c},
the $s$ smallest prime numbers larger than $\lambda$ are contained in $P^{I}_{\lambda,s}$, i.e., we choose $n=s$.
Since we have to choose each element of 
$P^{I}_{\lambda,s}$ exactly once and the choices of the generating vectors do not depend on each other,
the arrangement of the lattice sizes does not matter. Consequently, we can disregard the
randomness in the choice of the lattice sizes $M_1,\ldots,M_s$.
A simple example will demonstrate the advantages of the readjustment of $c$ in line \ref{alg:construct_mr1l_uniform:line_c_readjust} of Algorithm \ref{alg:construct_mr1l_uniform}.
In addition, the example treats the readjustment of $c$ in Algorithm \ref{alg:construct_mr1l_uniform_different_primes} as well.
\begin{Example}
The readjustment of $c$ in Algorithm \ref{alg:construct_mr1l_uniform}, cf. line \ref{alg:construct_mr1l_uniform:line_c_readjust},
allows for lower numbers of sampling nodes within the output of Algorithm \ref{alg:construct_mr1l_uniform}.
For instance, we assume $T=|I|=11$, $N_I=80$, $c=2$, and $\delta=\frac{11}{\e^6}\approx 0.0273$. Line \ref{alg:construct_mr1l_uniform:line_c_readjust} will readjust $c=8$ in that case.

We consider the above discussed approaches without using the readjustment compared to the indicated Algorithms \ref{alg:construct_mr1l_uniform} and \ref{alg:construct_mr1l_uniform_different_primes} in the following table. 
The number
$M=1-s+\sum_{r=1}^sM_r$ within the two lines right at the bottom is the number of sampling values that we achieve using the approach of distinct
prime lattice sizes. 
The other rows show $M$, which is in fact an upper bound on the number of the sampling values within $|\Lambda(\zb z_1,M_1,\ldots,\zb z_s,M_s)|$
since we may observe $|\Lambda(\zb z_j,M_j)\cap\Lambda(\zb z_r,M_r)|>1$ for $1\le j<r\le s$.
\begin{center}
\begin{tabular}{lrrrr|rrr}
\toprule
&$c$&$T$&$N_I$&$\delta$&$s$&$\lambda$&upper bound on $M$\\
\midrule
fixed $c$, $M_r\in P_{\lambda,1}^I$&$2$&$11$&$80$&$11\e^{-6}$&$12$&$80$&$985$\\
Algorithm \ref{alg:construct_mr1l_uniform}&$8$&$11$&$80$&$11\e^{-6}$&$4$&$80$&$329$\\
\midrule\midrule
fixed $c$, distinct $M_r\in P_{\lambda,s}^I$&$2$&$11$&$80$&$11\e^{-6}$&$12$&$80$&$1\,325$\\
Algorithm \ref{alg:construct_mr1l_uniform_different_primes}&$8$&$11$&$80$&$11\e^{-6}$&$4$&$80$&$367$\\
\bottomrule
\end{tabular}
\end{center}
\end{Example}

We would like to stress the highly interesting fact, that Algorithm \ref{alg:construct_mr1l_uniform}
as well as Algorithm \ref{alg:construct_mr1l_uniform_different_primes}
does not require the input of the frequency set $I$.
The only input parameters we need
are in some sense the key data of the frequency set $I$, i.e., the cardinality $T$ and the expansion $N$ 
of the frequency set $I$.
This observation leads to the following point of view.
Applying Algorithm \ref{alg:construct_mr1l_uniform} or \ref{alg:construct_mr1l_uniform_different_primes} with fixed $T$, $N$, $\delta$, and $c$, leads to
a multiple rank\mbox{-}1 lattice.
Then, with probability at least $1-\delta$ this multiple rank\mbox{-}1 lattice
is a reconstructing multiple rank\mbox{-}1 lattice for a frequency set $I$
with cardinality $|I|\le T$ and upper bound $N$ on the expansion $N_I$.
The crucial difference to the following considerations is that
the frequency set $I$ itself needs not to be known.

\begin{sloppypar}
Nevertheless, the multiple rank\mbox{-}1 lattices that are built using Algorithms \ref{alg:construct_mr1l_uniform} and \ref{alg:construct_mr1l_uniform_different_primes} do not guarantee the reconstruction property, i.e., the  full rank of the Fourier matrix $\zb A$, cf.~\eqref{eqn:Fourier_matrix}, for specific, given frequency sets $I$.
In general, this reconstruction property can not be checked easily. 
For instance, the computation of the echelon form of the matrix $\zb A(\Lambda(\zb z_1,M_1,\ldots,\zb z_s,M_s),I)$
or (lower bounds on) the smallest singular value of $\zb A(\Lambda(\zb z_1,M_1,\ldots,\zb z_s,M_s),I)$
(may) prove the reconstruction property. Corresponding test methods have a complexity in $\Omega(|I|^2)$, which will
cause huge computational costs if the frequency set $I$ has a high cardinality.
Accordingly, the direct validation of the reconstruction property is not convenient.
For that reason, we suggest to check the conditions on the multiple rank\mbox{-}1 lattice
$\zb A(\Lambda(\zb z_1,M_1,\ldots,\zb z_s,M_s),I)$ that
leads to the statements in Theorems \ref{thm:gen_mr1l} and \ref{thm:prob_bound_T}, i.e., we consider
the sets
$$
I_r:=\{\zb k\in I\colon \zb k\cdot\zb z_r\not\equiv\zb h\cdot\zb z_r\bmod{M_r}\text{ for all }\zb h\in I\setminus\{\zb k\} \},\quad r=1,\ldots,s.
$$
Our theoretical considerations in the proofs of Theorem \ref{thm:prob_bound_T}
estimated the probability that each $\zb k\in I$ is contained in at least one $I_r$, $r=1,\ldots,s$.
In other words, we estimated the probability that the equation 
\begin{align}
\bigcup_{r=1}^sI_r=I\label{def:union_Ir}
\end{align}
holds.
The complexity of the computation of each $I_r$ is in $\OO{(d+\log |I|)|I|}$. The union
of all $I_r$, $r=1,\ldots,s$ is also computable in $\OO{ds|I|\log (s|I|)}$, i.e., bounded by $C_{\delta}\,d\,|I|\log^2|I|$, where $C_\delta$ depends on $\delta$ and does not depend on the cardinality $|I|$ or the dimension $d$.
Accordingly, the computational costs to check \eqref{def:union_Ir} is in $\OO{d\,|I|\log^2|I|}$.
Certainly, this strategy will not detect each reconstructing multiple rank\mbox{-}1 lattice, cf. Example \ref{ex:mr1l_with_aliasing}, since equation \eqref{def:union_Ir} is sufficient but not necessary for the reconstruction property.
\end{sloppypar}

\begin{algorithm}[tb]
\caption{Determining reconstructing multiple rank\mbox{-}1 lattices with pairwise distinct lattice sizes (Corollary \ref{cor:distinct_mr1l} and Theorem \ref{thm:prob_bound_T})}\label{alg:construct_mr1l_uniform_different_primes}
  \begin{tabular}{p{1.4cm}p{4.05cm}p{8.8cm}}
    Input: 	& $T\in \N$ 		& upper bound on the cardinality of a frequency set $I$\\
    		& $d\in \N$			& dimension of the frequency set $I$\\
    		& $N\in \N$			& upper bound on the expansion of the frequency set $I$\\
    		& $\delta\in (0,1)$	& upper bound on failure probability\\
    		& $c\in\R$, $c>1$	& minimal oversampling factor
  \end{tabular}
		
  \begin{algorithmic}[1]
	\State $c=\max\left\{c,\frac{N}{T-1}\right\}$\label{alg:construct_mr1l_uniform_different_primes:line_readjust_c}
  	\State $\lambda=c(T-1)$
  	\State $s=\ceil{\left(\frac{c}{c-1}\right)^2\frac{\ln T-\ln \delta}{2}}$
  	\For{$\ell=1 \textnormal{ to }s$}
	  	\State $M_{\ell}=\argmin_{p \textnormal{ prime}}\{p>\max\{\lambda,M_1,\ldots,M_{\ell-1}\}\}$\label{alg:construct_mr1l_uniform_different_primes:line_Mj}
		\State choose $\zb z_{\ell}$ from $[0,M_{\ell}-1]^d\cap\Z^d$ uniformly at random
  	\EndFor
  \end{algorithmic}
  \begin{tabular}{p{1.4cm}p{4.05cm}p{8.84cm}}
    Output: & $M_1,\ldots,M_s$ 							& lattice sizes of rank\mbox{-}1 lattices and\\
		    & $\zb z_1,\ldots,\zb z_s$ 					& generating vectors of rank\mbox{-}1 lattices such that\\
		    & $\Lambda(\zb z_1,M_1,\ldots,\zb z_s,M_s)$ & is a reconstructing multiple rank\mbox{-}1 lattice\\
	    	&											& for $I$ with probability at least $1-\delta$
\\    \cmidrule{1-3}
    \multicolumn{2}{l}{Complexity: $\OO{\lambda\log\log\lambda+ds}$}&{\small for fixed $c$ and $\delta$: $\lambda\sim\max\{T,N\}$ and $s\sim\log T$}  
    \end{tabular}
\end{algorithm}

The last considerations give us a completely new perspective on the
construction of multiple rank\mbox{-}1 lattices. In general,
each $\zb k\in I$ may be contained in more than one $I_r$, $r\in\{1,\ldots,s\}$. 

On the one hand, it may happen that the equation
$\bigcup_{r=1}^{\tilde{s}}I_r=I$ with $\tilde{s}<s$ holds.
In that case, even the multiple rank\mbox{-}1 lattice $\Lambda(\zb z_1,M_1,\ldots,\zb z_{\tilde s},M_{\tilde s})$ is a reconstructing rank\mbox{-}1 lattice for the frequency set $I$, i.e., the last $s-\tilde{s}$ rank\mbox{-}1 lattices do not provide additional information and we can save $M_{\tilde{s}+1}+\ldots+M_s-s+\tilde{s}$ sampling points without loosing information about the trigonometric polynomial $p\in\Pi_I$.

On the other hand, it may happen that the inequality
$\bigcup_{r=1}^{s}I_r<I$ holds since there is a significant probability that
Algorithm \ref{alg:construct_mr1l_uniform} (or Algorithm \ref{alg:construct_mr1l_uniform_different_primes}, respectively) fails. Choosing a few additional rank\mbox{-}1 lattices may built a reconstructing multiple rank\mbox{-}1 lattice for $I$.

From this perspective, one can choose successively lattice sizes $M_\ell$ from $P^{I}_{\lambda,n}$ and generating vectors $\zb z_{\ell}$ from $[0,M_\ell-1]^d$, $\ell=1,\ldots$. The bound $\delta_s$, cf. \eqref{eq:def:delta_s}, decreases exponentially in $s$, which
implies that we can stop our strategy at a large enough $\tilde{s}$. In fact, a simple stopping criterion is easy to determine.
The set
\begin{align}
\bigcup_{r=1}^{\tilde{s}}\left\{\zb k\in I\colon \zb k\cdot\zb z_r\not\equiv\zb h\cdot\zb z_r\bmod{M_r}\text{ for all }\zb h\in I\setminus\{\zb k\}\right\}\label{eq:tilde_I}
\end{align}
must be equal to the frequency set $I$.

Moreover, the equality
\begin{align*}
\bigcup_{r=1}^{\ell-1}&\left\{\zb k\in I\colon \zb k\cdot\zb z_r\not\equiv\zb h\cdot\zb z_r\bmod{M_r}\text{ for all }\zb h\in I\setminus\{\zb k\}\right\}\\
&=\bigcup_{r=1}^{\ell}\left\{\zb k\in I\colon \zb k\cdot\zb z_r\not\equiv\zb h\cdot\zb z_r\bmod{M_r}\text{ for all }\zb h\in I\setminus\{\zb k\}\right\}
\end{align*}
yields that the additional sampling values on the sampling nodes that come from $\Lambda(\zb z_{\ell},M_{\ell})$ do not 
contribute to get closer to the goal, which is the equality of the frequency set $I$ and the joined frequency set in \eqref{eq:tilde_I}.
From this point of view, 
we can leave out $\Lambda(\zb z_{\ell},M_{\ell})$ without loosing significant information.

We recapitulate the
strategy based on Theorem \ref{thm:gen_mr1l} and the last-mentioned ideas in Algorithm \ref{alg:construct_mr1l_I} and estimate the number of sampling nodes $M=|\Lambda(\zb z_1,M_1,\ldots,\zb z_{\tilde{s}},M_{\tilde{s}})|$ of the outputs of Algorithm \ref{alg:construct_mr1l_I}.

With probability not less than $1-\delta$, $0<\delta<1$, Algorithm \ref{alg:construct_mr1l_I} determines a reconstructing multiple rank\mbox{-}1 lattice $\bigcup_{\ell=1}^{\tilde{s}}\Lambda(\zb z_\ell,M_\ell)$, $\tilde{s}\le s$ from \eqref{def:s} , for the frequency set $I$, which consists of at most 
\begin{align}
M:=|\Lambda(\zb z_1,M_1,\ldots,\zb z_{\tilde{s}},M_{\tilde{s}})|\le 1-\tilde{s}+\sum_{\ell=1}^{\tilde{s}}M_\ell\le s\max_{p\in P^{I,\lambda}_n}p\label{eq:upper_bound_M_general}
\end{align}
sampling nodes.
Since, we estimate $\tilde{s}\le s$ in \eqref{eq:upper_bound_M_general}, the estimate also holds for Algorithm \ref{alg:construct_mr1l_uniform}.
For the same reason, the statement of the following corollary
will also be valid for Algorithm~\ref{alg:construct_mr1l_uniform}.

\begin{Corollary}\label{cor:estimate_M_I_basic}
Let $0<\delta<1$, $1<c\in\R$, and $I\subset\Z^d$ a frequency set
of cardinality $T=|I|\ge 2$, $T<\infty$. We determine $\lambda$ and $s$ as defined in \eqref{eq:def_lambda1} and \eqref{def:s}.
Furthermore, we assume $N_I\le\lambda:=c(T-1)$, cf. \eqref{eq:def_expansion}, and we fix $n=1$.
Due to Bertrand's postulate and Lemma \ref{lem:PI_NI},
Algorithm \ref{alg:construct_mr1l_I}
composes a reconstructing multiple rank\mbox{-}1 lattice $\Lambda(\zb z_1,M_1,\ldots,\zb z_{\tilde{s}},M_{\tilde{s}})$ for the frequency set $I$ that consists of at most
\begin{align}
M\le 1-\tilde{s}+\sum_{\ell=1}^{\tilde{s}}M_\ell\le \ceil{\left(\frac{c}{c-1}\right)^2\frac{\ln T-\ln\delta}{2}}2c(T-1)
<C_{c,\delta}T\,\ln{T}
\label{eq:upper_bound_M_T_general}
\end{align}
sampling nodes
with probability $1-\delta$.
For the sake of completeness, we give a simple upper bound on 
$$C_{c,\delta}\le 2c\left(\left(\frac{c}{c-1}\right)^2\frac{1-\log_2\delta}{2}+\log_2\e\right),$$
that depends logarithmically on the upper bound $\delta$ of the failure probability.
\qed
\end{Corollary}
Not surprisingly, the reconstructing multiple rank\mbox{-}1 lattices that are constructed using Algorithm \ref{alg:construct_mr1l_I}
need some oversampling.
However, the occurring oversampling factors $\frac{M}{T}$
are bounded by a term $C_{c,\delta}\,\ln T$ with probability $1-\delta$ in the case that the expansion of the frequency set $I$ is bounded by $N_I\le c(T-1)$.
Consequently, the oversampling factor scales at most logarithmically in $T$ 
and is independent on the dimension $d$ of the frequency set $I$.

\begin{algorithm}[tb]
\caption{Determining reconstructing multiple rank\mbox{-}1 lattices
using the ideas in the context of \eqref{eq:tilde_I}
}\label{alg:construct_mr1l_I}
  \begin{tabular}{p{1.4cm}p{4.05cm}p{8.8cm}}
    Input: 	& $I\subset\Z^d$ 	& frequency set\\
    		& $c\in\R$, $c>1$	& minimal oversampling factor\\
    		& $n\in\N$			& upper bound on the number of distinct lattice sizes
  \end{tabular}
		
  \begin{algorithmic}[1]
  	\State $\lambda=c(|I|-1)$
	\State $P^{I}_{\lambda,n}=\{p_j\colon  p_j\text{ is the $j$th smallest prime satisfying }\lambda<p_j \text{ and } |I_{\bmod p_j}|=|I|\}$, cf. \eqref{eq:def_ImodM}
	\State $\tilde{I}=\emptyset$
	\State $\tilde{s}=0$
	\While{$|\tilde{I}|<|I|$}
		\State $\tilde{s}=\tilde{s}+1$
		\State choose $M_{\tilde{s}}$ from $P^{I}_{\lambda,n}$ at random
		\State choose $\zb z_{\tilde{s}}$ from $[0,M_{\tilde{s}}-1]^d\cap\Z^d$ uniformly at random
		\If {$\{\zb k\in I\colon \not\exists \zb h\in I\setminus\{\zb k\}\text{ with }\zb k\cdot \zb z_{\tilde{s}}\equiv\zb h\cdot \zb z_{\tilde{s}}\imod{M_{\tilde{s}}}\}\not\subset \tilde{I}$}
	    \State compute $\tilde{I}=\tilde{I}\cup \{\zb k\in I\colon \not\exists \zb h\in I\setminus\{\zb k\}\text{ with }\zb k\cdot \zb z_{\tilde{s}}\equiv\zb h\cdot \zb z_{\tilde{s}}\imod{M_{\tilde{s}}}\}$
	    \Else\label{alg:construct_mr1l_I_line:avoid_useless_r1l_1}
			\State	$\tilde{s}=\tilde{s}-1$
	    \EndIf\label{alg:construct_mr1l_I_line:avoid_useless_r1l_2}
	\EndWhile
  \end{algorithmic}
  \begin{tabular}{p{1.4cm}p{4.05cm}p{8.84cm}}
    Output: & $M_1,\ldots,M_{\tilde{s}}$ & lattice sizes of rank\mbox{-}1 lattices and\\
	    & $\zb z_1,\ldots,\zb z_{\tilde{s}}$ & generating vectors of rank\mbox{-}1 lattices such that\\
	    & $\Lambda(\zb z_1,M_1,\ldots,\zb z_{\tilde{s}},M_{\tilde{s}})$ &  is a reconstructing multiple rank\mbox{-}1 lattice for $I$
\\    \cmidrule{1-3}
    \multicolumn{3}{l}{Complexity: $\OO{|I|(\log|I|+d)\log|I|}$\quad{\small w.h.p.~for $c(|I|-1)\geq N_I$, $n\lesssim\log |I|$, and $c$ fixed}}
  \end{tabular}
\end{algorithm}

As mentioned above, we will also follow the approach of multiple rank\mbox{-}1 lattices as subsampling schemes of huge single rank\mbox{-}1 lattices.
To this end, we will construct reconstructing multiple rank\mbox{-}1 lattices that consists
of rank\mbox{-}1 lattices which have pairwise distinct lattice sizes $M_1,\ldots,M_{\tilde{s}}$, cf. Algorithm \ref{alg:construct_mr1l_I_distinct_primes}.
Due to Corollary \ref{cor:distinct_mr1l} and Theorem \ref{thm:prob_bound_T} the estimate
\begin{align}
M:=|\Lambda(\zb z_1,M_1,\ldots,\zb z_{\tilde{s}},M_{\tilde{s}})|\le \tilde{s}\,\max_{p\in P^{I,\lambda}_{\tilde{s}}}p\le s\max_{p\in P^{I,\lambda}_s}p
\label{eq:estimate_M_distinct_primes_basic}
\end{align}
hold for
$2\le|I|=T<\infty$, $1>\delta>0$, $1<c\in\R$, $\lambda$ and $s$ as
determined in \eqref{eq:def_lambda1} and \eqref{def:s}
with probability $1-\delta$.
In order to estimate the number of sampling nodes of the outputs of Algorithm \ref{alg:construct_mr1l_I_distinct_primes},
we need an upper bound on the largest prime within $P^{I}_{\lambda,s}$.
Taking into account some additional requirements on the number $\lambda$, the inclusion
$P^{I}_{\lambda,s}\subset(\lambda,2\lambda]$ holds.
A similar strategy is already applied in \cite[L.9]{ArGiRo16}.

Since we use the number $s$ as upper bound on $\tilde{s}$ in \eqref{eq:estimate_M_distinct_primes_basic}, the estimate of the following corollary hold for Algorithm \ref{alg:construct_mr1l_I_distinct_primes} and for 
Algorithm \ref{alg:construct_mr1l_uniform_different_primes} as well.

\begin{Corollary}\label{cor:estimate_M_Alg2}
We consider an arbitrary frequency set $I$, $2\le|I|=T<\infty$, with expansion $N_I$, cf. \eqref{eq:def_expansion}. Furthermore, we fix $0<\delta<1$, $1<c\in\R$, and determine $s=\ceil{\left(\frac{c}{c-1}\right)^2\frac{\ln T-\ln\delta}{2}}$ and $\lambda=\max\{c(T-1),N_I,4s\ln s\}$.
Then, with probability not less than $1-\delta$, Algorithm \ref{alg:construct_mr1l_I_distinct_primes} determines a reconstructing multiple rank\mbox{-}1 lattice $\bigcup_{\ell=1}^{\tilde{s}}\Lambda(\zb z_\ell,M_\ell)$ for the frequency set $I$, which consists of at most 
\begin{align}
M\le 2\ceil{\left(\frac{c}{c-1}\right)^2\frac{\ln T-\ln\delta}{2}}\max\{c(T-1),N_I,4s\ln s\}
\label{eq:estimate_M_distinct_primes}
\end{align}
sampling nodes.
If $c(T-1)$ is the dominating term in the maximum in \eqref{eq:estimate_M_distinct_primes}, we estimate
\begin{align}
M\le \ceil{\left(\frac{c}{c-1}\right)^2\frac{\ln T-\ln\delta}{2}}2c(T-1)<C_{c,\delta}T\,\ln T.\label{eq:estimate_M_distinct_primes_1}
\end{align}

\end{Corollary}
\begin{proof}
Since the statement in \eqref{eq:estimate_M_distinct_primes_basic} is a direct consequence of Theorem \ref{thm:prob_bound_T}, the goal is to prove the embedding  $P^{I,\lambda}_s \subset(\lambda,2\lambda]$ which yields \eqref{eq:estimate_M_distinct_primes}.

Due to the fact that $\lambda\ge N_I$ holds, the set $P^I$ contains all primes within $(\lambda,2\lambda]$, cf. Lemma \ref{lem:PI_NI}.
Thus, we have to ensure, that the interval $(\lambda,2\lambda]$ contains at least $s$ different prime numbers
$p_j$, $j=1,\ldots,s$. We follow the argumentation in the proof of \cite[L. 9]{ArGiRo16}
and distinguish four different cases:\newline
\underline{$s\ge 4$.}
We have $\lambda\ge4\,s\ln s>22$.
Due to \cite[Cor. 3]{RoSchoe62}, the number of primes within the interval $(\lambda,2\lambda]$ is bounded from below by $\frac{3\lambda}{5\ln\lambda}$ for all $\lambda\ge20.5$.
We consider the mapping $\lambda\mapsto \frac{3\lambda}{5\ln\lambda}$ for $\lambda>22$ and observe that this mapping monotonically increases with $\lambda$. 
We estimate
\begin{align*}
\frac{3\lambda}{5\ln\lambda}\ge s\frac{\ln{s^{12/5}}}{\ln(4 s\,\ln{s})}\ge s
\end{align*}
for $s^{12/5}\ge 4 s\ln{s}$, which is fulfilled for $s\ge 4$.
\newline
\underline{$s=3$.}
We determine $\lambda\ge 12\ln 3>13$ and
know that for $\lambda>22$ there exist at least $4$ prime numbers within
the interval $(\lambda,2\lambda]$.
The same interval for $13<\lambda\le22$
should contain at least $3$ prime numbers.
We determine
\begin{align*}
\{17,19,23\}&\subset(\lambda,2\lambda]\qquad\text{for}\qquad\lambda\in [11.5,17)\\
\{23,29,31\}&\subset(\lambda,2\lambda]\qquad\text{for}\qquad\lambda\in [17,22]
\end{align*}
Hence, for all $\lambda\ge 11.5$ the interval $(\lambda,2\lambda]$ contains at least $3$  different prime number.
\newline
\underline{$s=2$.}
We determine $\lambda\ge 8\ln{2}>5.5$.
With the argumentations above, we determine two prime numbers within the sets
$(\lambda,2\lambda]$.
\begin{align*}
\{7,11\}&\subset(\lambda,2\lambda]\qquad\text{for}\qquad\lambda\in [5.5,7)\\
\{11,13\}&\subset(\lambda,2\lambda]\qquad\text{for}\qquad\lambda\in [7,11)\\
\{17,19\}&\subset(\lambda,2\lambda]\qquad\text{for}\qquad\lambda\in [11,17)
\end{align*}
For larger $\lambda$ we have at least $3$  prime numbers in the interval under consideration.
\newline
\underline{$s=1$.}
We determine $\lambda\ge c(T-1)>1$. Due to Bertrand's postulate, there exists at least one
prime in the interval $(\lambda,2\lambda]$.

The term $4s\,\ln s$ within the termination of $\lambda$ ensures the existence of sufficiently many
prime numbers that are bounded from above by $2\lambda$ as well as large enough, i.e., larger than $\lambda$.
According to Corollary \ref{cor:distinct_mr1l}, Algorithm \ref{alg:construct_mr1l_I_distinct_primes} determines a multiple rank\mbox{-}1 lattice of pairwise distinct prime sizes less or equal to $2\lambda$ with probability at least $1-\delta$.
We call $M$ the total number of sampling nodes within the output of Algorithm \ref{alg:construct_mr1l_I_distinct_primes} and observe the estimates in \eqref{eq:estimate_M_distinct_primes} and \eqref{eq:estimate_M_distinct_primes_1}.
\end{proof}
The last corollary provides a specific estimate of the number of sampling nodes
within a reconstructing multiple rank\mbox{-}1 lattice  that is determined by Algorithm \ref{alg:construct_mr1l_I_distinct_primes}. We notice that the statement is almost the same as in Corollary
\ref{cor:estimate_M_I_basic}. The single difference is that an estimate $M\le 2sc(T-1)\le C_{c,\delta}T\,\ln{T}$ requires $\lambda:=c(T-1)\ge 4s\,\ln{s}$ in addition to $\lambda:=c(T-1)\ge N_I$.
A rough estimate yields $s\,\ln{s}\le s^2\in\OO{\log^2{T}}$
for fixed $c$ and $\delta$. Thus, a large enough $T$ 
implies the inequality $c(T-1)\ge 4s\,\ln{s}$.
Assuming \eqref{eq:estimate_M_distinct_primes_1} is fulfilled, the sampling sets that are determined
using Algorithm \ref{alg:construct_mr1l_I_distinct_primes} need oversampling, where
the oversampling factor scales logarithmically in terms of the cardinality $T$ of the frequency set $I$.

At this point, we would like to  stress on the modifications of Algorithm \ref{alg:construct_mr1l_I_distinct_primes}.
We deterministically choose the lattice sizes $M_\ell$
in Line \ref{alg:construct_mr1l_I_distinct_primes:determinstic_not_random_lattice_sizes}.
The theoretical results are based on the idea that we choose $s$ rank\mbox{-}1 lattices of distinct sizes
at random, where the lattice sizes are contained in the set $P^{I}_{\lambda,s}$. Since we have to choose each element of 
$P^{I}_{\lambda,s}$ exactly once and the choices of the generating vectors do not depend on each other,
the arrangement of the lattice sizes does not matter.
Moreover, the termination criterion of Algorithm \ref{alg:construct_mr1l_I_distinct_primes} is the one that we used in Algorithm \ref{alg:construct_mr1l_I}, which is already discussed above, cf. the context of \eqref{eq:tilde_I}.
This termination criterion may cause that the number $\tilde{s}$ of rank\mbox{-}1 lattices that are determined by Algorithm  \ref{alg:construct_mr1l_I_distinct_primes} is much less than the number $s$, cf. \eqref{def:s}, that is theoretically determined.
Taking this into account, the growing sequence of lattice sizes $M_1,\ldots,M_{\tilde{s}}$ may avoid
the largest possible lattice sizes.

Another point to mention is that Algorithm \ref{alg:construct_mr1l_I_distinct_primes}
constructs a reconstructing multiple rank\mbox{-}1 lattice only with high probability, provided that $\delta$ is small.
Possibly, the output of Algorithm \ref{alg:construct_mr1l_I_distinct_primes}
is not a reconstructing multiple rank\mbox{-}1 lattice for the frequency set $I$.
In detail, if $\tilde{s}=s$ and $\tilde I\subsetneq I$ hold, 
Algorithm \ref{alg:construct_mr1l_I_distinct_primes} is failing.
Clearly, one can simply detect these cases and may avoid
the corrupt output by restarting Algorithm \ref{alg:construct_mr1l_I_distinct_primes}
or increasing $s$, where the latter approach leads
to multiple rank\mbox{-}1 lattices with a number of sampling nodes 
that may not allow for the estimate in \eqref{eq:estimate_M_distinct_primes}.

\begin{algorithm}[tb]
\caption{Determining reconstructing multiple rank-1 lattices  with pairwise distinct lattice sizes
using the ideas in the context of \eqref{eq:tilde_I}
}\label{alg:construct_mr1l_I_distinct_primes}
  \begin{tabular}{p{1.4cm}p{5.05cm}p{7.8cm}}
    Input: 	& $I\subset\Z^d$ 	& frequency set\\
    		& $c\in(1,\infty)\subset\R$ 	& oversampling factor\\
    		& $\delta\in(0,1)\subset\R$			& upper bound on failure probability
  \end{tabular}
		
  \begin{algorithmic}[1]
  	\State $s=\ceil{\left(\frac{c}{c-1}\right)^2\frac{\ln T-\ln\delta}{2}}$
  	\State $\lambda=c(T-1)$
	\State determine $P^{I}_{\lambda,s}$, cf. \eqref{eq:def_PIlambdan}, and arrange $p_1<\ldots<p_s$
	\State $\tilde{I}=\emptyset$
	\State $\tilde{s}=0$
	\While {$|\tilde{I}|<|I|$ and $\tilde{s}<s$}
		\State $\tilde{s}=\tilde{s}+1$
		\State choose $M_{\tilde{s}}=p_{\tilde{s}}\in P^{I}_{\lambda,s}$\label{alg:construct_mr1l_I_distinct_primes:determinstic_not_random_lattice_sizes}
		\State choose $\zb z_{\tilde{s}}$ from $[0,M_{\tilde{s}}-1]^d\cap\Z^d$ uniformly at random
		\If {$\{\zb k\in I\colon \not\exists \zb h\in I\setminus\{\zb k\}\text{ with }\zb k\cdot \zb z_{\tilde{s}}\equiv\zb h\cdot \zb z_{\tilde{s}}\imod{M_{\tilde{s}}}\}\not\subset \tilde{I}$}
	    \State compute $\tilde{I}=\tilde{I}\cup \{\zb k\in I\colon \not\exists \zb h\in I\setminus\{\zb k\}\text{ with }\zb k\cdot \zb z_{\tilde{s}}\equiv\zb h\cdot \zb z_{\tilde{s}}\imod{M_{\tilde{s}}}\}$
	    \Else\label{alg:construct_mr1l_I_distinct_primes_line:avoid_useless_r1l_1}
			\State	$\tilde{s}=\tilde{s}-1$
	    \EndIf\label{alg:construct_mr1l_I_distinct_primes_line:avoid_useless_r1l_2}
	\EndWhile
  \end{algorithmic}
  \begin{tabular}{p{1.4cm}p{5.05cm}p{7.84cm}}
    Output: & $M_1,\ldots,M_{\tilde{s}}$ & lattice sizes of rank\mbox{-}1 lattices and\\
	    & $\zb z_1,\ldots,\zb z_{\tilde{s}}$ & generating vectors of rank\mbox{-}1 lattices such that\\
	    & $\Lambda(\zb z_1,M_1,\ldots,\zb z_{\tilde{s}},M_{\tilde{s}})$ &  is a reconstructing multiple rank\mbox{-}1 lattice\\
	   	&											& for $I$ with probability at least $1-\delta$  
\\    \cmidrule{1-3}
    \multicolumn{2}{l}{Complexity: $\OO{|I|(\log|I|+d)\log|I|}$}&{\small w.h.p.~for $c(|I|-1)\geq N_I$ and fixed $c$ and $\delta$}
  \end{tabular}
\end{algorithm}

Anyway, a simple example deals with the result of Corollary \ref{cor:estimate_M_Alg2}.
\begin{Example}
We fix $c=3$ and $\delta=\frac{1}{100}$.
For a frequency set $I\subset\Z^d$ with $T:=|I|\ge 64$ and $N_I\le 3(T-1)$, we apply Algorithm \ref{alg:construct_mr1l_I_distinct_primes} in order to construct a reconstructing multiple rank\mbox{-}1 lattice $\Lambda(\zb z_1,M_1,\ldots,\zb z_\ell,M_\ell)$ for $I$.
Then Algorithm \ref{alg:construct_mr1l_I_distinct_primes} will succeed with a total number of sampling values $M=1-s+\sum_{j=1}^\ell M_\ell< 16\,T\,\ln T$ with probability greater than $1-\delta=0.99$.
\end{Example}
\begin{proof}
We estimate
\begin{align*}
s&\le \left(\frac{c}{c-1}\right)^2\frac{\ln T-\ln\delta}{2}+1=
\frac{9}{4}\,\frac{\ln T+\ln 100}{2}+1
<\frac{9}{4}\,\frac{\ln T+\frac{10}{9}\ln 64}{2}+1
\\
&<\frac{19}{8}\,\ln T+1<\frac{21}{8}\,\ln T
\\
4s\,\ln s&\le 3\,(T-1)
\intertext{for $T\ge 64$.
Consequently, the equality $\lambda=3\,(T-1)$ in Cor. \ref{cor:estimate_M_Alg2} holds. The number $M$ of used sampling values is bounded by}
M&\le 2\lambda s\le 6\,(T-1)\,\frac{21}{8}\,\ln T< \frac{63}{4}\,T\,\ln T.
\end{align*}
\end{proof}

\section{Heuristic upgrades}\label{sec:heu_upgr}
In this Section, we consider Theorem \ref{thm:prob_bound_T} in more detail and iteratively apply the ideas therein.

\subsection{Iteratively determined lattices}\label{sec:decreasing_lambda}

We denote $I_1=I$, $T_1=T$, $\lambda_1=\lambda$ and fix the number $n=1$ in Theorem \ref{thm:gen_mr1l}, i.e., the set $P^{I_1}_{\lambda_1,1}$ contains only one prime number $M_1$ larger than $\lambda_1$.
If the strategies described in Section \ref{sec:main_results} succeeds, each Fourier coefficient of a specific frequency within $I_1$ can be uniquely reconstructed
using one of the $s$ rank\mbox{-}1 lattices $\Lambda(\zb v_\ell,M_1)$, $\ell=1,\ldots,s$, that are determined according to Theorem \ref{thm:gen_mr1l}, cf.
Theorem \ref{thm:prob_bound_T} and its interpretation in the context of \eqref{def:union_Ir}. This yields
\begin{align}
I_1&=\bigcup_{\ell=1}^{s}\{\zb k\in I_1\colon\zb k\cdot\zb v_\ell\not\equiv\zb h\cdot\zb v_\ell\imod{M_1},\,\forall\zb h\in I_1\setminus\{\zb k\} \}\nonumber
\intertext{and we estimate}
T&\le\sum_{\ell=1}^s\left|\{\zb k\in I_1\colon\zb k\cdot\zb v_\ell\not\equiv\zb h\cdot\zb v_\ell\imod{ M_1},\,\forall\zb h\in I_1\setminus\{\zb k\} \}\right|\nonumber\\
\frac{T}{s}&\le\max_{\ell=1,\ldots,s}\left|\{\zb k\in I_1\colon\zb k\cdot\zb v_\ell\not\equiv\zb h\cdot\zb v_\ell\imod{M_1},\,\forall\zb h\in I_1\setminus\{\zb k\} \}\right|.\label{eq:choose_best_of_s_vectors}
\end{align}
Accordingly, with probability at least $1-\delta$ 
we can choose $\ell'\in\{1,\ldots,s\}$, determine $\zb z_1=\zb v_{\ell'}$ such that
sampling along the rank\mbox{-}1 lattice $\Lambda(\zb z_1,M_1)$ allows for the unique
reconstruction of at least
$\frac{T}{s}$ Fourier coefficients of each polynomial $p_1=\sum_{\zb k\in I_1}\hat{p}_{\zb k}\e^{2\pi\ii\zb k\cdot\zb \circ}\in\Pi_{I_1}$.
Let $\Lambda(\zb z_{1},M_1)$ be a rank\mbox{-}1 lattice coming from the ideas above. We define
\begin{align*}
\tilde{I}_1=&\{\zb k\in I_1\colon\zb k\cdot\zb z_{1}\not\equiv\zb h\cdot\zb z_{1}\bmod M_1,\,\forall\zb h\in I_1\setminus\{\zb k\} \},
\end{align*}
where we assume that $\Lambda(\zb z_{1},M_1)$ is a rank\mbox{-}1 lattice such that $|\tilde{I}_1|\ge\frac{T}{s}$
holds.
We determine the Fourier coefficients of $p_1$
$$
\hat{p}_{\zb k}=M_{1}^{-1}\sum_{l=0}^{M_{1}-1}p_1\left(\frac{j}{M_{1}}\zb z_{1}\right)\e^{-2\pi\ii\frac{l}{M_{1}}\zb k\cdot\zb z_{1}}
$$
for all frequencies $\zb k\in \tilde{I}_1$ using a single one-dimensional FFT.
The essential steps are the computation of the frequency set $\tilde{I}_1$ and a rank\mbox{-}1 lattice FFT, cf. \cite[Alg. 3.2]{kaemmererdiss}, which have a complexity in $\OO{(d+\log{T_1})T_1}$ and $\OO{M_{1}\log M_{1}+dT_1}$, respectively.
Therefore the total complexity of this approach is in $\OO{M_{1}\log M_{1}+(d+\log T_1)T_1}$.
In a next step we consider the polynomial
\begin{align}
p_2(\zb x)=p_1(\zb x)-\sum_{\zb k\in \tilde{I}_1}\hat{p}_{\zb k}\e^{2\pi\ii\zb k\cdot\zb x}=\sum_{\zb k\in I_1\setminus \tilde{I}_1}\hat{p}_{\zb k}\e^{2\pi\ii\zb k\cdot\zb x}\label{eq:reduced_polynomial}
\end{align}
and reduce the problem to the reconstruction of the trigonometric polynomial $p_2\in\Pi_{I_2}$ with a frequency set $I_2=I_1\setminus\tilde{I}_1$ of lower cardinality. We apply Theorems \ref{thm:gen_mr1l} with $n=1$ and \ref{thm:prob_bound_T} to the trigonometric polynomial $p_2$ in \eqref{eq:reduced_polynomial}.
In this way we apply this strategy successively until the frequency set $I_j=\emptyset$ is empty.
Algorithm \ref{alg:construct_mr1l_3} depicts the construction of the sampling sets that
are formed using this strategy. A detailed analysis of the algorithm leads to the 
following findings.

\begin{algorithm}[tb]
\caption{Determining reconstructing multiple rank-1 lattices (Theorem \ref{thm:decreasing_lambda})}\label{alg:construct_mr1l_3}
  \begin{tabular}{p{1.4cm}p{5.05cm}p{7.8cm}}
    Input: 	& $I\subset\Z^d$ 	& frequency set\\
    		& $c\in(1,\infty)\subset\R$ 	& oversampling factor\\
    		& $\delta\in(0,1)\subset\R$			& upper bound on failure probability
  \end{tabular}
  		
  \begin{algorithmic}[1]
  \State $T_1=|I|$
  \State $\ell=0$
  \While{$|I|>0$}
  	\State $\ell=\ell+1$
  	\State $T=|I|$
	\State $s=\ceil{\left(\frac{c}{c-1}\right)^2\frac{\ln T + \ln T_1-\ln \delta}{2}}$
  	\State $\lambda=c(T-1)$
  	\State determine $M_{\ell}\in P^{I}_{\lambda,1}$\label{alg:construct_mr1l_3_determine_M_ell}
  	\For {$j=1,\ldots,s$}
		\State choose $\zb v_j\in[0,M_{\ell}-1]^d\cap\Z^d$ uniformly at random
		\State determine $K_j=\left|\left\{\zb k\in I\colon \zb k\cdot\zb v_j\not\equiv\zb h\cdot\zb v_j\imod{M_{\ell}},\,\forall\zb h\in I\setminus\{\zb k\}\right\}\right|$
	\EndFor
	\If{$\max_{j=1,\ldots,s}K_j\ge 1$}
	\State choose $j_0\in\{t\colon K_{t}=\max_{j=1,\ldots,s}K_j\}$ and determine $\zb z_\ell=\zb v_{j_0}$
	\State $I=I\setminus\{\zb k\in I\colon \zb k\cdot\zb z_{\ell}\not\equiv\zb h\cdot\zb z_{\ell}\imod{M_{\ell}},\,\forall\zb h\in I\setminus\{\zb k\}\}$
	\Else\label{alg:construct_mr1l_3_avoid_needless_lattices1}
		\State $\ell=\ell-1$
	\EndIf\label{alg:construct_mr1l_3_avoid_needless_lattices2}
  \EndWhile
  \end{algorithmic}
  \begin{tabular}{p{1.4cm}p{5.05cm}p{7.8cm}}
    Output: & $M_1,\ldots,M_{\ell}$ & lattice sizes of rank\mbox{-}1 lattices and\\
	    & $\zb z_1,\ldots,\zb z_{\ell}$ & generating vectors of rank\mbox{-}1 lattices such that\\
	    & $\Lambda(\zb z_1,M_1,\ldots,\zb z_{\ell},M_{\ell})$ &  is a reconstructing multiple rank\mbox{-}1 lattice
\\    \cmidrule{1-3}
    \multicolumn{2}{l}{Complexity: $\OO{d\,N_I\,|I|\log^2|I|}$}&{\small w.h.p.~for fixed $c$ and $\delta$}\\
    \multicolumn{2}{l}{\phantom{Complexity: }$\OO{d\,|I|\log^3|I|}$}&{\small w.h.p.~for $|I|\log^2|I|\gtrsim N_I^2$, fixed $c$ and $\delta$}
  \end{tabular}
\end{algorithm}

\begin{Theorem}\label{thm:decreasing_lambda}
We assume $I_1:=I\subset\Z^d$
is a frequency set of finite cardinality $|I_1|<\infty$,
$\delta\in\R$, $0<\delta<1$, and $c\in\R$, $c>1$ fixed real numbers.
In addition, we denote for $j=1,\ldots$
\begin{itemize}
\item $T_j:=|I_j|$,
\item $\lambda_j:=c(T_j-1)$,
\item $s_j:=\ceil{\left(\frac{c}{c-1}\right)^2\frac{\ln T_j+\ln T_1-\ln \delta}{2}}$,
\item $M_j\in P^{I_j}_{\lambda_j,1}$ prime number larger than $\lambda_j$,
\item $\zb z_j$ the best possible of $s_j$ chosen random generating vectors,
\item $I_{j+1}:=\{\zb k\in I_{j}\colon \exists\zb h\in I_{j}\setminus\{\zb k\} \text{ s.t. }\zb k\cdot\zb z_j\equiv\zb h\cdot\zb z_j\imod{M_j}\}$,
\end{itemize}
and a number $k:=\min\{\ceil{s_1\,\ln T_1},T_1\}$.
Then, with probability not less than $1-\frac{k\delta}{T_1}\ge 1-\delta$ the cardinality of the frequency set $I_{k+1}$ is zero, i.e.,
$\Lambda(\zb z_1,M_1,\ldots,\zb z_k,M_k)$ is a reconstructing multiple rank\mbox{-}1 lattice for $I$.
\end{Theorem}

\begin{proof}
We start with the frequency set $I_1$ and choose $s_1$ vectors from $[0,M_1-1]^d$ at random.
Due to Theorem \ref{thm:prob_bound_T}, Theorem \ref{thm:gen_mr1l}, and the considerations in the context of \eqref{eq:choose_best_of_s_vectors}, the best possible of the $s_1$ randomly chosen vectors provides
$T_2\le (1-1/s_1)\,T_1$ with probability of at least $1-\frac{\delta}{T_1}$, i.e., the probability of failing this
property is not greater than $\frac{\delta}{T_1}$.

Similar considerations can be done for $I_j$, $j=2,\ldots ,k$.
We have $s_j=\ceil{\left(\frac{c}{c-1}\right)^2\frac{\ln T_j-\ln \frac{\delta}{T_1}}{2}}$.
We apply Theorems \ref{thm:prob_bound_T} and \ref{thm:gen_mr1l} on $I_j$ and we determine the best of $s_j$ randomly chosen vectors.
With probability less or equal $\frac{\delta}{T_1}$ we observe $T_{j+1}>(1-1/s_j)\,T_j$. 
We estimate
\begin{align*}
\mathcal{P}\left(\bigcap_{j=1}^k\left\{T_{j+1}\le(1-1/s_j)\,T_j\right\}\right)&=1-\mathcal{P}\left(\bigcup_{j=1}^k\left\{T_{j+1}>(1-1/s_j)\,T_j\right\}\right)\\
&\ge 1-\sum_{j=1}^k\mathcal{P}\left(\left\{T_{j+1}>(1-1/s_j)\,T_j\right\}\right)\ge 1-\frac{k\delta}{T_1}.
\end{align*}
Now, we estimate the number $k$.
On the one hand, if $T_{j+1}\le(1-1/s_j)\,T_j$, $j=1\ldots, k$, hold, we have
$T_{k+1}\le T_1\prod_{j=1}^{k}(1-1/s_j)\le(1-1/s_1)^k\,T_1$ and in particular $T_{k+1}=0$ for
$(1-1/s_1)^k<T_1^{-1}$, which is fulfilled for
\begin{align}
k\ge s_1\,\ln T_1>\frac{\ln{T_1}}{\ln{s_1}-\ln(s_1-1)}.\label{eq:decreasing_lambda_k_est1}
\end{align}

On the other hand, since $\infty>s_j\ge 1$, the conditions $T_{j+1}\le(1-1/s_j)\,T_j$, $j=1\ldots, k$, imply
$$0\le T_{j+1}\le\begin{cases}
T_j-1&\colon T_j\geq 1,\\
0&\colon T_j=0,
\end{cases}$$
and $T_{T_1+1}= 0$, i.e., 
even if the estimate in \eqref{eq:decreasing_lambda_k_est1} yields $k\ge s_1\,\ln T_1$ the number $k$ of used lattices may be already bounded from above by $k\le T_1$.

In summary, with probability greater than $1-\frac{k\delta}{T_1}\ge 1-\delta$, we construct
a reconstructing multiple rank\mbox{-}1 lattice consisting of at most
$k=\min\{\ceil{s_1\,\ln T_1},T_1\}$ joined rank\mbox{-}1 lattices.
\end{proof}

The latter approach is indicated in Algorithm \ref{alg:construct_mr1l_3}.
An additional restriction on the expansion $N_I$ of the considered frequency set $I$ allows for a rough estimate of the number $M$ of sampling nodes of the reconstructing rank\mbox{-}1 lattice that is constructed by Algorithm \ref{alg:construct_mr1l_3}.

\begin{Corollary} \label{cor:estimate_M_decreasing_lambda}
We assume $I_1:=I\subset\Z^d$
is a frequency set of finite cardinality $T_1=|I_1|\ge 3$,
$\delta\in\R$, $0<\delta<1$, and $c\in\R$, $c>1$ fixed real numbers.
In addition, we assume $N_{I_1}\le\frac{c (T_1-1)}{\ln T_1}$.
Then, with probability at least $1-\delta$ Algorithm \ref{alg:construct_mr1l_3} constructs a reconstructing multiple rank\mbox{-}1 lattice $\Lambda(\zb z_1,M_1,\ldots,\zb z_\ell,M_\ell)$ of
\begin{align}
M\le 1-\ell+\sum_{j=1}^{\ell}M_j\le\left(2\left(\frac{c}{c-1}\right)^2\left(2\ln T_1-\ln \delta\right)+6\right)cT_1\label{eq:upper_bound_M_T_decreasing_lambda}
\end{align}
 sampling nodes.
\end{Corollary}

\begin{proof}
Due to Bertrand's postulate and Lemma \ref{lem:M_g_NI} we can find $M_j$, $j=1,\ldots,k$, in Theorem \ref{thm:decreasing_lambda}
such that $M_j\le 2\max\{\lambda_j,N_{I_j}\}\le 2\max\{\lambda_j,N_{I_1}\}$. Moreover, we know that the inequality
$\ell\le k$ holds with probability at least $1-\delta$. Accordingly, we estimate
\begin{align*}
M&\le 1-\ell+\sum_{j=0}^{\ell} M_j\le 2\sum_{j=1}^k\max\{\lambda_j,N_{I_1}\}
\le 2\left(kN_{I_1}+\sum_{j=1}^k \lambda_j\right)\\
&\le 2k N_{I_1}+2c\sum_{j=1}^k (T_j-1)
\le 2k \frac{c(T_1-1)}{\ln T_1}+2c\left(-k+\sum_{j=1}^k T_j\right)
\\
&\le 2ck \left(\frac{T_1-1}{\ln T_1}-1\right)+2cT_1\left(1+\sum_{j=1}^{k-1} \prod_{l=1}^j\left(1-1/s_{l}\right)\right)\\
&\le 2c \left(\left(s_1+\frac{1}{\ln T_1}\right)T_1\right)+2cT_1\sum_{j=0}^{\infty} \left(1-1/s_1\right)^j\\
&< 2cT_1\left(2s_1+\frac{1}{\ln T_1}\right)\overset{T_1>\e}{<}2(2s_1+1)cT_1
\le \left(2\left(\frac{c}{c-1}\right)^2\left(2\ln T_1-\ln \delta\right)+6\right)cT_1
\end{align*}
with probability at least $1-\delta$.
\end{proof}
We stress on the fact that the upper bound on the sampling nodes of a reconstructing multiple rank\mbox{-}1 lattice
in \eqref{eq:upper_bound_M_T_decreasing_lambda} is larger than the bound in \eqref{eq:upper_bound_M_T_general}.
Nevertheless the bounds are in the same order $\Theta(T_1\,\ln T_1)$ and $\Theta(\ln\delta)$ with respect to $T_1=|I_1|$ and $\delta$, respectively. Moreover, both upper bounds do not depend on the dimension $d$.

At this point, we comment on some characteristics of Algorithm \ref{alg:construct_mr1l_3}.
Against the theoretical considerations, we added the lines \ref{alg:construct_mr1l_3_avoid_needless_lattices1}--\ref{alg:construct_mr1l_3_avoid_needless_lattices2} in Algorithm \ref{alg:construct_mr1l_3}, that avoid to use a rank\mbox{-}1 lattice that does not
yield additional information compared to the already determined rank\mbox{-}1 lattices.
Furthermore, Algorithm \ref{alg:construct_mr1l_3} will not terminate until
a reconstructing rank\mbox{-}1 lattice for the input is determined.
Assuming Algorithm \ref{alg:construct_mr1l_3} will not terminate implies that
at one $\ell$ the while loop is an endless loop, which means that $|I|$
does not decrease. The lattice size $M_{\ell}$ is fixed in this endless loop.
Accordingly, we test $n s$, $n=1,\ldots$, vectors $\zb v_j$ for their reconstruction property.
Taking Theorem \ref{thm:gen_mr1l} and the estimates in \eqref{eq:proof_of_thm:gen_mr1l} into account,
we observe that the probability that a Fourier coefficient of a single fixed frequency within $I$ can not be reconstructed
using $n s$ randomly chosen vectors is at most  $\e^{-2ns\left(\frac{c-1}{c}\right)^2}$, which decays exponentially
in $n$. Thus, in practice we will not observe endless loops in Algorithm \ref{alg:construct_mr1l_3} and, consequently, the algorithm terminates.
Nevertheless, the output of Algorithm \ref{alg:construct_mr1l_3} may not fulfill the upper bound
in \eqref{eq:upper_bound_M_T_decreasing_lambda}.
From this point of view, each multiple rank-1 lattice $\Lambda(\zb z_1,M_1,\ldots,\zb z_\ell,M_\ell)$ which is the output of Algorithm \ref{alg:construct_mr1l_3}
is actually a reconstucting multiple rank-1 lattice for the frequency set $I$ under consideration. Consequently, the parameter $\delta$, the so-called upper bound on the failure probability, does not bound the probability that
the reconstruction property of the sampling set $\Lambda(\zb z_1,M_1,\ldots,\zb z_\ell,M_\ell)$ fails. However,
$\delta\in(0,1)$ is an upper bound on the probability that the number $\ell$ of joined rank-1 lattices is greater than $k$ in Theorem~\ref{thm:decreasing_lambda} and that the number of sampling nodes within $\Lambda(\zb z_1,M_1,\ldots,\zb z_\ell,M_\ell)$ does not fulfill inequality \eqref{eq:upper_bound_M_T_decreasing_lambda} in Corollary~\ref{cor:estimate_M_decreasing_lambda}.

\begin{algorithm}[tb]
\caption{Determining reconstructing MR1L (Theorem \ref{thm:decreasing_lambda})}\label{alg:construct_mr1l_4}
  \begin{tabular}{p{1.4cm}p{5.05cm}p{7.8cm}}
    Input: 	& $I\subset\Z^d$ 	& frequency set\\
    		& $c\in(1,\infty)\subset\R$ 	& oversampling factor\\
    		& $\delta\in(0,1)\subset\R$			& upper bound on failure probability
  \end{tabular}
  		
  \begin{algorithmic}[1]
	\Statex lines 1 to 7 identical with those of Algorithm \ref{alg:construct_mr1l_3}
	\makeatletter
	\setcounter{ALG@line}{7}
	\makeatother
  	\State determine $M_{\ell}=\min\left\{P^{I,\lambda}_{\ell}\setminus\{M_1,\ldots,M_{\ell-1}\}\right\}$
  	\Statex lines 9 to 19 identical with those of Algorithm \ref{alg:construct_mr1l_3}
  \end{algorithmic}
  \begin{tabular}{p{1.4cm}p{5.05cm}p{7.84cm}}
    Output: & $M_1,\ldots,M_{\ell}$ & lattice sizes of rank\mbox{-}1 lattices and\\
	    & $\zb z_1,\ldots,\zb z_{\ell}$ & generating vectors of rank\mbox{-}1 lattices such that\\
	    & $\Lambda(\zb z_1,M_1,\ldots,\zb z_{\ell},M_{\ell})$ &  is a reconstructing multiple rank\mbox{-}1 lattice
\\    \cmidrule{1-3}
    \multicolumn{2}{l}{Complexity: $\OO{d\,N_I\,|I|\log^2|I|}$}&{\small w.h.p.~for fixed $c$ and $\delta$}\\
    \multicolumn{2}{l}{\phantom{Complexity: }$\OO{d\,|I|\log^3|I|}$}&{\small w.h.p.~for $|I|\log^2|I|\gtrsim N_I^2$, fixed $c$ and $\delta$}
  \end{tabular}
\end{algorithm}

Furthermore, we would like to mention that the rank\mbox{-}1 lattice sizes $M_1,\ldots,M_{\ell}$, that are determined using Algorithm \ref{alg:construct_mr1l_3}, are not necessarily distinct, and thus, the determined reconstructing multiple rank\mbox{-}1 lattice is not necessarily a subsampling scheme of a huge single rank\mbox{-}1 lattice, cf. the context of Corollary \ref{cor:distinct_mr1l}.
A simple modification on line \ref{alg:construct_mr1l_3_determine_M_ell} of Algorithm \ref{alg:construct_mr1l_3}
provides even this nice property, cf. Algorithm \ref{alg:construct_mr1l_4}.
Certainly, an additional assumption $\lambda_1=c(T_1-1)\ge 4k\,\ln k$ in Corollary \ref{cor:estimate_M_decreasing_lambda} would guarantee the existence of
at least $k$ distinct primes within the interval $(\lambda_1,2\lambda_1]$, cf. Corollary \ref{cor:estimate_M_Alg2}  for details.
For $N_I$ small enough, an obvious estimate yields $\sum_{j=1}^\ell M_j\le 2kcT_1\in\OO{T_1\ln^2T_1}$, i.e., in comparison to the statement of Corollary \ref{cor:estimate_M_decreasing_lambda} the estimate suffers from an additional logarithmic term in $T_1$.
However, we did not exploit the decreasing $\lambda_j$, $j=1,\ldots, k$, in the obvious estimate.
Depending on $T_1$, $N_{I_1}$, $\delta$, and $c$, a detailed analysis of this estimate
might save a logarithmic term here.

\subsection{Stretching the set \texorpdfstring{$P^{I}$}{PI}}
\label{sec:stretching_PI}

Up to now, all presented algorithms need to choose the lattice sizes from a set $P^{I}_{\lambda,n}$, where
possibly $\min_{p\in P^{I}_{\lambda,n}}p\gg\lambda:=c(T-1)$, cf. \eqref{eq:def_lambda1}, occurs. In order to allow for the use of smaller
lattice sizes, we will weaken the requirements on the prime number sets
$P^{I}$.
To this end, we consider the set $I_{\bmod M}$, cf. \eqref{eq:def_ImodM}, in more detail and define the sets
\begin{align*}
I_{\bmod M}^{\textnormal{u}}:&=\{\zb h\in I_{\bmod M}\colon \zb h\equiv\zb k_1\imod{M}, \zb k_1\in I,\text{ and }\zb h\not\equiv \zb k_2 \imod{M}\forall \zb k_2\in I\setminus\{\zb k_1\}\}\\
I_{\bmod M}^{\textnormal{c}}:&=
\{
\zb h\in I_{\bmod M}\colon \exists \zb k_1,\zb k_2\in I, \zb k_1\neq\zb k_2, \text{ s.t. } \zb h\equiv \zb k_1\equiv \zb k_2\imod{M}
\}\\
&=I_{\bmod M}\setminus I_{\bmod M}^{\textnormal{u}}.
\end{align*}
The modulo $M$ operation on $I$ is a mapping from $I$ to $I_{\bmod M}$.
Accordingly, the sets $I_{\bmod M}^{\textnormal{u}}$ consists of the images $\zb h\in I_{\bmod M}$ of all uniquely mapped
vectors $\zb k\in I$ and the set
$I_{\bmod M}^{\textnormal{c}}$ contains the vectors $\zb h\in I_{\bmod M}$
that are the images of at least two different $\zb k_1,\zb k_2\in I$, i.e.,
the images of vectors from $I$ that collides under the modulo operation.
In addition, we define the sets  $I^{\textnormal{u}}_M$ as the inverse image of 
$I_{\bmod M}^{\textnormal{u}}$.

Sampling a trigonometric polynomial $p\in\Pi_I$ at a rank\mbox{-}1 lattice of lattice size $M$ reads as
\begin{align*}
p\left(\frac{j}{M}\zb z\right)=\sum_{\zb k\in I}\hat{p}_{\zb k}\e^{2\pi\ii\zb k\cdot\zb z\frac{j}{M}}
=\sum_{\zb k\in I}\hat{p}_{\zb k}\e^{2\pi\ii(\zb k\bmod{M})\cdot\zb z\frac{j}{M}}
=\sum_{\zb l\in I_{\bmod M} }\left(\sum_{\substack{\zb k\in I\\\zb l=\zb k\bmod M}}\hat{p}_{\zb k}\right)\e^{2\pi\ii\zb l\cdot\zb z\frac{j}{M}}.
\end{align*}
Accordingly, a reconstructing multiple rank\mbox{-}1 lattice for $I_{\!\bmod M}$ that consists of single rank\mbox{-}1 lattices of lattice sizes
$M_1=\ldots=M_s=M$ allow for the unique reconstruction of all
\begin{align}
\hat{\tilde{p}}_{\zb l}:=\sum_{\substack{\zb k\in I\\\zb l=\zb k\bmod M}}\hat{p}_{\zb k}.\label{eq:def_aliasing_modM}
\end{align}
Obviously, if $\zb k\in I_{M}^{\textnormal{u}}$ holds, the sum at the right hand side of \eqref{eq:def_aliasing_modM}
contains exactly one summand and we achieve $\hat{p}_{\zb k}=\hat{\tilde{p}}_{\zb k\bmod M}$, i.e., we uniquely reconstruct
all Fourier coefficients $\hat{p}_{\zb k}$, $\zb k\in I^{\textnormal{u}}_M$.
If we observe $I_{\bmod M}^{\textnormal{c}}\neq\emptyset$, we cannot reconstruct all Fourier coefficients of $p\in\Pi_I$
in this way. However, a straightforward strategy similar to the approach in Section \ref{sec:decreasing_lambda}
can be applied. We consider the trigonometric polynomial
\begin{align}
p_2(\zb x)=p(\zb x)-\sum_{\zb k\in I_{M}^{\textnormal{u}}}\hat{p}_{\zb k}\e^{2\pi\ii\zb k\cdot\zb x}=\sum_{\zb k\in I\setminus I_{M}^{\textnormal{u}}}\hat{p}_{\zb k}\e^{2\pi\ii\zb k\cdot\zb x}\label{eq:reduced_polynomial1}
\end{align}
in a next step.

At this point, we would like to mention that it is not necessary that the sets $I_{\bmod M_j}^{\textnormal{u}}$, $j=1,\ldots,\ell$ are non-empty in order to uniquely reconstruct all trigonometric polynomials with frequencies supported on $I$ using the multiple rank\mbox{-}1 lattice $\Lambda(\zb z_1,M_1,\ldots,\zb z_\ell,M_\ell)$, which is illustrated in the following example.

\begin{example}\label{ex:mr1l_with_aliasing}
Let the frequency set
$$
I=\{\zb k_1,\ldots,\zb k_6\}\subset\Z^d
$$
with $k_{1,j}=\ldots=k_{6,j}$, $j=2,\ldots,d$, and $k_{j,1}=(0,2,5,7,16,21)^\top$, i.e., the set $I$
is located on a line that is  parallel to the first coordinate axis.
We consider the multiple rank\mbox{-}1 lattice $\Lambda(\zb z_1,M_1,\zb z_2,M_2,\zb z_3,M_3)$ with
$z_{j,1}=1$ for $j=1,2,3$ and $M_1=2$, $M_2=3$, $M_3=5$, and
we observe that $I_{\bmod M_j}^{\textnormal{u}}=\emptyset$, $j=1,2,3$, hold.
Nevertheless, the Fourier matrix
$\zb A(\Lambda(\zb z_1,M_1,\zb z_2,M_2,\zb z_3,M_3),I)$ has full column rank.
\end{example}

According to the last example, even the requirement $I_{\bmod M}^{\textnormal{u}}\neq \emptyset$
seems to restrict the considered rank\mbox{-}1 lattice sizes unnecessarily. However, the approaches of the last sections
can easily deal with the sets $I_{\bmod M}$ and the requirement $I_{\bmod M}^{\textnormal{u}}\neq \emptyset$
allows for a successive reduction of the frequency set $I$, cf. the context of \eqref{eq:reduced_polynomial1}.

\subsubsection*{Bound the number of reconstructable Fourier coefficients from below}\label{sec:P_I_bound_below}

One way to guarantee the success in using the aforementioned strategy is to ensure that $I_{\bmod M}^{\textnormal{u}}$
contains at least a significant number of frequencies, which only depends on $M$ for a given frequency set $I$.
We search for the smallest prime number $M$ such that the number of frequencies within $I_{\bmod M}^{\textnormal{u}}$ fulfills $|I_{\bmod M}^{\textnormal{u}}|\ge\frac{|I|}{C}$, $C>1$, and $M>c(|I_{\!\bmod M}|-1)$, $c>1$, i.e., we define
\begin{align}
M^{I,C}_c:=\min\left\{M\in\N\colon M \text{ prime with }|I_{\bmod M}^{\textnormal{u}}|\ge\frac{|I|}{C}\text{ and }M>c(|I_{\!\bmod M}|-1)\right\}.
\label{eq:def_Ptilde}
\end{align}

\begin{algorithm}[tb]
\caption{Determining reconstructing multiple rank-1 lattices (Section \ref{sec:P_I_bound_below})}\label{alg:construct_mr1l_5}
  \begin{tabular}{p{1.4cm}p{5.05cm}p{7.8cm}}
    Input: 	& $I\subset\Z^d$, $|I|<\infty$ 		& frequency set\\
    		& $c\in(1,\infty)\subset\R$ 		& oversampling factor\\
    		& $C\in[1,|I|]$								& fixed constant
  \end{tabular}

  \begin{algorithmic}[1]
  \State $s=0$
  \While{$|I|>1$}
	\State determine $M^{I,C}_c$, cf. \eqref{eq:def_Ptilde}
	\State construct $I_{\bmod M^{I,C}_c}$, cf. \eqref{eq:def_ImodM}
	\State construct reconstructing multiple rank\mbox{-}1 lattice $\Lambda'=\Lambda(\zb z_{s+1},M_{s+1},\ldots,\zb z_{s+\ell'},M_{s+\ell'})$
	 \Statex\hskip\algorithmicindent for $I_{\bmod M^{I,C}_c}$ using Algorithm \ref{alg:construct_mr1l_I} with $M_{s+1}=\ldots=M_{s+\ell'}=M^{I,C}_c$
	\State $s=s+\ell'$
	\State determine $I=I\setminus\{\zb k\in I\colon \zb k\bmod M^{I,C}_c=\zb h\in I_{\bmod M^{I,C}_c}^{\textnormal{u}}\}$
  \EndWhile
	
  \end{algorithmic}
  \begin{tabular}{p{1.4cm}p{5.05cm}p{7.84cm}}
    Output: & $M_1,\ldots,M_s$ & lattice sizes of rank\mbox{-}1 lattices and\\
	    	& $\zb z_1,\ldots,\zb z_s$ & generating vectors of rank\mbox{-}1 lattices such that\\
	    	& $\Lambda(\zb z_1,M_1,\ldots,\zb z_s,M_s)$ &  is a reconstructing multiple rank\mbox{-}1 lattice
\\
    \midrule
    \multicolumn{3}{l}{Complexity: $\OO{N_I(d|I|\log |I|+\log\log N_I)}$\quad\small with high prob.~for fixed $c$ and $C$}
  \end{tabular}
\end{algorithm}

The application of one of the Algorithms \ref{alg:construct_mr1l_uniform} or
\ref{alg:construct_mr1l_I} with fixed rank\mbox{-}1 lattice sizes $M$ allows for the direct reconstruction of all Fourier coefficients
$\hat{p}_{\zb k}$, $\zb k\in I_{M}^{\textnormal{u}}$.
Taking the last considerations as well as the iterative approach described in the context of \eqref{eq:reduced_polynomial1} into account, we achieve the search strategy for reconstructing multiple rank-1 lattices
that is outlined in Algorithm~\ref{alg:construct_mr1l_5}.
Similar to the considerations in Theorem \ref{thm:decreasing_lambda}, one realizes
that the loop in Algorithm~\ref{alg:construct_mr1l_5} is passed at most
$\min\{C\log{T},T\}$ times.
Assuming that $M_c^{I,C}$ is in the magnitude of $T=|I|$,
with a probability of at least $1-\delta$ the number of sampling nodes in the output of Algorithm \ref{alg:construct_mr1l_5}
can simply be estimated by $T$ times a polylogarithmic term in $T$.

\subsection{Combining both strategies}
We can even combine the strategies described in Sections \ref{sec:decreasing_lambda} and \ref{sec:stretching_PI}. In each step of the construction in Algorithm \ref{alg:construct_mr1l_5}, we can choose the best of the determined $\ell'$ rank\mbox{-}1 lattices that is used for a reconstruction. With high probability one of the rank\mbox{-}1 lattices allows for the unique reconstruction of $\frac{|I|}{Cs}$ frequencies of $I$. Consequently, we can go straightforward in determining a reconstructing rank\mbox{-}1 lattice for a smaller frequency set of cardinality of at most $\left(1-\frac{1}{Cs}\right)|I|$ frequencies, where $s$ depends on $c$, $\delta$, $|I|$, and $|I_{\bmod M}|$. Restricting the used prime lattice sizes $M$ being distinct from those that are already used in each step, we can also determine reconstructing multiple rank\mbox{-}1 lattices
that are subsampling schemes of a huge single rank\mbox{-}1 lattice, cf. \cite[Cor. 2.2]{Kae16}.

\section{Numerical tests}\label{sec:num_test}

Since the complexity of the fast Fourier transform related to the reconstruction of the trigonometric polynomial
mainly depends on the number of used sampling nodes, we focus 
on the oversampling factor $M/|I|$, i.e., the ratio of the number of used sampling nodes $M$ compared to the number
$|I|$ of Fourier coefficients that will be reconstructed. In addition, the aforementioned complexity depends on the number $s$
of used single rank\mbox{-}1 lattices that are combined to a multiple rank-1 lattice. Consequently, we also present these values.
Furthermore, we will compute some of the condition numbers of the Fourier matrices $\zb A(\Lambda(\zb z_1,M_1,\ldots,\zb z_s,M_s),I)$, cf.~\eqref{eqn:Fourier_matrix}, that realizes the evaluation of the
considered trigonometric polynomials, even though we do not have any  theoretical statements on them.
Due to the computational costs, we computed the condition numbers only for frequency sets of low cardinality, i.e., $|I|\le20\,000$.

A first numerical test on a very specific frequency set $I$ shows the advantages of Algorithm \ref{alg:construct_mr1l_5}.
Second, we consider random frequency sets and approve in some sense
the universality of Algorithms \ref{alg:construct_mr1l_uniform} and
\ref{alg:construct_mr1l_uniform_different_primes}.
Last, we consider dyadic hyperbolic crosses as frequency sets $I$ and compare Algorithms \ref{alg:construct_mr1l_uniform} to \ref{alg:construct_mr1l_4}. We leave out Algorithm \ref{alg:construct_mr1l_5} here, since we do not expect significant advantages due to the structure of the frequency sets.

\subsection{Specific frequency set}

We fix the parameters $\delta=0.5$ and $c=1.1$ and we consider
the $d$-dimensional frequency set
$$
I:=\{\zb k_1,\ldots,\zb k_5\}=\left\{
\left(\begin{array}{c}
0\\
\zb h
\end{array}\right),
\left(\begin{array}{c}
6\,251\\
\zb h
\end{array}\right)
\left(\begin{array}{c}
10\,879\\
\zb h
\end{array}\right)
\left(\begin{array}{c}
15\,457\\
\zb h
\end{array}\right)
\left(\begin{array}{c}
19\,499\\
\zb h
\end{array}\right)
\right\},
$$
where $\zb h\in \Z^{d-1} $  is fixed.
Algorithms \ref{alg:construct_mr1l_uniform} and \ref{alg:construct_mr1l_uniform_different_primes}
will determine a multiple rank\mbox{-}1 lattice that contains at least one rank\mbox{-}1 lattice of size  $M_1=19\,501$
since this is the smallest prime number larger than $N_I=19\,499$.
The smallest prime number within $P^I$ is  $53$. Accordingly, the outputs of Algorithms \ref{alg:construct_mr1l_I} to
\ref{alg:construct_mr1l_4} will contain at least one rank\mbox{-}1 lattice of size $M_1=53$.
Applying Algorithm~\ref{alg:construct_mr1l_5} with $C=2$ determined $M_1=5$ and $M_2=3$
in each of  our $10\,000$ numerical tests, i.e., the determined reconstructing multiple rank\mbox{-}1 lattices 
consist of only $M=1-2+5+3=7$ sampling nodes. The observed condition numbers 
$\cond(\zb A)=\sqrt{\frac{19+\sqrt{21}}{11-\sqrt{61}}}\approx 2.7191$
of the Fourier matrices $\zb A$ were
fixed.

\subsection{Random frequency sets}

\begin{table}[tb]
\centering
\begin{tabular}{crrr|rr|r}
\toprule
&$T$&$s$&$1-\delta$& $M$& $M/T$ & {\#}reco\\
\midrule
\multirow{5}{*}{\begin{sideways} Algorithm \ref{alg:construct_mr1l_uniform} \end{sideways}}
& 148 & 10 & 2.784e-03 & 3\,061 & 2.07e+01 & 9\,904 \\
& 1\,808 & 15 & 2.346e-05 & 54\,241 & 3.00e+01 & 9\,986 \\
& 22\,026 & 20 & 2.115e-05 & 881\,041 & 4.00e+01 & 9\,999\\
& 268\,337 & 25 & 1.068e-06 & 13\,416\,901 & 5.00e+01 & 10\,000 \\
& 3\,269\,017 & 30 & 1.139e-07 & 196\,141\,261 & 6.00e+01 & 10\,000 \\
\midrule
\multirow{5}{*}{\begin{sideways} Algorithm \ref{alg:construct_mr1l_uniform_different_primes} \end{sideways}}
& 148 & 10 & 2.784e-03 & 3\,315 & 2.24e+01 & 9\,951 \\
& 1\,808 & 15 & 2.346e-05 & 55\,061 & 3.05e+01 & 9\,988 \\
& 22\,026 & 20 & 2.115e-05 & 882\,833 & 4.01e+01 & 9\,999 \\
& 268\,337 & 25 & 1.068e-06 & 13\,420\,041 & 5.00e+01 & 10\,000 \\
& 3\,269\,017 & 30 & 1.139e-07 & 196\,148\,197 & 6.00e+01 & 10\,000 \\
\bottomrule
\end{tabular}
\caption{Applying Algorithms \ref{alg:construct_mr1l_uniform} and \ref{alg:construct_mr1l_uniform_different_primes} for fixed $T$, $s$, and $c=2$, yielded $M=1-s+\sum_{\ell=1}^sM_{\ell}$ and $M/T$.
The resulting multiple rank\mbox{-}1 lattices have been tested for their reconstruction property against 10\,000 randomly chosen frequency sets $I\subset[1,300]^3$ of cardinality $T$, where each $\zb k\in I$ was  randomly and uniformly chosen in $[1,300]^3\cap\Z^3$.
Column {\#}reco shows (a lower bound on) the number of frequency sets $I$, for which each $p\in \Pi_I$ was successfully reconstructable
using the sampling values from the determined multiple rank\mbox{-}1 lattice.
}\label{tab:MR1L_uniform_test}
\end{table}

Since Algorithms~\ref{alg:construct_mr1l_uniform} and \ref{alg:construct_mr1l_uniform_different_primes}
only depend on the cardinality $T$ and the expansion $N_I$ of the considered frequency set $I$,
we performed the following numerical test.
We fix $c=2$ and $s=10,15,20,25,30$, and we compute $T=\left\lfloor\e^{\frac{s}{2}}\right\rfloor$, i.e.,
depending on $c=2$ and $s$ we choose $T$ as large as possible such that the estimate $1-\delta$ of the success probability in Theorem \ref{thm:prob_bound_T} is greater than zero.
Accordingly, the statement of Theorem~\ref{thm:prob_bound_T} ensures the construction of reconstructing multiple rank\mbox{-}1 lattices
only with a low probability.

For each $s\in\{10,15,20,25,30\}$ we constructed exactly one reconstructing multiple rank\mbox{-}1 lattice 
for a frequency set  $I$  with $N_I=299$ and 
$T_s=\left\lfloor\e^{\frac{s}{2}}\right\rfloor$.
Subsequently, we determined $10\,000$ randomly chosen frequency sets $I\subset[1,300]^3$
with $|I|= T_s$, where each frequency $\zb k\in I$ was chosen uniformly at random from $[1,300]^3$.
For each of the frequency sets $I$ we checked the equality in \eqref{def:union_Ir},
which guarantees a full column rank Fourier matrix and, thus, a unique reconstruction of
all trigonometric polynomials in $\Pi_I$ by the means of the sampling values at the multiple rank\mbox{-}1
lattice. In Table \ref{tab:MR1L_uniform_test}, we collect the results and present
the number of frequency sets $I$ that fulfilled the equality in \eqref{def:union_Ir}.
Due to the fact that we do not check the column rank  of the Fourier matrices, 
the numbers in the column {\#}reco in Table \ref{tab:MR1L_uniform_test} are only lower bounds on the
number of frequency sets $I$, where the Fourier matrix $\zb A(\Lambda(\zb z_1,M_1,\ldots,\zb z_s,M_s),I)$
has full column rank.

We observe that the determined multiple rank\mbox{-}1 lattices are reconstructing multiple rank\mbox{-}1 lattices
for at least 99\% of the tested frequency sets $I$ in our numerical tests.
We interpret this observation as a consequence of the rough estimates in our
proofs. In detail, the multiple application of union bounds affects
strongly the estimate on the failure probability $\delta$, thus  
we expect that the considered approaches behave much better
in practice.

\subsection{Dyadic hyperbolic crosses}\label{sec:num_test:hc}

\begin{figure}
\centering
\begin{tikzpicture}[baseline=(current axis.south)]
\begin{semilogxaxis}[
    clip=false,
    scale only axis,
    xmin=1,xmax=2664192, ymin=0, ymax=65,
    title={Oversampling factors},
    xtick={7,34,501,5336,47264,370688,2664192},
    x tick label style={align=center,font=\scriptsize},
    y tick label style={font=\scriptsize},
    xticklabels={{7\\[0.3em]1},{34\\[0.3em]2},{501\\[0.3em]4},{5\,336\\[0.3em]6},{47\,264\\[0.3em]8}, {370\,688\\[0.3em]10},{2\,664\,192\\[0.3em]12}},
    width=0.55\textwidth,
    height=0.35\textwidth,
    legend entries={Algorithm \ref{alg:construct_mr1l_uniform}, Algorithm \ref{alg:construct_mr1l_uniform_different_primes},Algorithm \ref{alg:construct_mr1l_I}, Algorithm \ref{alg:construct_mr1l_I_distinct_primes}, Algorithm \ref{alg:construct_mr1l_3}, Algorithm \ref{alg:construct_mr1l_4}
    },
    transpose legend,
    legend to name = leghypcross6d,
    legend columns=2,
    legend style={font=\footnotesize, /tikz/every even column/.append style={column sep=0.5cm}},
    cycle list name=MR1LOF
    ]
  \addplot coordinates {
(7, 1.043e+001) (34, 1.750e+001) (138, 2.401e+001) (501, 2.817e+001) (1683, 3.404e+001) (5336, 3.805e+001) (16172, 4.201e+001) (47264, 4.600e+001) (134048, 5.000e+001) (370688, 5.600e+001) (1003136, 6.000e+001) (2664192, 6.200e+001) 
};
  \addplot coordinates {
(7, 1.814e+001) (34, 2.221e+001) (138, 2.707e+001) (501, 2.919e+001) (1683, 3.479e+001) (5336, 3.836e+001) (16172, 4.211e+001) (47264, 4.605e+001) (134048, 5.004e+001) (370688, 5.601e+001) (1003136, 6.001e+001) (2664192, 6.200e+001) 
};
  \addplot+[dashed] coordinates {
(7, 3.571e+000) (34, 9.735e+000) (138, 1.001e+001) (501, 1.409e+001) (1683, 2.203e+001) (5336, 1.602e+001) (16172, 1.800e+001) (47264, 2.400e+001) (134048, 2.600e+001) (370688, 3.800e+001) (1003136, 3.000e+001) (2664192, 3.200e+001) 
};
  \addplot+[dashed] coordinates {
(7, 1.857e+000) (34, 1.085e+001) (138, 1.266e+001) (501, 1.222e+001) (1683, 1.617e+001) (5336, 2.212e+001) (16172, 2.002e+001) (47264, 2.801e+001) (134048, 2.601e+001) (370688, 2.600e+001) (1003136, 3.400e+001) (2664192, 3.800e+001) 
};
  \addplot+[dashed] coordinates {
(7, 1.857e+000) (34, 2.147e+000) (138, 2.268e+000) (501, 2.501e+000) (1683, 2.695e+000) (5336, 2.795e+000) (16172, 2.840e+000) (47264, 2.988e+000) (134048, 2.988e+000) (370688, 3.049e+000) (1003136, 3.049e+000) (2664192, 3.089e+000) 
};
  \addplot+[dashed] coordinates {
(7, 1.857e+000) (34, 2.029e+000) (138, 2.355e+000) (501, 2.637e+000) (1683, 2.725e+000) (5336, 2.862e+000) (16172, 2.800e+000) (47264, 2.995e+000) (134048, 3.000e+000) (370688, 3.014e+000) (1003136, 3.067e+000) (2664192, 3.082e+000) 
};
\node  at (axis description cs:0,0) [text width=31pt, anchor=north east, align=center,xshift=14pt,yshift=1.5pt] {\scriptsize $|H_{n}^6|$};
\node  at (axis description cs:0,0) [text width=31pt, anchor=north east, align=center,xshift=14pt,yshift=-11.5pt] {\scriptsize $n$};
\end{semilogxaxis}
\end{tikzpicture}
\begin{tikzpicture}[baseline=(current axis.south)]
\begin{loglogaxis}[
    clip=false,
    scale only axis,
    xmin=1,xmax=16172, ymin=0.9, ymax=10,
    title={Condition numbers},
    xtick={7,34,501,5336},
    ytick={1,2,5,7,9},
    x tick label style={align=center,font=\scriptsize},
    y tick label style={font=\scriptsize},
    xticklabels={{7\\[0.3em]1},{34\\[0.3em]2},{501\\[0.3em]4},{5\,336\\[0.3em]6}},
    yticklabels={1,2,5,7,9},
    width=0.25\textwidth,
    height=0.35\textwidth,
    cycle list name=MR1LOF
    ]
  \addplot coordinates {
(7, 1.947e+000) (34, 1.633e+000) (138, 1.525e+000) (501, 1.499e+000) (1683, 1.440e+000) (5336, 1.392e+000) (16172, 1.394e+000) 
};
  \addplot coordinates {
(7, 1.215e+000) (34, 1.341e+000) (138, 1.476e+000) (501, 1.453e+000) (1683, 1.415e+000) (5336, 1.423e+000) (16172, 1.402e+000) 
};
  \addplot+[dashed] coordinates {
(7, 2.172e+000) (34, 2.074e+000) (138, 2.037e+000) (501, 2.586e+000) (1683, 1.843e+000) (5336, 1.754e+000) (16172, 1.716e+000) 
};
  \addplot+[dashed] coordinates {
(7, 1.000e+000) (34, 1.883e+000) (138, 2.122e+000) (501, 1.857e+000) (1683, 1.712e+000) (5336, 1.608e+000) (16172, 1.763e+000) 
};
  \addplot+[dashed] coordinates {
(7, 1.000e+000) (34, 5.996e+000) (138, 7.933e+000) (501, 7.768e+000) (1683, 7.880e+000) (5336, 8.113e+000) (16172, 8.115e+000) 
};
  \addplot+[dashed] coordinates {
(7, 1.000e+000) (34, 8.201e+000) (138, 7.771e+000) (501, 8.844e+000) (1683, 8.958e+000) (5336, 7.857e+000) (16172, 8.502e+000) 
};
\node  at (axis description cs:0,0) [text width=31pt, anchor=north east, align=center,xshift=14pt,yshift=1.5pt] {\scriptsize $|H_{n}^6|$};
\node  at (axis description cs:0,0) [text width=31pt, anchor=north east, align=center,xshift=14pt,yshift=-11.5pt] {\scriptsize $n$};
\end{loglogaxis}
\end{tikzpicture}
\ref{leghypcross6d}
\caption{Oversampling factors of reconstructing multiple rank\mbox{-}1 lattices and the condition numbers of corresponding Fourier matrices for
hyperbolic cross frequency sets $H_n^6$.}
\label{fig:numtests_hypcross_d6}
\end{figure}
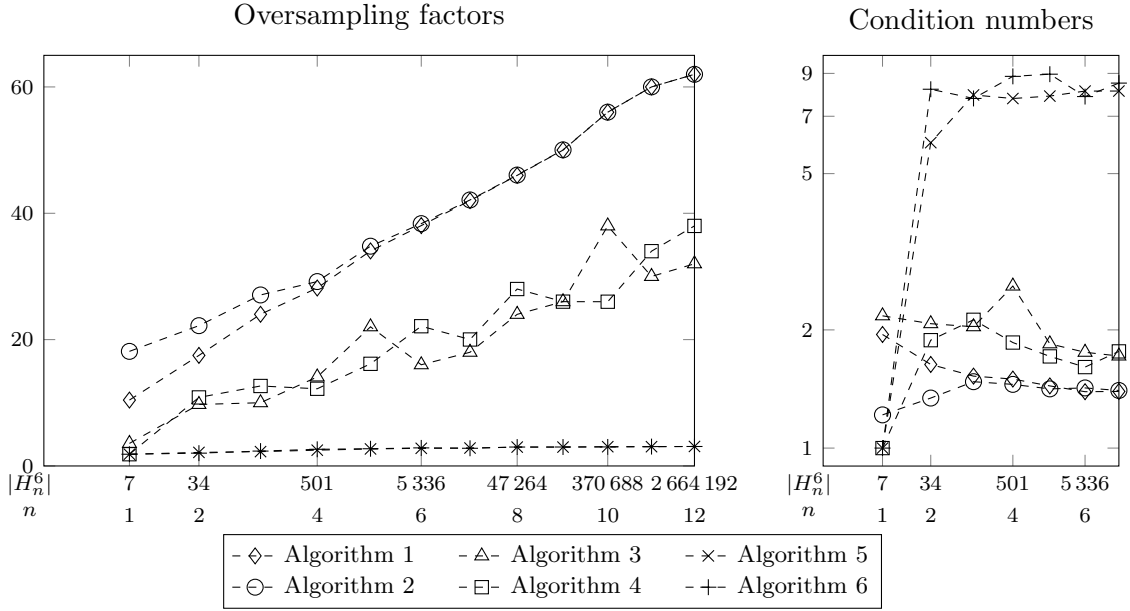

\begin{figure}[tb]
\centering
\begin{tikzpicture}[baseline=(current axis.south)]
  \begin{semilogxaxis}[
     clip=false,
     scale only axis,
     xmin=5,xmax=1900, ymin=0, ymax=75,
     xtick={13,34,89,229,593,1538},
     x tick label style={align=center, font=\scriptsize},
     y tick label style={font=\scriptsize},
     xticklabels={{13\\[0.3em]3},{34\\[0.3em]6},{89\\[0.3em]11},{229\\[0.3em]19},{593\\[0.3em]32},{1\,538\\[0.3em]53}},
    xlabel style={align=center, text width=7cm, font=\footnotesize},
    every axis legend/.append style={nodes={right}},
    legend entries={Algorithm \ref{alg:construct_mr1l_uniform}, Algorithm \ref{alg:construct_mr1l_uniform_different_primes},Algorithm \ref{alg:construct_mr1l_I}, Algorithm \ref{alg:construct_mr1l_I_distinct_primes}, Algorithm \ref{alg:construct_mr1l_3}, Algorithm \ref{alg:construct_mr1l_4}
    },
    transpose legend,
    legend to name = leghypcrossnfix,
    legend columns=2,
    legend style={font=\footnotesize, /tikz/every even column/.append style={column sep=0.5cm}},
    title style={font=\footnotesize},
    title={$n=2$},
    width=0.4\textwidth,
    cycle list name=MR1LOF,
    ]
  \addplot coordinates {
(8, 1.212e+01) (13, 1.515e+01) (19, 1.521e+01) (26, 1.604e+01) (34, 1.750e+01) (43, 1.844e+01) (53, 2.002e+01) (64, 1.970e+01) (76, 2.172e+01) (89, 2.201e+01) (103, 2.244e+01) (118, 2.219e+01) (134, 2.401e+01) (151, 2.432e+01) (169, 2.386e+01) (188, 2.413e+01) (208, 2.613e+01) (229, 2.589e+01) (251, 2.600e+01) (274, 2.591e+01) (298, 2.609e+01) (323, 2.600e+01) (349, 2.808e+01) (376, 2.793e+01) (404, 2.800e+01) (433, 2.833e+01) (463, 2.806e+01) (494, 2.806e+01) (526, 2.795e+01) (559, 2.995e+01) (593, 3.000e+01) (628, 3.005e+01) (664, 2.996e+01) (701, 3.013e+01) (739, 3.004e+01) (778, 3.004e+01) (818, 3.000e+01) (859, 3.004e+01) (901, 2.997e+01) (944, 3.200e+01) (988, 3.203e+01) (1033, 3.203e+01) (1079, 3.203e+01) (1126, 3.197e+01) (1174, 3.197e+01) (1223, 3.200e+01) (1273, 3.203e+01) (1324, 3.198e+01) (1376, 3.200e+01) (1429, 3.198e+01) (1483, 3.202e+01) (1538, 3.402e+01) (1594, 3.398e+01) (1651, 3.398e+01) (1709, 3.414e+01) 
  };
  \addplot coordinates {
(8, 1.888e+01) (13, 2.115e+01) (19, 2.111e+01) (26, 2.073e+01) (34, 2.221e+01) (43, 2.253e+01) (53, 2.474e+01) (64, 2.320e+01) (76, 2.554e+01) (89, 2.531e+01) (103, 2.556e+01) (118, 2.470e+01) (134, 2.671e+01) (151, 2.684e+01) (169, 2.617e+01) (188, 2.624e+01) (208, 2.806e+01) (229, 2.790e+01) (251, 2.848e+01) (274, 2.770e+01) (298, 2.745e+01) (323, 2.762e+01) (349, 2.995e+01) (376, 2.963e+01) (404, 2.925e+01) (433, 2.979e+01) (463, 2.942e+01) (494, 2.922e+01) (526, 2.915e+01) (559, 3.155e+01) (593, 3.129e+01) (628, 3.116e+01) (664, 3.178e+01) (701, 3.105e+01) (739, 3.089e+01) (778, 3.089e+01) (818, 3.117e+01) (859, 3.093e+01) (901, 3.101e+01) (944, 3.316e+01) (988, 3.281e+01) (1033, 3.281e+01) (1079, 3.319e+01) (1126, 3.276e+01) (1174, 3.263e+01) (1223, 3.300e+01) (1273, 3.288e+01) (1324, 3.247e+01) (1376, 3.273e+01) (1429, 3.273e+01) (1483, 3.277e+01) (1538, 3.487e+01) (1594, 3.467e+01) (1651, 3.450e+01) (1709, 3.477e+01) 
 };
  \addplot coordinates {
(8, 4.125e+00) (13, 2.231e+00) (19, 5.737e+00) (26, 6.038e+00) (34, 5.853e+00) (43, 8.209e+00) (53, 8.019e+00) (64, 9.859e+00) (76, 9.882e+00) (89, 1.001e+01) (103, 1.020e+01) (118, 1.413e+01) (134, 1.201e+01) (151, 1.014e+01) (169, 9.947e+00) (188, 1.207e+01) (208, 1.407e+01) (229, 9.961e+00) (251, 1.200e+01) (274, 1.395e+01) (298, 1.405e+01) (323, 1.400e+01) (349, 1.404e+01) (376, 1.397e+01) (404, 1.600e+01) (433, 1.214e+01) (463, 1.403e+01) (494, 1.203e+01) (526, 1.398e+01) (559, 1.797e+01) (593, 1.400e+01) (628, 1.402e+01) (664, 1.797e+01) (701, 1.607e+01) (739, 1.803e+01) (778, 1.602e+01) (818, 1.800e+01) (859, 1.402e+01) (901, 1.598e+01) (944, 1.600e+01) (988, 1.602e+01) (1033, 1.401e+01) (1079, 1.802e+01) (1126, 1.399e+01) (1174, 1.599e+01) (1223, 1.600e+01) (1273, 1.601e+01) (1324, 1.399e+01) (1376, 1.800e+01) (1429, 1.199e+01) (1483, 1.801e+01) (1538, 1.801e+01) (1594, 1.599e+01) (1651, 1.599e+01) (1709, 1.607e+01) 
  };
  \addplot coordinates {
(8, 4.375e+00) (13, 4.538e+00) (19, 4.053e+00) (26, 9.115e+00) (34, 6.147e+00) (43, 1.147e+01) (53, 6.170e+00) (64, 8.297e+00) (76, 1.062e+01) (89, 1.053e+01) (103, 1.086e+01) (118, 1.057e+01) (134, 1.246e+01) (151, 1.495e+01) (169, 1.030e+01) (188, 1.474e+01) (208, 1.456e+01) (229, 1.227e+01) (251, 1.696e+01) (274, 1.233e+01) (298, 1.437e+01) (323, 1.439e+01) (349, 1.666e+01) (376, 1.222e+01) (404, 1.424e+01) (433, 1.233e+01) (463, 1.224e+01) (494, 1.224e+01) (526, 1.845e+01) (559, 1.222e+01) (593, 1.219e+01) (628, 1.428e+01) (664, 1.447e+01) (701, 1.634e+01) (739, 1.618e+01) (778, 1.215e+01) (818, 1.635e+01) (859, 1.624e+01) (901, 1.840e+01) (944, 1.421e+01) (988, 1.413e+01) (1033, 1.825e+01) (1079, 1.844e+01) (1126, 1.617e+01) (1174, 1.614e+01) (1223, 1.625e+01) (1273, 1.827e+01) (1324, 1.612e+01) (1376, 2.029e+01) (1429, 1.617e+01) (1483, 1.620e+01) (1538, 1.619e+01) (1594, 1.612e+01) (1651, 2.015e+01) (1709, 1.825e+01) 
  };
  \addplot coordinates {
(8, 2.125e+00) (13, 2.231e+00) (19, 1.947e+00) (26, 2.115e+00) (34, 2.265e+00) (43, 2.442e+00) (53, 2.396e+00) (64, 2.359e+00) (76, 2.461e+00) (89, 2.371e+00) (103, 2.456e+00) (118, 2.432e+00) (134, 2.515e+00) (151, 2.722e+00) (169, 2.704e+00) (188, 2.612e+00) (208, 2.774e+00) (229, 2.694e+00) (251, 2.729e+00) (274, 2.719e+00) (298, 2.883e+00) (323, 2.820e+00) (349, 2.805e+00) (376, 2.742e+00) (404, 2.780e+00) (433, 2.871e+00) (463, 2.922e+00) (494, 2.877e+00) (526, 2.880e+00) (559, 2.828e+00) (593, 2.895e+00) (628, 2.842e+00) (664, 2.983e+00) (701, 2.909e+00) (739, 2.946e+00) (778, 2.955e+00) (818, 3.004e+00) (859, 2.972e+00) (901, 3.000e+00) (944, 2.927e+00) (988, 2.997e+00) (1033, 2.990e+00) (1079, 2.956e+00) (1126, 3.026e+00) (1174, 2.986e+00) (1223, 3.013e+00) (1273, 3.002e+00) (1324, 2.961e+00) (1376, 2.982e+00) (1429, 3.015e+00) (1483, 2.978e+00) (1538, 3.054e+00) (1594, 3.014e+00) (1651, 3.041e+00) (1709, 3.090e+00) 
  };
  \addplot coordinates {
(8, 2.125e+00) (13, 2.231e+00) (19, 1.947e+00) (26, 2.192e+00) (34, 2.029e+00) (43, 2.302e+00) (53, 2.358e+00) (64, 2.297e+00) (76, 2.487e+00) (89, 2.551e+00) (103, 2.612e+00) (118, 2.415e+00) (134, 2.515e+00) (151, 2.682e+00) (169, 2.645e+00) (188, 2.495e+00) (208, 2.793e+00) (229, 2.738e+00) (251, 2.761e+00) (274, 2.741e+00) (298, 2.822e+00) (323, 2.771e+00) (349, 2.777e+00) (376, 2.779e+00) (404, 2.879e+00) (433, 2.926e+00) (463, 2.901e+00) (494, 2.872e+00) (526, 2.800e+00) (559, 2.860e+00) (593, 2.922e+00) (628, 2.938e+00) (664, 2.944e+00) (701, 2.889e+00) (739, 2.978e+00) (778, 2.927e+00) (818, 3.001e+00) (859, 2.986e+00) (901, 2.951e+00) (944, 2.942e+00) (988, 3.009e+00) (1033, 3.039e+00) (1079, 2.937e+00) (1126, 2.997e+00) (1174, 2.984e+00) (1223, 2.977e+00) (1273, 3.013e+00) (1324, 2.998e+00) (1376, 2.991e+00) (1429, 2.969e+00) (1483, 2.988e+00) (1538, 3.051e+00) (1594, 3.055e+00) (1651, 2.993e+00) (1709, 3.056e+00) 
  };
\node  at (axis description cs:0,0) [
    text width=31pt, anchor=north east, align=center,xshift=14pt,yshift=1.5pt
    ]
    {\scriptsize $|H_{2}^d|$};
\node  at (axis description cs:0,0) [
    text width=31pt, anchor=north east, align=center,xshift=14pt,yshift=-11.5pt] {\scriptsize $d$};
 \end{semilogxaxis}
\end{tikzpicture}
\;
\begin{tikzpicture}[baseline=(current axis.south)]
  \begin{semilogxaxis}[
    clip=false,
    scale only axis,
    xmin=13,xmax=42000, ymin=0, ymax=75,
    xtick={38,138,518,1958,7665,30688},
    x tick label style={align=center, font=\scriptsize},
    y tick label style={font=\scriptsize},
    xticklabels={{38\\[0.3em]3},{138\\[0.3em]6},{518\\[0.3em]11},{1\,958\\[0.3em]19},{7\,665\\[0.3em]32},{30\,688\\[0.3em]53}},
    xlabel style={align=center, font=\footnotesize},
    title style={font=\footnotesize},
    title={$n=3$},
    width=0.4\textwidth,
    cycle list name=MR1LOF
    ]
  \addplot coordinates {
(20, 1.605e+001) (38, 1.850e+001) (63, 2.002e+001) (96, 2.178e+001) (138, 2.401e+001) (190, 2.388e+001) (253, 2.611e+001) (328, 2.608e+001) (416, 2.820e+001) (518, 2.806e+001) (635, 3.014e+001) (768, 3.012e+001) (918, 3.218e+001) (1086, 3.209e+001) (1273, 3.203e+001) (1480, 3.202e+001) (1708, 3.416e+001) (1958, 3.400e+001) (2231, 3.400e+001) (2528, 3.601e+001) (2850, 3.600e+001) (3198, 3.600e+001) (3573, 3.602e+001) (3976, 3.599e+001) (4408, 3.801e+001) (4870, 3.799e+001) (5363, 3.801e+001) (5888, 3.800e+001) (6446, 3.800e+001) (7038, 4.001e+001) (7665, 3.999e+001) (8328, 4.000e+001) (9028, 4.000e+001) (9766, 4.000e+001) (10543, 4.000e+001) (11360, 4.200e+001) (12218, 4.200e+001) (13118, 4.200e+001) (14061, 4.200e+001) (15048, 4.200e+001) (16080, 4.200e+001) (17158, 4.200e+001) (18283, 4.400e+001) (19456, 4.400e+001) (20678, 4.400e+001) (21950, 4.401e+001) (23273, 4.400e+001) (24648, 4.400e+001) (26076, 4.400e+001) (27558, 4.400e+001) (29095, 4.400e+001) (30688, 4.600e+001) (32338, 4.600e+001) (34046, 4.600e+001) (35813, 4.600e+001) 
  };
  \addplot coordinates {
(20, 2.175e+001) (38, 2.297e+001) (63, 2.357e+001) (96, 2.472e+001) (138, 2.707e+001) (190, 2.596e+001) (253, 2.864e+001) (328, 2.775e+001) (416, 2.978e+001) (518, 2.923e+001) (635, 3.099e+001) (768, 3.093e+001) (918, 3.312e+001) (1086, 3.310e+001) (1273, 3.288e+001) (1480, 3.273e+001) (1708, 3.479e+001) (1958, 3.452e+001) (2231, 3.455e+001) (2528, 3.658e+001) (2850, 3.650e+001) (3198, 3.659e+001) (3573, 3.641e+001) (3976, 3.646e+001) (4408, 3.832e+001) (4870, 3.827e+001) (5363, 3.834e+001) (5888, 3.820e+001) (6446, 3.819e+001) (7038, 4.033e+001) (7665, 4.021e+001) (8328, 4.025e+001) (9028, 4.020e+001) (9766, 4.020e+001) (10543, 4.017e+001) (11360, 4.218e+001) (12218, 4.221e+001) (13118, 4.215e+001) (14061, 4.221e+001) (15048, 4.212e+001) (16080, 4.213e+001) (17158, 4.213e+001) (18283, 4.413e+001) (19456, 4.413e+001) (20678, 4.413e+001) (21950, 4.411e+001) (23273, 4.410e+001) (24648, 4.410e+001) (26076, 4.409e+001) (27558, 4.410e+001) (29095, 4.407e+001) (30688, 4.609e+001) (32338, 4.611e+001) (34046, 4.610e+001) (35813, 4.609e+001) 
 };
  \addplot coordinates {
(20, 6.050e+000) (38, 6.184e+000) (63, 6.016e+000) (96, 7.927e+000) (138, 1.201e+001) (190, 9.953e+000) (253, 1.004e+001) (328, 1.204e+001) (416, 1.410e+001) (518, 1.603e+001) (635, 1.206e+001) (768, 1.205e+001) (918, 1.810e+001) (1086, 1.605e+001) (1273, 1.601e+001) (1480, 1.401e+001) (1708, 1.407e+001) (1958, 1.600e+001) (2231, 1.600e+001) (2528, 1.801e+001) (2850, 1.600e+001) (3198, 1.800e+001) (3573, 1.801e+001) (3976, 1.800e+001) (4408, 2.000e+001) (4870, 1.800e+001) (5363, 2.000e+001) (5888, 1.800e+001) (6446, 1.800e+001) (7038, 2.001e+001) (7665, 2.000e+001) (8328, 2.000e+001) (9028, 2.200e+001) (9766, 2.000e+001) (10543, 2.200e+001) (11360, 2.400e+001) (12218, 2.400e+001) (13118, 2.000e+001) (14061, 2.200e+001) (15048, 2.200e+001) (16080, 2.200e+001) (17158, 1.800e+001) (18283, 2.400e+001) (19456, 2.200e+001) (20678, 2.600e+001) (21950, 2.201e+001) (23273, 2.200e+001) (24648, 2.400e+001) (26076, 2.400e+001) (27558, 2.200e+001) (29095, 2.200e+001) (30688, 2.200e+001) (32338, 2.200e+001) (34046, 2.400e+001) (35813, 2.600e+001) 
  };
  \addplot coordinates {
(20, 9.050e+000) (38, 9.079e+000) (63, 6.238e+000) (96, 6.031e+000) (138, 1.266e+001) (190, 1.024e+001) (253, 1.262e+001) (328, 1.445e+001) (416, 1.448e+001) (518, 1.852e+001) (635, 1.624e+001) (768, 1.216e+001) (918, 1.629e+001) (1086, 1.427e+001) (1273, 1.827e+001) (1480, 1.414e+001) (1708, 1.826e+001) (1958, 1.811e+001) (2231, 1.612e+001) (2528, 1.611e+001) (2850, 1.610e+001) (3198, 2.019e+001) (3573, 2.013e+001) (3976, 2.016e+001) (4408, 2.007e+001) (4870, 1.805e+001) (5363, 2.009e+001) (5888, 2.004e+001) (6446, 1.603e+001) (7038, 2.411e+001) (7665, 2.206e+001) (8328, 2.005e+001) (9028, 2.005e+001) (9766, 1.803e+001) (10543, 2.004e+001) (11360, 2.003e+001) (12218, 2.205e+001) (13118, 2.003e+001) (14061, 2.005e+001) (15048, 2.002e+001) (16080, 2.003e+001) (17158, 2.203e+001) (18283, 2.203e+001) (19456, 2.203e+001) (20678, 2.404e+001) (21950, 2.203e+001) (23273, 2.202e+001) (24648, 2.203e+001) (26076, 2.402e+001) (27558, 2.002e+001) (29095, 2.001e+001) (30688, 2.403e+001) (32338, 2.403e+001) (34046, 2.403e+001) (35813, 2.202e+001) 
  };
  \addplot coordinates {
(20, 2.050e+000) (38, 2.237e+000) (63, 2.175e+000) (96, 2.177e+000) (138, 2.529e+000) (190, 2.574e+000) (253, 2.502e+000) (328, 2.576e+000) (416, 2.632e+000) (518, 2.693e+000) (635, 2.770e+000) (768, 2.809e+000) (918, 2.794e+000) (1086, 2.872e+000) (1273, 2.907e+000) (1480, 2.898e+000) (1708, 3.003e+000) (1958, 3.039e+000) (2231, 3.003e+000) (2528, 2.960e+000) (2850, 2.976e+000) (3198, 3.090e+000) (3573, 2.990e+000) (3976, 3.055e+000) (4408, 3.042e+000) (4870, 3.022e+000) (5363, 3.043e+000) (5888, 3.074e+000) (6446, 3.097e+000) (7038, 3.083e+000) (7665, 3.134e+000) (8328, 3.113e+000) (9028, 3.146e+000) (9766, 3.115e+000) (10543, 3.152e+000) (11360, 3.144e+000) (12218, 3.144e+000) (13118, 3.137e+000) (14061, 3.137e+000) (15048, 3.155e+000) (16080, 3.160e+000) (17158, 3.153e+000) (18283, 3.159e+000) (19456, 3.179e+000) (20678, 3.182e+000) (21950, 3.197e+000) (23273, 3.175e+000) (24648, 3.190e+000) (26076, 3.194e+000) (27558, 3.203e+000) (29095, 3.183e+000) (30688, 3.205e+000) (32338, 3.185e+000) (34046, 3.190e+000) (35813, 3.206e+000) 
  };
  \addplot coordinates {
(20, 2.050e+000) (38, 2.132e+000) (63, 2.175e+000) (96, 2.281e+000) (138, 2.283e+000) (190, 2.426e+000) (253, 2.573e+000) (328, 2.619e+000) (416, 2.675e+000) (518, 2.654e+000) (635, 2.874e+000) (768, 2.759e+000) (918, 2.920e+000) (1086, 2.844e+000) (1273, 2.892e+000) (1480, 2.893e+000) (1708, 2.978e+000) (1958, 3.001e+000) (2231, 2.987e+000) (2528, 2.983e+000) (2850, 3.007e+000) (3198, 3.082e+000) (3573, 3.035e+000) (3976, 3.039e+000) (4408, 3.080e+000) (4870, 3.030e+000) (5363, 3.080e+000) (5888, 3.072e+000) (6446, 3.079e+000) (7038, 3.078e+000) (7665, 3.154e+000) (8328, 3.122e+000) (9028, 3.099e+000) (9766, 3.136e+000) (10543, 3.138e+000) (11360, 3.146e+000) (12218, 3.156e+000) (13118, 3.163e+000) (14061, 3.163e+000) (15048, 3.156e+000) (16080, 3.154e+000) (17158, 3.179e+000) (18283, 3.186e+000) (19456, 3.186e+000) (20678, 3.196e+000) (21950, 3.163e+000) (23273, 3.193e+000) (24648, 3.187e+000) (26076, 3.179e+000) (27558, 3.194e+000) (29095, 3.205e+000) (30688, 3.192e+000) (32338, 3.209e+000) (34046, 3.183e+000) (35813, 3.194e+000) 
  };
\node  at (axis description cs:0,0) [
    text width=31pt, anchor=north east, align=center,xshift=14pt,yshift=1.5pt] {\scriptsize $|H_{3}^d|$};
\node  at (axis description cs:0,0) [
    text width=31pt, anchor=north east, align=center,xshift=14pt,yshift=-11.5pt] {\scriptsize $d$};
 \end{semilogxaxis}
\end{tikzpicture}
\\[1em]
\begin{tikzpicture}[baseline=(current axis.south)]

  \begin{semilogxaxis}[
    clip=false,
    scale only axis,
    xmin=36,xmax=760000, ymin=0, ymax=75,
    xtick={104,501,2586,13852,79593,481029},
    x tick label style={align=center, font=\scriptsize},
    y tick label style={font=\scriptsize},
    xticklabels={{104\\[0.3em]3},{501\\[0.3em]6},{2\,586\\[0.3em]11},{13\,852\\[0.3em]19},{79\,593\\[0.3em]32},{481\,029\\[0.3em]53}},
    xlabel style={align=center, text width=7cm,
    font=\footnotesize},
    title style={font=\footnotesize},
    title={{$n=4$}},
    width=0.4\textwidth,
    cycle list name=MR1LOF
    ]
  \addplot coordinates {
(48, 2.002e+001) (104, 2.222e+001) (192, 2.388e+001) (321, 2.592e+001) (501, 2.817e+001) (743, 3.000e+001) (1059, 3.215e+001) (1462, 3.202e+001) (1966, 3.398e+001) (2586, 3.599e+001) (3338, 3.601e+001) (4239, 3.810e+001) (5307, 3.799e+001) (6561, 3.799e+001) (8021, 4.004e+001) (9708, 4.000e+001) (11644, 4.200e+001) (13852, 4.204e+001) (16356, 4.200e+001) (19181, 4.401e+001) (22353, 4.400e+001) (25899, 4.400e+001) (29847, 4.400e+001) (34226, 4.601e+001) (39066, 4.600e+001) (44398, 4.600e+001) (50254, 4.800e+001) (56667, 4.800e+001) (63671, 4.800e+001) (71301, 4.800e+001) (79593, 4.800e+001) (88584, 5.000e+001) (98312, 5.000e+001) (108816, 5.000e+001) (120136, 5.000e+001) (132313, 5.000e+001) (145389, 5.200e+001) (159407, 5.200e+001) (174411, 5.200e+001) (190446, 5.200e+001) (207558, 5.200e+001) (225794, 5.400e+001) (245202, 5.400e+001) (265831, 5.400e+001) (287731, 5.400e+001) (310953, 5.400e+001) (335549, 5.400e+001) (361572, 5.400e+001) (389076, 5.600e+001) (418116, 5.600e+001) (448748, 5.600e+001) (481029, 5.600e+001) (515017, 5.600e+001) (550771, 5.600e+001) (588351, 5.600e+001) 
  };
  \addplot coordinates {
(48, 2.406e+001) (104, 2.532e+001) (192, 2.606e+001) (321, 2.724e+001) (501, 2.919e+001) (743, 3.099e+001) (1059, 3.302e+001) (1462, 3.277e+001) (1966, 3.467e+001) (2586, 3.655e+001) (3338, 3.639e+001) (4239, 3.846e+001) (5307, 3.828e+001) (6561, 3.823e+001) (8021, 4.020e+001) (9708, 4.011e+001) (11644, 4.219e+001) (13852, 4.213e+001) (16356, 4.214e+001) (19181, 4.418e+001) (22353, 4.412e+001) (25899, 4.408e+001) (29847, 4.409e+001) (34226, 4.608e+001) (39066, 4.606e+001) (44398, 4.605e+001) (50254, 4.806e+001) (56667, 4.806e+001) (63671, 4.806e+001) (71301, 4.805e+001) (79593, 4.804e+001) (88584, 5.004e+001) (98312, 5.004e+001) (108816, 5.004e+001) (120136, 5.004e+001) (132313, 5.003e+001) (145389, 5.203e+001) (159407, 5.203e+001) (174411, 5.202e+001) (190446, 5.202e+001) (207558, 5.202e+001) (225794, 5.402e+001) (245202, 5.402e+001) (265831, 5.402e+001) (287731, 5.402e+001) (310953, 5.402e+001) (335549, 5.402e+001) (361572, 5.401e+001) (389076, 5.601e+001) (418116, 5.601e+001) (448748, 5.601e+001) (481029, 5.601e+001) (515017, 5.601e+001) (550771, 5.601e+001) (588351, 5.601e+001) 
 };
  \addplot coordinates {
(48, 2.021e+000) (104, 1.213e+001) (192, 1.194e+001) (321, 1.396e+001) (501, 1.610e+001) (743, 1.400e+001) (1059, 1.407e+001) (1462, 1.601e+001) (1966, 1.599e+001) (2586, 1.599e+001) (3338, 2.201e+001) (4239, 1.805e+001) (5307, 2.000e+001) (6561, 1.600e+001) (8021, 2.202e+001) (9708, 2.000e+001) (11644, 2.000e+001) (13852, 2.202e+001) (16356, 2.400e+001) (19181, 2.000e+001) (22353, 2.200e+001) (25899, 2.600e+001) (29847, 2.000e+001) (34226, 2.201e+001) (39066, 2.400e+001) (44398, 2.200e+001) (50254, 2.200e+001) (56667, 2.800e+001) (63671, 2.600e+001) (71301, 2.600e+001) (79593, 2.800e+001) (88584, 3.000e+001) (98312, 2.600e+001) (108816, 2.600e+001) (120136, 3.000e+001) (132313, 2.400e+001) (145389, 2.600e+001) (159407, 2.800e+001) (174411, 2.600e+001) (190446, 2.800e+001) (207558, 2.800e+001) (225794, 2.800e+001) (245202, 2.600e+001) (265831, 2.800e+001) (287731, 3.000e+001) (310953, 3.000e+001) (335549, 3.000e+001) (361572, 2.800e+001) (389076, 3.200e+001) (418116, 3.000e+001) (448748, 2.800e+001) (481029, 2.800e+001) (515017, 3.200e+001) (550771, 3.000e+001) (588351, 2.600e+001) 
  };
  \addplot coordinates {
(48, 6.229e+000) (104, 8.529e+000) (192, 1.029e+001) (321, 1.215e+001) (501, 1.429e+001) (743, 1.833e+001) (1059, 1.621e+001) (1462, 1.414e+001) (1966, 1.617e+001) (2586, 1.812e+001) (3338, 1.606e+001) (4239, 2.014e+001) (5307, 2.209e+001) (6561, 2.005e+001) (8021, 2.004e+001) (9708, 2.002e+001) (11644, 2.003e+001) (13852, 2.004e+001) (16356, 1.802e+001) (19181, 2.004e+001) (22353, 2.404e+001) (25899, 2.202e+001) (29847, 2.001e+001) (34226, 2.402e+001) (39066, 3.002e+001) (44398, 2.401e+001) (50254, 2.401e+001) (56667, 2.802e+001) (63671, 2.402e+001) (71301, 2.601e+001) (79593, 2.801e+001) (88584, 2.401e+001) (98312, 2.601e+001) (108816, 2.601e+001) (120136, 2.601e+001) (132313, 2.801e+001) (145389, 2.601e+001) (159407, 3.001e+001) (174411, 2.601e+001) (190446, 2.801e+001) (207558, 2.801e+001) (225794, 2.800e+001) (245202, 2.801e+001) (265831, 3.001e+001) (287731, 2.600e+001) (310953, 2.800e+001) (335549, 2.800e+001) (361572, 3.000e+001) (389076, 2.800e+001) (418116, 3.000e+001) (448748, 3.200e+001) (481029, 2.800e+001) (515017, 3.200e+001) (550771, 2.800e+001) (588351, 3.200e+001) 
  };
  \addplot coordinates {
(48, 2.021e+000) (104, 2.087e+000) (192, 2.214e+000) (321, 2.464e+000) (501, 2.677e+000) (743, 2.693e+000) (1059, 2.745e+000) (1462, 2.839e+000) (1966, 2.840e+000) (2586, 2.878e+000) (3338, 2.886e+000) (4239, 2.965e+000) (5307, 2.988e+000) (6561, 3.030e+000) (8021, 3.040e+000) (9708, 3.098e+000) (11644, 3.057e+000) (13852, 3.058e+000) (16356, 3.081e+000) (19181, 3.105e+000) (22353, 3.142e+000) (25899, 3.113e+000) (29847, 3.159e+000) (34226, 3.123e+000) (39066, 3.181e+000) (44398, 3.172e+000) (50254, 3.187e+000) (56667, 3.187e+000) (63671, 3.196e+000) (71301, 3.180e+000) (79593, 3.201e+000) (88584, 3.197e+000) (98312, 3.203e+000) (108816, 3.210e+000) (120136, 3.220e+000) (132313, 3.225e+000) (145389, 3.212e+000) (159407, 3.230e+000) (174411, 3.219e+000) (190446, 3.226e+000) (207558, 3.228e+000) (225794, 3.212e+000) (245202, 3.238e+000) (265831, 3.240e+000) (287731, 3.247e+000) (310953, 3.235e+000) (335549, 3.239e+000) (361572, 3.239e+000) (389076, 3.232e+000) (418116, 3.240e+000) (448748, 3.256e+000) (481029, 3.248e+000) (515017, 3.253e+000) (550771, 3.250e+000) (588351, 3.254e+000) 
  };
  \addplot coordinates {
(48, 2.229e+000) (104, 2.087e+000) (192, 2.474e+000) (321, 2.321e+000) (501, 2.573e+000) (743, 2.629e+000) (1059, 2.811e+000) (1462, 2.886e+000) (1966, 2.922e+000) (2586, 2.930e+000) (3338, 2.934e+000) (4239, 2.920e+000) (5307, 3.020e+000) (6561, 3.040e+000) (8021, 3.003e+000) (9708, 3.090e+000) (11644, 3.074e+000) (13852, 3.050e+000) (16356, 3.110e+000) (19181, 3.116e+000) (22353, 3.139e+000) (25899, 3.131e+000) (29847, 3.152e+000) (34226, 3.143e+000) (39066, 3.107e+000) (44398, 3.169e+000) (50254, 3.150e+000) (56667, 3.184e+000) (63671, 3.167e+000) (71301, 3.186e+000) (79593, 3.191e+000) (88584, 3.191e+000) (98312, 3.202e+000) (108816, 3.219e+000) (120136, 3.199e+000) (132313, 3.228e+000) (145389, 3.224e+000) (159407, 3.224e+000) (174411, 3.212e+000) (190446, 3.234e+000) (207558, 3.228e+000) (225794, 3.219e+000) (245202, 3.237e+000) (265831, 3.238e+000) (287731, 3.242e+000) (310953, 3.245e+000) (335549, 3.247e+000) (361572, 3.230e+000) (389076, 3.244e+000) (418116, 3.252e+000) (448748, 3.248e+000) (481029, 3.252e+000) (515017, 3.246e+000) (550771, 3.252e+000) (588351, 3.251e+000) 
  };
\node  at (axis description cs:0,0) [
    text width=31pt, anchor=north east, align=center,xshift=14pt,yshift=1.5pt] {\scriptsize $|H_{4}^d|$};
\node  at (axis description cs:0,0) [
    text width=31pt, anchor=north east, align=center,xshift=14pt,yshift=-11.5pt] {\scriptsize $d$};
 \end{semilogxaxis}
\end{tikzpicture}
\;
\begin{tikzpicture}[baseline=(current axis.south)]
\begin{semilogxaxis}[
    clip=false,
    scale only axis,
    xmin=84,xmax=10000000,ymin=0,ymax=75,
    xtick={272, 1683, 11584, 85520, 704073, 6300482},
    x tick label style={align=center, font=\scriptsize},
    y tick label style={font=\scriptsize},
	xticklabels={{272\\[0.3em]3},{1\,683\\[0.3em]6},{11\,584\\[0.3em]11},{85\,520\\[0.3em]19},{704\,073\\[0.3em]32},{6\,300\,482\\[0.3em]53}},
    xlabel style={align=center,
    font=\footnotesize},
    title style={font=\footnotesize},
    title={$n=5$},
    width=0.4\textwidth,
    cycle list name=MR1LOF
    ]
  \addplot coordinates {
(112, 2.181e+01) (272, 2.610e+01) (552, 2.995e+01) (1002, 3.197e+01) (1683, 3.404e+01) (2668, 3.607e+01) (4043, 3.600e+01) (5908, 3.801e+01) (8378, 4.000e+01) (11584, 4.200e+01) (15674, 4.201e+01) (20814, 4.400e+01) (27189, 4.400e+01) (35004, 4.600e+01) (44485, 4.600e+01) (55880, 4.800e+01) (69460, 4.800e+01) (85520, 5.000e+01) (104380, 5.000e+01) (126386, 5.000e+01) (151911, 5.200e+01) (181356, 5.200e+01) (215151, 5.200e+01) (253756, 5.400e+01) (297662, 5.400e+01) (347392, 5.400e+01) (403502, 5.600e+01) (466582, 5.600e+01) (537257, 5.600e+01) (616188, 5.800e+01) (704073, 5.800e+01) (801648, 5.800e+01) (909688, 5.800e+01) (1029008, 6.000e+01) (1160464, 6.000e+01) (1304954, 6.000e+01) (1463419, 6.000e+01) (1636844, 6.200e+01) (1826259, 6.200e+01) (2032740, 6.200e+01) (2257410, 6.200e+01) (2501440, 6.200e+01) (2766050, 6.400e+01) (3052510, 6.400e+01) (3362141, 6.400e+01) (3696316, 6.400e+01) (4056461, 6.400e+01) (4444056, 6.600e+01) (4860636, 6.600e+01) (5307792, 6.600e+01) (5787172, 6.600e+01) (6300482, 6.600e+01) (6849487, 6.600e+01) (7436012, 6.800e+01) (8061943, 6.800e+01) 
  };
  \addplot coordinates {
(112, 2.404e+01) (272, 2.790e+01) (552, 3.152e+01) (1002, 3.289e+01) (1683, 3.479e+01) (2668, 3.655e+01) (4043, 3.632e+01) (5908, 3.825e+01) (8378, 4.027e+01) (11584, 4.218e+01) (15674, 4.219e+01) (20814, 4.412e+01) (27189, 4.408e+01) (35004, 4.608e+01) (44485, 4.605e+01) (55880, 4.805e+01) (69460, 4.807e+01) (85520, 5.004e+01) (104380, 5.004e+01) (126386, 5.003e+01) (151911, 5.203e+01) (181356, 5.203e+01) (215151, 5.202e+01) (253756, 5.402e+01) (297662, 5.402e+01) (347392, 5.401e+01) (403502, 5.601e+01) (466582, 5.601e+01) (537257, 5.601e+01) (616188, 5.801e+01) (704073, 5.801e+01) (801648, 5.801e+01) (909688, 5.801e+01) (1029008, 6.001e+01) (1160464, 6.000e+01) (1304954, 6.001e+01) (1463419, 6.001e+01) (1636844, 6.200e+01) (1826259, 6.200e+01) (2032740, 6.200e+01) (2257410, 6.200e+01) (2501440, 6.200e+01) (2766050, 6.400e+01) (3052510, 6.400e+01) (3362141, 6.400e+01) (3696316, 6.400e+01) (4056461, 6.400e+01) (4444056, 6.600e+01) (4860636, 6.600e+01) (5307792, 6.600e+01) (5787172, 6.600e+01) (6300482, 6.600e+01) (6849487, 6.600e+01) (7436012, 6.800e+01) (8061943, 6.800e+01) 
 };
  \addplot coordinates {
(112, 3.973e+00) (272, 1.004e+01) (552, 1.398e+01) (1002, 1.399e+01) (1683, 1.802e+01) (2668, 1.603e+01) (4043, 1.800e+01) (5908, 1.801e+01) (8378, 2.000e+01) (11584, 2.400e+01) (15674, 2.401e+01) (20814, 2.400e+01) (27189, 2.600e+01) (35004, 2.200e+01) (44485, 2.200e+01) (55880, 2.400e+01) (69460, 2.600e+01) (85520, 2.600e+01) (104380, 2.800e+01) (126386, 2.600e+01) (151911, 2.800e+01) (181356, 2.800e+01) (215151, 2.600e+01) (253756, 2.600e+01) (297662, 2.800e+01) (347392, 2.800e+01) (403502, 2.800e+01) (466582, 2.800e+01) (537257, 2.800e+01) (616188, 2.800e+01) (704073, 2.800e+01) (801648, 3.200e+01) (909688, 2.800e+01) (1029008, 3.000e+01) (1160464, 3.000e+01) (1304954, 3.200e+01) (1463419, 3.200e+01) (1636844, 3.600e+01) (1826259, 3.200e+01) (2032740, 3.000e+01) (2257410, 3.200e+01) (2501440, 3.200e+01) (2766050, 3.400e+01) (3052510, 3.400e+01) (3362141, 3.600e+01) (3696316, 3.200e+01) (4056461, 3.200e+01) (4444056, 3.200e+01) (4860636, 3.400e+01) (5307792, 3.400e+01) (5787172, 3.400e+01) (6300482, 3.600e+01) (6849487, 3.400e+01) (7436012, 3.400e+01) (8061943, 3.200e+01) 
  };
  \addplot coordinates {
(112, 6.045e+00) (272, 1.458e+01) (552, 1.219e+01) (1002, 1.414e+01) (1683, 1.823e+01) (2668, 2.020e+01) (4043, 2.008e+01) (5908, 2.007e+01) (8378, 1.806e+01) (11584, 2.003e+01) (15674, 2.407e+01) (20814, 2.203e+01) (27189, 2.402e+01) (35004, 2.603e+01) (44485, 2.401e+01) (55880, 2.601e+01) (69460, 2.602e+01) (85520, 2.801e+01) (104380, 2.601e+01) (126386, 2.601e+01) (151911, 2.601e+01) (181356, 3.001e+01) (215151, 2.601e+01) (253756, 2.400e+01) (297662, 3.001e+01) (347392, 2.800e+01) (403502, 3.200e+01) (466582, 2.800e+01) (537257, 3.200e+01) (616188, 2.800e+01) (704073, 3.200e+01) (801648, 3.200e+01) (909688, 3.000e+01) (1029008, 3.200e+01) (1160464, 3.200e+01) (1304954, 3.200e+01) (1463419, 3.200e+01) (1636844, 3.000e+01) (1826259, 3.600e+01) (2032740, 3.200e+01) (2257410, 3.200e+01) (2501440, 3.400e+01) (2766050, 3.200e+01) (3052510, 3.200e+01) (3362141, 3.600e+01) (3696316, 3.200e+01) (4056461, 3.400e+01) (4444056, 3.600e+01) (4860636, 3.400e+01) (5307792, 3.400e+01) (5787172, 3.400e+01) (6300482, 3.400e+01) (6849487, 3.400e+01) (7436012, 3.600e+01) (8061943, 3.600e+01) 
  };
  \addplot coordinates {
(112, 2.152e+00) (272, 2.261e+00) (552, 2.524e+00) (1002, 2.698e+00) (1683, 2.712e+00) (2668, 2.769e+00) (4043, 2.901e+00) (5908, 2.820e+00) (8378, 2.987e+00) (11584, 3.034e+00) (15674, 3.029e+00) (20814, 3.069e+00) (27189, 3.055e+00) (35004, 3.069e+00) (44485, 3.114e+00) (55880, 3.158e+00) (69460, 3.113e+00) (85520, 3.141e+00) (104380, 3.151e+00) (126386, 3.123e+00) (151911, 3.186e+00) (181356, 3.191e+00) (215151, 3.215e+00) (253756, 3.200e+00) (297662, 3.198e+00) (347392, 3.209e+00) (403502, 3.205e+00) (466582, 3.216e+00) (537257, 3.213e+00) (616188, 3.226e+00) (704073, 3.245e+00) (801648, 3.248e+00) (909688, 3.251e+00) (1029008, 3.242e+00) (1160464, 3.243e+00) (1304954, 3.249e+00) (1463419, 3.251e+00) (1636844, 3.263e+00) (1826259, 3.258e+00) (2032740, 3.261e+00) (2257410, 3.262e+00) (2501440, 3.262e+00) (2766050, 3.260e+00) (3052510, 3.262e+00) (3362141, 3.266e+00) (3696316, 3.271e+00) (4056461, 3.271e+00) (4444056, 3.272e+00) (4860636, 3.272e+00) (5307792, 3.273e+00) (5787172, 3.275e+00) (6300482, 3.275e+00) (6849487, 3.276e+00) (7436012, 3.275e+00) (8061943, 3.276e+00) 
  };
  \addplot coordinates {
(112, 2.152e+00) (272, 2.232e+00) (552, 2.418e+00) (1002, 2.568e+00) (1683, 2.749e+00) (2668, 2.627e+00) (4043, 2.868e+00) (5908, 2.925e+00) (8378, 2.979e+00) (11584, 3.041e+00) (15674, 3.004e+00) (20814, 3.016e+00) (27189, 3.046e+00) (35004, 3.110e+00) (44485, 3.098e+00) (55880, 3.136e+00) (69460, 3.133e+00) (85520, 3.158e+00) (104380, 3.154e+00) (126386, 3.165e+00) (151911, 3.177e+00) (181356, 3.176e+00) (215151, 3.199e+00) (253756, 3.202e+00) (297662, 3.204e+00) (347392, 3.210e+00) (403502, 3.219e+00) (466582, 3.227e+00) (537257, 3.218e+00) (616188, 3.235e+00) (704073, 3.223e+00) (801648, 3.235e+00) (909688, 3.235e+00) (1029008, 3.242e+00) (1160464, 3.247e+00) (1304954, 3.255e+00) (1463419, 3.243e+00) (1636844, 3.260e+00) (1826259, 3.256e+00) (2032740, 3.257e+00) (2257410, 3.254e+00) (2501440, 3.265e+00) (2766050, 3.273e+00) (3052510, 3.263e+00) (3362141, 3.261e+00) (3696316, 3.268e+00) (4056461, 3.269e+00) (4444056, 3.271e+00) (4860636, 3.272e+00) (5307792, 3.267e+00) (5787172, 3.271e+00) (6300482, 3.272e+00) (6849487, 3.278e+00) (7436012, 3.274e+00) (8061943, 3.277e+00) 
  };
\node  at (axis description cs:0,0) [
    text width=31pt, anchor=north east, align=center,xshift=14pt,yshift=1.5pt] {\scriptsize $|H_{5}^d|$};
\node  at (axis description cs:0,0) [
    text width=31pt, anchor=north east, align=center,xshift=14pt,yshift=-11.5pt] {\scriptsize $d$};
 \end{semilogxaxis}
\end{tikzpicture}
\\[1em]
\ref{leghypcrossnfix}
\caption{Oversampling factors of reconstructing multiple rank\mbox{-}1 lattices for
hyperbolic cross frequency sets $H_n^d$, $n=2,3,4,5$ fixed.}\label{fig:numtests_hypcross_nfix}
\end{figure}
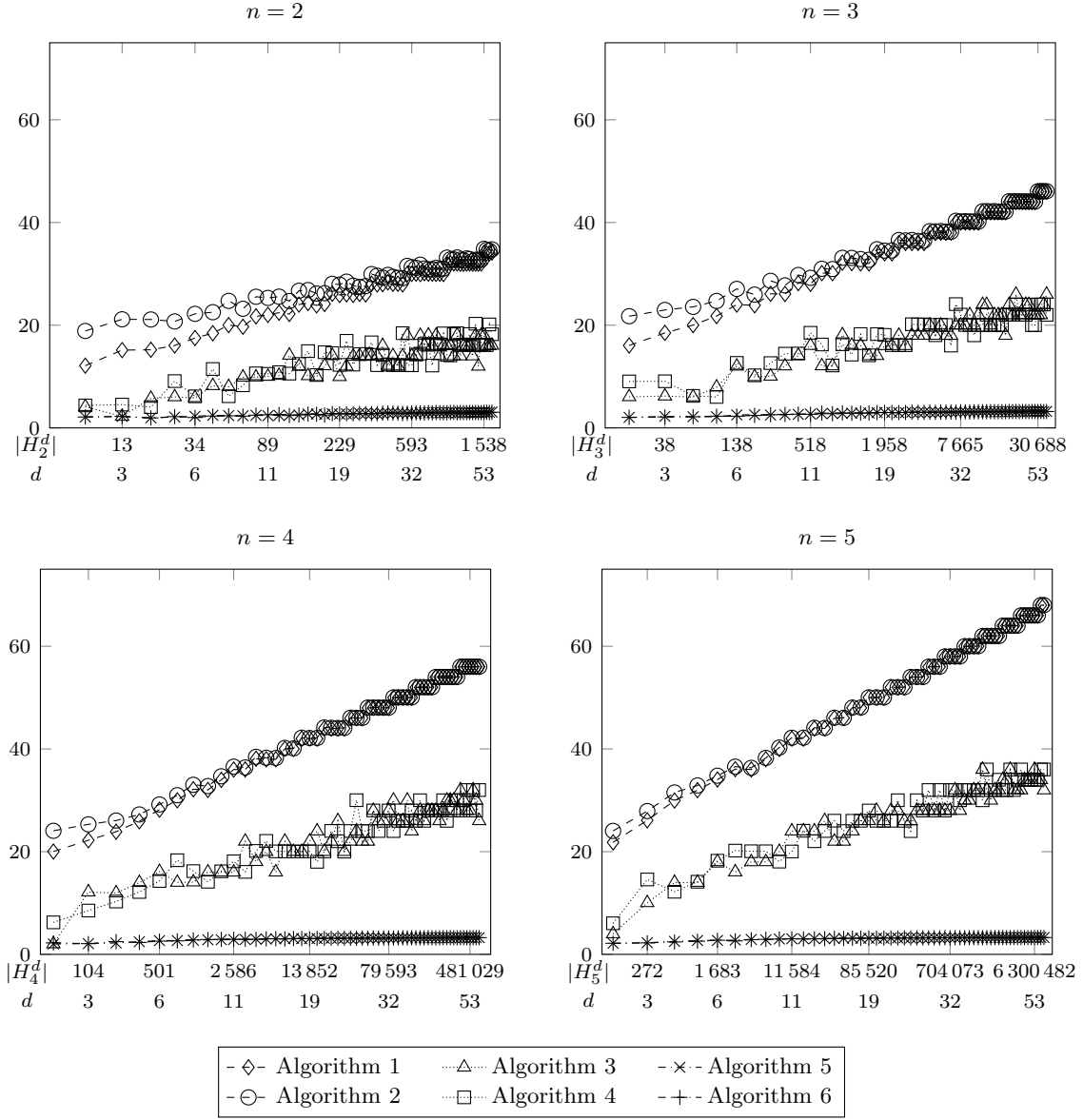

\begin{figure}[tb]
\centering
\begin{tikzpicture}
\begin{loglogaxis}[
    clip=false,
    scale only axis,
    xmin=100,xmax=20000, ymin=0.9, ymax=10,
    title={Condition numbers},
    xtick={272,1683,11584},
    ytick={1,2,5,7,9},
    x tick label style={align=center,font=\scriptsize},
    y tick label style={font=\scriptsize},
    xticklabels={{272\\[0.3em]3},{1\,683\\[0.3em]6},{11\,584\\[0.3em]11}},
    yticklabels={1,2,5,7,9},
    width=0.35\textwidth,
    height=0.25\textwidth,
    every axis legend/.append style={nodes={right}},
    legend entries={Algorithm \ref{alg:construct_mr1l_uniform}, Algorithm \ref{alg:construct_mr1l_uniform_different_primes},Algorithm \ref{alg:construct_mr1l_I}, Algorithm \ref{alg:construct_mr1l_I_distinct_primes}, Algorithm \ref{alg:construct_mr1l_3}, Algorithm \ref{alg:construct_mr1l_4}
    },
    legend pos=outer north east,
    legend style={font=\footnotesize, /tikz/every even column/.append style={column sep=0.5cm}},
    title style={font=\footnotesize},
    cycle list name=MR1LOF
    ]
  \addplot coordinates {
(112, 1.425e+00) (272, 1.400e+00) (552, 1.478e+00) (1002, 1.466e+00) (1683, 1.560e+00) (2668, 1.418e+00) (4043, 1.426e+00) (5908, 1.428e+00) (8378, 1.401e+00) (11584, 1.391e+00) (15674, 1.398e+00) 
  };
  \addplot coordinates {
(112, 1.398e+00) (272, 1.442e+00) (552, 1.427e+00) (1002, 1.471e+00) (1683, 1.426e+00) (2668, 1.429e+00) (4043, 1.468e+00) (5908, 1.443e+00) (8378, 1.403e+00) (11584, 1.424e+00) (15674, 1.418e+00) 
 };
  \addplot coordinates {
(112, 1.732e+00) (272, 1.869e+00) (552, 1.943e+00) (1002, 1.714e+00) (1683, 1.709e+00) (2668, 1.766e+00) (4043, 1.745e+00) (5908, 1.710e+00) (8378, 1.691e+00) (11584, 1.675e+00) (15674, 1.619e+00) 
  };
  \addplot coordinates {
(112, 1.822e+00) (272, 1.716e+00) (552, 1.831e+00) (1002, 1.860e+00) (1683, 1.745e+00) (2668, 1.685e+00) (4043, 1.658e+00) (5908, 1.707e+00) (8378, 1.720e+00) (11584, 1.716e+00) (15674, 1.580e+00) 
  };
  \addplot coordinates {
(112, 6.348e+00) (272, 8.385e+00) (552, 6.202e+00) (1002, 8.842e+00) (1683, 8.701e+00) (2668, 6.965e+00) (4043, 7.521e+00) (5908, 8.023e+00) (8378, 8.001e+00) (11584, 7.254e+00) (15674, 7.191e+00) 
  };
  \addplot coordinates {
(112, 6.348e+00) (272, 8.697e+00) (552, 7.645e+00) (1002, 7.751e+00) (1683, 7.060e+00) (2668, 8.960e+00) (4043, 7.927e+00) (5908, 7.672e+00) (8378, 7.805e+00) (11584, 8.488e+00) (15674, 7.058e+00) 
  };
\node  at (axis description cs:0,0) [
    text width=31pt, anchor=north east, align=center,xshift=14pt,yshift=1.5pt
    ]
    {\scriptsize $|H_{5}^d|$};
\node  at (axis description cs:0,0) [
    text width=31pt, anchor=north east, align=center,xshift=14pt,yshift=-11.5pt] {\scriptsize $d$};
 \end{loglogaxis}
\end{tikzpicture}
\caption{Condition numbers of Fourier matrices $\zb A(\Lambda(\zb z_1,M_1,\ldots,\zb z_s,M_s),H_5^d)$, $d=2,\ldots, 12$, for different reconstructing multiple rank\mbox{-}1 lattices $\Lambda(\zb z_1,M_1,\ldots,\zb z_s,M_s)$.}
\label{fig:numtests_hypcross_n5_cond}
\end{figure}
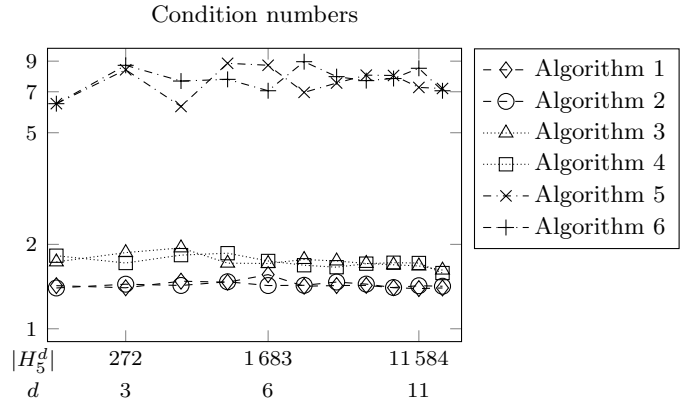

We consider dyadic hyperbolic crosses 
$$
I=H_n^d:=\bigcup_{\|\zb j\|_1=n}G_{j_1}\times\ldots\times G_{j_d},\qquad G_{j}=(2^{j-1},2^{j-1}]\cap\Z,
$$
as frequency sets $I$ and construct reconstructing multiple rank\mbox{-}1 lattices using Algorithms~\ref{alg:construct_mr1l_uniform} to \ref{alg:construct_mr1l_4} with fixed parameters $c=2$, $\delta=0.5$. The parameters $T$, $d$, and $N$ in Algorithms \ref{alg:construct_mr1l_uniform} and \ref{alg:construct_mr1l_uniform_different_primes} are determined by the frequency sets
$H_n^d$.
Similar to the considerations in \cite{Kae16}, we focus on two different settings.

First we fix the dimension $d=6$ and increase the refinement $n$ up to twelve.
The results are collected in the first rows of Table \ref{tab:numtests_hypcross}.
In addition, Figure \ref{fig:numtests_hypcross_d6} shows the oversampling factors of the constructed
multiple rank\mbox{-}1 lattices and some condition numbers. The constant slopes in the plots of the oversampling factors
are caused by logarithmic dependencies on the cardinality of the frequency set $I$.
Similar to the observations in \cite{Kae16} and even better than our theoretical results, the oversampling factors
of the multiple rank\mbox{-}1 lattices that are constructed using Algorithms \ref{alg:construct_mr1l_3} and \ref{alg:construct_mr1l_4}
seem to stagnate.

\begin{table}[tb]
\centering
\setlength{\tabcolsep}{2em}
\begin{tabular}{l|ccccc}
$d$&6&50&42&18&10\\
$n$&7&2&3&4&5\\
$|I|$&16\,172&1\,376&16\,080&11\,644&8\,378\\
$\operatorname{cond}(\zb A)$&1.8e4&3.5e5&1.6e7&1.3e6&1.1e5
\end{tabular}
\caption{Condition numbers of Fourier matrices $\zb A(\mathcal{X},H_n^d)$, $\mathcal{X}$ are the related sparse grids, for selected parameters $d$ and $n$.}\label{fig:cond_sg}
\end{table}

\begin{sidewaystable}[tb]
\centering
\scalebox{0.75}{
\begin{tabular}{crrr|rrr|rrr|rrr|rrr|rrr|rrr|}
&&&&
\multicolumn{3}{c|}{Algorithm \ref{alg:construct_mr1l_uniform}}&
\multicolumn{3}{c|}{Algorithm \ref{alg:construct_mr1l_uniform_different_primes}}&
\multicolumn{3}{c|}{Algorithm \ref{alg:construct_mr1l_I}}&
\multicolumn{3}{c|}{Algorithm \ref{alg:construct_mr1l_I_distinct_primes}}&
\multicolumn{3}{c|}{Algorithm \ref{alg:construct_mr1l_3}}&
\multicolumn{3}{c|}{Algorithm \ref{alg:construct_mr1l_4}}
\\
&$d$ &$n$ &$|I|$
& $s$ & $M/|I|$ & $\cond(\zb A)$
& $s$ & $M/|I|$ & $\cond(\zb A)$
& $\tilde{s}$ & $M/|I|$ & $\cond(\zb A)$
& $\tilde{s}$ & $M/|I|$ & $\cond(\zb A)$
& $\ell$ & $M/|I|$ & $\cond(\zb A)$
& $\ell$ & $M/|I|$ & $\cond(\zb A)$\\
\cmidrule(r){1-4}
\cmidrule(lr){5-7}
\cmidrule(lr){8-10}
\cmidrule(lr){11-13}
\cmidrule(lr){14-16}
\cmidrule(lr){17-19}
\cmidrule(lr){20-22}
\multirow{11}{*}{\begin{sideways} dimension  $d=6$ \end{sideways}}
& 6 & 1 & 7 & 6 & 10.4 & 1.9 &  6 & 18.1 & 1.2 &  2 & 3.6 & 2.2 &  1 & 1.9 & 1.0 &  1 & 1.9 & 1.0 &  1 & 1.9 & 1.0\\
& 6 & 2 & 34 & 9 & 17.5 & 1.6 &  9 & 22.2 & 1.3 &  5 & 9.7 & 2.1 &  5 & 10.9 & 1.9 &  2 & 2.1 & 6.0 &  2 & 2.0 & 8.2\\
& 6 & 3 & 138 & 12 & 24.0 & 1.5 &  12 & 27.1 & 1.5 &  5 & 10.0 & 2.0 &  6 & 12.7 & 2.1 &  3 & 2.3 & 7.9 &  3 & 2.4 & 7.8\\
& 6 & 4 & 501 & 14 & 28.2 & 1.5 &  14 & 29.2 & 1.5 &  7 & 14.1 & 2.6 &  6 & 12.2 & 1.9 &  4 & 2.5 & 7.8 &  4 & 2.6 & 8.8\\
& 6 & 5 & 1\,683 & 17 & 34.0 & 1.4 &  17 & 34.8 & 1.4 &  11 & 22.0 & 1.8 &  8 & 16.2 & 1.7 &  5 & 2.7 & 7.9 &  5 & 2.7 & 9.0\\
& 6 & 6 & 5\,336 & 19 & 38.1 & 1.4 &  19 & 38.4 & 1.4 &  8 & 16.0 & 1.8 &  11 & 22.1 & 1.6 &  6 & 2.8 & 8.1 &  6 & 2.9 & 7.9\\
& 6 & 7 & 16\,172 & 21 & 42.0 & 1.4 &  21 & 42.1 & 1.4 &  9 & 18.0 & 1.7 &  10 & 20.0 & 1.8 &  7 & 2.8 & 8.1 &  7 & 2.8 & 8.5\\
& 6 & 8 & 47\,264 & 23 & 46.0 & -- &  23 & 46.1 & -- &  12 & 24.0 & -- &  14 & 28.0 & -- &  9 & 3.0 & -- &  8 & 3.0 & --\\
& 6 & 9 & 134\,048 & 25 & 50.0 & -- &  25 & 50.0 & -- &  13 & 26.0 & -- &  13 & 26.0 & -- &  10 & 3.0 & -- &  10 & 3.0 & --\\
& 6 & 10 & 370\,688 & 28 & 56.0 & -- &  28 & 56.0 & -- &  19 & 38.0 & -- &  13 & 26.0 & -- &  11 & 3.0 & -- &  11 & 3.0 & --\\
& 6 & 11 & 1\,003\,136 & 30 & 60.0 & -- &  30 & 60.0 & -- &  15 & 30.0 & -- &  17 & 34.0 & -- &  12 & 3.0 & -- &  12 & 3.1 & --\\
& 6 & 12 & 2\,664\,192 & 31 & 62.0 & -- &  31 & 62.0 & -- &  16 & 32.0 & -- &  19 & 38.0 & -- &  13 & 3.1 & -- &  13 & 3.1 & --\\
\cmidrule(r){1-4}
\cmidrule(lr){5-7}
\cmidrule(lr){8-10}
\cmidrule(lr){11-13}
\cmidrule(lr){14-16}
\cmidrule(lr){17-19}
\cmidrule(lr){20-22}
\multirow{7}{*}{\begin{sideways} expansion N=3\end{sideways}}
& 2 & 2 & 8 &6 & 12.1 & 1.5 & 6 & 18.9 & 1.2 & 2 & 4.1 & 1.7 & 2 & 4.4 & 1.7 & 1 & 2.1 & 1.0 & 1 & 2.1 & 1.0\\
& 10 & 2 & 76 &11 & 21.7 & 1.7 & 11 & 25.5 & 1.7 & 5 & 9.9 & 2.1 & 5 & 10.6 & 1.9 & 2 & 2.5 & 4.6 & 3 & 2.5 & 6.8\\
& 18 & 2 & 208 &13 & 26.1 & 1.5 & 13 & 28.1 & 1.6 & 7 & 14.1 & 1.9 & 7 & 14.6 & 1.8 & 4 & 2.8 & 7.3 & 3 & 2.8 & 6.2\\
& 26 & 2 & 404 &14 & 28.0 & 1.5 & 14 & 29.3 & 1.5 & 8 & 16.0 & 1.8 & 7 & 14.2 & 1.9 & 4 & 2.8 & 6.7 & 4 & 2.9 & 7.9\\
& 34 & 2 & 664 &15 & 30.0 & 1.5 & 15 & 31.8 & 1.5 & 9 & 18.0 & 1.7 & 7 & 14.5 & 1.8 & 5 & 3.0 & 6.7 & 5 & 2.9 & 6.9\\
& 42 & 2 & 988 &16 & 32.0 & 1.4 & 16 & 32.8 & 1.5 & 8 & 16.0 & 1.8 & 7 & 14.1 & 1.9 & 5 & 3.0 & 6.9 & 5 & 3.0 & 7.2\\
& 50 & 2 & 1\,376 &16 & 32.0 & 1.5 & 16 & 32.7 & 1.5 & 9 & 18.0 & 1.7 & 10 & 20.3 & 1.7 & 5 & 3.0 & 6.8 & 5 & 3.0 & 7.8\\
\cmidrule(r){1-4}
\cmidrule(lr){5-7}
\cmidrule(lr){8-10}
\cmidrule(lr){11-13}
\cmidrule(lr){14-16}
\cmidrule(lr){17-19}
\cmidrule(lr){20-22}
\multirow{7}{*}{\begin{sideways} expansion N=7\end{sideways}}
& 2 & 3 & 20 &8 & 16.1 & 1.4 & 8 & 21.8 & 1.4 & 3 & 6.0 & 1.7 & 4 & 9.1 & 1.7 & 1 & 2.0 & 1.0 & 1 & 2.0 & 1.0\\
& 10 & 3 & 416 &14 & 28.2 & 1.4 & 14 & 29.8 & 1.6 & 7 & 14.1 & 1.9 & 7 & 14.5 & 1.7 & 4 & 2.6 & 7.5 & 4 & 2.7 & 9.5\\
& 18 & 3 & 1\,708 &17 & 34.2 & 1.4 & 17 & 34.8 & 1.5 & 7 & 14.1 & 1.9 & 9 & 18.3 & 1.7 & 6 & 3.0 & 7.1 & 6 & 3.0 & 7.7\\
& 26 & 3 & 4\,408 &19 & 38.0 & 1.4 & 19 & 38.3 & 1.4 & 10 & 20.0 & 1.7 & 10 & 20.1 & 1.7 & 6 & 3.0 & 7.5 & 6 & 3.1 & 7.0\\
& 34 & 3 & 9\,028 &20 & 40.0 & 1.4 & 20 & 40.2 & 1.4 & 11 & 22.0 & 1.6 & 10 & 20.0 & 1.7 & 7 & 3.1 & 7.2 & 7 & 3.1 & 6.8\\
& 42 & 3 & 16\,080 &21 & 42.0 & 1.4 & 21 & 42.1 & 1.4 & 11 & 22.0 & 1.7 & 10 & 20.0 & 1.7 & 8 & 3.2 & 7.4 & 8 & 3.2 & 7.2\\
& 50 & 3 & 26\,076 &22 & 44.0 & -- & 22 & 44.1 & -- & 12 & 24.0 & -- & 12 & 24.0 & -- & 9 & 3.2 & -- & 8 & 3.2 & --\\
\cmidrule(r){1-4}
\cmidrule(lr){5-7}
\cmidrule(lr){8-10}
\cmidrule(lr){11-13}
\cmidrule(lr){14-16}
\cmidrule(lr){17-19}
\cmidrule(lr){20-22}
\multirow{7}{*}{\begin{sideways} expansion N=15\end{sideways}}
& 2 & 4 & 48 &10 & 20.0 & 1.4 & 10 & 24.1 & 1.4 & 1 & 2.0 & 1.0 & 3 & 6.2 & 1.9 & 1 & 2.0 & 1.0 & 2 & 2.2 & 5.4\\
& 10 & 4 & 1\,966 &17 & 34.0 & 1.5 & 17 & 34.7 & 1.4 & 8 & 16.0 & 1.9 & 8 & 16.2 & 1.9 & 6 & 2.8 & 7.4 & 5 & 2.9 & 6.3\\
& 18 & 4 & 11\,644 &21 & 42.0 & 1.4 & 21 & 42.2 & 1.4 & 10 & 20.0 & 1.7 & 10 & 20.0 & 1.8 & 7 & 3.1 & 7.0 & 7 & 3.1 & 7.5\\
& 26 & 4 & 39\,066 &23 & 46.0 & -- & 23 & 46.1 & -- & 12 & 24.0 & -- & 15 & 30.0 & -- & 9 & 3.2 & -- & 9 & 3.1 & --\\
& 34 & 4 & 98\,312 &25 & 50.0 & -- & 25 & 50.0 & -- & 13 & 26.0 & -- & 13 & 26.0 & -- & 10 & 3.2 & -- & 10 & 3.2 & --\\
& 42 & 4 & 207\,558 &26 & 52.0 & -- & 26 & 52.0 & -- & 14 & 28.0 & -- & 14 & 28.0 & -- & 10 & 3.2 & -- & 10 & 3.2 & --\\
& 50 & 4 & 389\,076 &28 & 56.0 & -- & 28 & 56.0 & -- & 16 & 32.0 & -- & 14 & 28.0 & -- & 11 & 3.2 & -- & 11 & 3.2 & --\\
\cmidrule(r){1-4}
\cmidrule(lr){5-7}
\cmidrule(lr){8-10}
\cmidrule(lr){11-13}
\cmidrule(lr){14-16}
\cmidrule(lr){17-19}
\cmidrule(lr){20-22}
\multirow{7}{*}{\begin{sideways} expansion N=31\end{sideways}}
& 2 & 5 & 112 &11 & 21.8 & 1.4 & 11 & 24.0 & 1.4 & 2 & 4.0 & 1.7 & 3 & 6.0 & 1.8 & 2 & 2.2 & 6.3 & 2 & 2.2 & 6.3\\
& 10 & 5 & 8\,378 &20 & 40.0 & 1.4 & 20 & 40.3 & 1.4 & 10 & 20.0 & 1.7 & 9 & 18.1 & 1.7 & 7 & 3.0 & 8.0 & 7 & 3.0 & 7.8\\
& 18 & 5 & 69\,460 &24 & 48.0 & -- & 24 & 48.1 & -- & 13 & 26.0 & -- & 13 & 26.0 & -- & 9 & 3.1 & -- & 9 & 3.1 & --\\
& 26 & 5 & 297\,662 &27 & 54.0 & -- & 27 & 54.0 & -- & 14 & 28.0 & -- & 15 & 30.0 & -- & 11 & 3.2 & -- & 11 & 3.2 & --\\
& 34 & 5 & 909\,688 &29 & 58.0 & -- & 29 & 58.0 & -- & 14 & 28.0 & -- & 15 & 30.0 & -- & 12 & 3.3 & -- & 12 & 3.2 & --\\
& 42 & 5 & 2\,257\,410 &31 & 62.0 & -- & 31 & 62.0 & -- & 16 & 32.0 & -- & 16 & 32.0 & -- & 13 & 3.3 & -- & 13 & 3.3 & --\\
& 50 & 5 & 4\,860\,636 &33 & 66.0 & -- & 33 & 66.0 & -- & 17 & 34.0 & -- & 17 & 34.0 & -- & 14 & 3.3 & -- & 14 & 3.3 & --\\
\end{tabular}
}
\caption{Numbers $s$, $\tilde{s}$, $\ell$ of joined rank\mbox{-}1 lattices
for the construction of reconstructing multiple rank\mbox{-}1 lattices for hyperbolic cross frequency sets of different dimensions $d$ and refinements $n$ using Algorithms \ref{alg:construct_mr1l_uniform} to \ref{alg:construct_mr1l_4}; oversampling factors $M/|I|$ and condition numbers $\cond(\zb A)$ of corresponding Fourier matrices, parameters $c=2$ and $\delta=0.5$ fixed.}
\label{tab:numtests_hypcross}
\end{sidewaystable}

Second, we fix the refinements $n=2,3,4,5$ and consider growing dimension $d$.
The lower rows of Table \ref{tab:numtests_hypcross} presents some
results of this numerical experiment. Figure \ref{fig:numtests_hypcross_nfix}
depicts the numerical tests in more detail.

In accordance to our theoretical results, we observe at most linearity in the oversampling factor $M/T$
with respect to $\log{T}$. Another interesting observation is that Algorithms \ref{alg:construct_mr1l_I} and
\ref{alg:construct_mr1l_I_distinct_primes} constructs multiple rank\mbox{-}1 lattices of lower cardinality than
Algorithm \ref{alg:construct_mr1l_uniform} and \ref{alg:construct_mr1l_uniform_different_primes},
which coincides with the considerations that led
to Algorithms \ref{alg:construct_mr1l_I} and \ref{alg:construct_mr1l_I_distinct_primes}, cf. the context of \eqref{eq:tilde_I}.
We observe that the application of Algorithms \ref{alg:construct_mr1l_3} and \ref{alg:construct_mr1l_4} result in the smallest oversampling factors in practice.
The oversampling factors are quite small (less than 3.3) and seem to stagnate, which means that the oversampling factors behave notably better than the theoretical results.
A similar behavior of the oversampling factors was already observed in \cite{Kae16}.

As mentioned above, we additionally computed the condition numbers of the Fourier matrices $\zb A(\Lambda(\zb z_1,M_1,\ldots,\zb z_s,M_s),I)$, where the frequency sets $I$ are of cardinality up to $20\,000$, cf. Figures \ref{fig:numtests_hypcross_d6} and \ref{fig:numtests_hypcross_n5_cond}.
We observed condition numbers that were
bounded by two, three, and twelve using the multiple rank\mbox{-}1 lattices
$\Lambda(\zb z_1,M_1,\ldots,\zb z_s,M_s)$ that are determined by Algorithms \ref{alg:construct_mr1l_uniform} and \ref{alg:construct_mr1l_uniform_different_primes}, Algorithms \ref{alg:construct_mr1l_I} and \ref{alg:construct_mr1l_I_distinct_primes}, and Algorithms \ref{alg:construct_mr1l_3} and \ref{alg:construct_mr1l_4}, respectively. In general, the tests does not allow for an reliable
interpretation. Nevertheless, we suppose that
the condition numbers do not substantially increase for growing
cardinality of the frequency set~$I$.

At this point, we would like to compare the obtained numerical results to
those for dyadic sparse grids, since these are the natural spatial discretizations for the considered dyadic
hyperbolic crosses in frequency domain, cf. \cite{Halla92}.
Due to the correlated construction of dyadic sparse grids and dyadic hyperbolic crosses,
the oversampling factor is exactly one and we focus on the condition number, which was investigated in
\cite{KaKu10}. Therein, the authors showed lower bounds on the condition numbers of the Fourier
matrices $\zb A(\mathcal{X},H_n^d)$, cf. \eqref{eq:Fourier_matrix_X}, where the spatial discretizations $\mathcal{X}$ are the related sparse grids.
These bounds increase for growing matrix dimensions, i.e., for growing refinement $n$ and growing
spatial dimension $d$, cf. \cite{KaKu10} for details. Besides the theoretical considerations,
numerically determined condition numbers were presented. Some of them are stated in Table \ref{fig:cond_sg}.
Even for matrices of less than $20\,000$ columns, one computes condition numbers up to several millions.
The numerically determined, outstanding condition numbers for multiple rank-1 lattice
spatial discretizations are substantially lower, cf. rows 9, 21, 27, 31, 37 of Table~\ref{tab:numtests_hypcross}, and consequently, preferable.

\section*{Conclusion}

The concept of sampling schemes that consists of a set of rank\mbox{-}1 lattices
allows for the reconstruction of multivariate trigonometric polynomials
with frequencies supported on arbitrary frequency sets $I\subset\Z^d$, $N_I\lesssim|I|$.
The oversampling factor, i.e., the ratio of the required number of used sampling values
to the number of frequencies, is bounded by terms that are logarithmic in 
the number of frequencies with a very high probability.

Numerical tests, cf. also \cite{Kae16}, indicate that even the condition numbers of
the corresponding Fourier matrices are bounded, which is in fact the crucial remaining
unproved characteristic.
Reasonable bounds on these condition numbers that coincide roughly
with the numerical observations might
imply excellent approximation properties and even
a highly improved Marcinkiewicz-type theorem for hyperbolic cross
trigonometric polynomials --- at least in $L_2$, cf. \cite{ByKaUlVo16} and \cite{Tem17}, respectively.

\subsection*{Acknowledgments}
The author thanks the referees for the valuable comments and suggestions and
gratefully acknowledges the funding by the Deutsche Forschungsgemeinschaft (DFG, German Research Foundation) -- 380648269.

{\scriptsize
\bibliographystyle{abbrv}

}
\end{document}